\documentclass{amsart}

\usepackage{amsmath,amssymb,yhmath, url,color}
\usepackage{amsfonts}
\usepackage{amsmath}
\usepackage{amssymb}
\usepackage{enumerate}
\usepackage{tikz}
\usepackage{mathrsfs}
\usepackage{lscape}

\newtheorem{lemma}{Lemma}[section]
\newtheorem{prop}[lemma]{Proposition}
\newtheorem{cor}[lemma]{Corollary}
\newtheorem{thm}[lemma]{Theorem}
\newtheorem{example}[lemma]{Example}
\newtheorem{thm?}[lemma]{Theorem?}

\newtheorem{ques}[lemma]{Question}
\newtheorem{conj}[lemma]{Conjecture}

\newtheorem{fact}{Fact}

\newtheorem{remark}[lemma]{Remark}

\begin{document}
\title{Torsion Points on CM Elliptic Curves Over Real Number Fields}
\author{Abbey Bourdon}
\address{University of Georgia}
\email{abourdon@uga.edu}
\author{Pete L. Clark}
\address{University of Georgia}
\email{plclark@gmail.com}
\author{James Stankewicz}
\address{University of Bristol/Heilbronn Institute for Mathematical Research}
\email{j.stankewicz@bristol.ac.uk}

%\date{\today}

\newcommand{\etalchar}[1]{$^{#1}$}
\newcommand{\F}{\mathbb{F}}
\newcommand{\et}{\textrm{\'et}}
\newcommand{\ra}{\ensuremath{\rightarrow}}
\newcommand{\FF}{\F}
\newcommand{\Z}{\mathbb{Z}}
\newcommand{\N}{\mathbb{N}}
\newcommand{\ch}{}
\newcommand{\R}{\mathbb{R}}
\newcommand{\PP}{\mathbb{P}}
\newcommand{\pp}{\mathfrak{p}}
\newcommand{\C}{\mathbb{C}}
\newcommand{\Q}{\mathbb{Q}}
\newcommand{\tpqr}{\widetilde{\triangle(p,q,r)}}
\newcommand{\ab}{\operatorname{ab}}
\newcommand{\Aut}{\operatorname{Aut}}
\newcommand{\gk}{\mathfrak{g}_K}
\newcommand{\gq}{\mathfrak{g}_{\Q}}
\newcommand{\OQ}{\overline{\Q}}
\newcommand{\Out}{\operatorname{Out}}
\newcommand{\End}{\operatorname{End}}
\newcommand{\Gon}{\operatorname{Gon}}
\newcommand{\Gal}{\operatorname{Gal}}
\newcommand{\CT}{(\mathcal{C},\mathcal{T})}
\newcommand{\ttop}{\operatorname{top}}
\newcommand{\lcm}{\operatorname{lcm}}
\newcommand{\Div}{\operatorname{Div}}
\newcommand{\OO}{\mathcal{O}}
\newcommand{\rank}{\operatorname{rank}}
\newcommand{\tors}{\operatorname{tors}}
\newcommand{\IM}{\operatorname{IM}}
\newcommand{\CM}{\operatorname{CM}}
\newcommand{\Frac}{\operatorname{Frac}}
\newcommand{\Pic}{\operatorname{Pic}}
\newcommand{\coker}{\operatorname{coker}}
\newcommand{\Cl}{\operatorname{Cl}}
\newcommand{\loc}{\operatorname{loc}}
\newcommand{\GL}{\operatorname{GL}}
\newcommand{\PSL}{\operatorname{PSL}}
\newcommand{\Frob}{\operatorname{Frob}}
\newcommand{\Hom}{\operatorname{Hom}}
\newcommand{\Coker}{\operatorname{\coker}}
\newcommand{\Ker}{\ker}
\renewcommand{\ggg}{\mathfrak{g}}
\newcommand{\sep}{\operatorname{sep}}
\newcommand{\new}{\operatorname{new}}
\newcommand{\Ok}{\mathcal{O}_K}
\newcommand{\ord}{\operatorname{ord}}
\newcommand{\mm}{\mathfrak{m}}
\newcommand{\Ohell}{T_{\ell}(\OO)}

\begin{abstract}
We study torsion subgroups of elliptic curves with complex multiplication (CM) defined over number fields which admit a real embedding.  We give a complete classification of the groups which arise up to isomorphism as the torsion subgroup of a CM elliptic curve defined over a number field of odd degree: there are infinitely many. However, if we fix an odd integer $d$ and consider number fields of degree $dp$ as $p$ ranges over all prime numbers, all but finitely many torsion subgroups that appear for CM elliptic curves actually occur in a degree dividing $d$.  This implies an absolute bound on the size of torsion subgroups of CM elliptic curves defined over number fields of degree $dp$. In the case where $d=1$, there are six ``Olson groups'' which arise as torsion subgroups of CM elliptic curves over $\Q$, and there are precisely $17$ ``non-Olson'' CM elliptic curves defined over a number field of (variable) prime degree.

%However, if we fix an odd integer $d$, there is an absolute bound on the size of the group that can appear as the torsion subgroup of a CM elliptic curve defined over a number field of degree $dp$, where $p$ ranges over all the prime numbers.

%  If there are infinitely many Sophie Germain primes in a certain arithmetic %progression then there are infinitely many primes dividing the order of some CM %elliptic curve defined over a number field of degree twice a prime.
% Also, for every $N \in \Z^+$, the modular curve $Y_1(N)$ has infinitely many points %of odd degree.
\end{abstract}

\maketitle

\tableofcontents
%\addtocontents{toc}{\protect\enlargethispage{3\baselineskip}}

\noindent
We denote by $\mathcal{P}$ the set of all prime numbers.  For $n \in \Z^+$ let $\zeta_n = e^{\frac{2 \pi i}{n}} \in \C$, and put
$\Q(\zeta_n)^+ = \Q(\zeta_n + \zeta_n^{-1})$.  For a field $F$, let $\overline{F}$ be an algebraic closure, let $F^{\sep}$ be the maximal separable subextension of $\overline{F}/F$, and let $\ggg_F = \Aut(F^{\sep}/F) = \Aut(\overline{F}/F)$ be the absolute Galois group of $F$.
  A \textbf{real number field} is a number field which admits an embedding into $\R$.  Thus every odd degree number field is real. If $E$ is an elliptic curve with complex multiplication by an order $\OO$ of discriminant $\Delta$ in an imaginary quadratic field $K$, then $F_{\Delta}=\Q(j(E))$, $K_{\Delta}=K(j(E))$, and $h(\Delta)=\#\Pic(\OO)=[K_{\Delta}:K]$. For an imaginary quadratic field $K$, let $w(K) = \# \mathcal{O}_K^{\times}$.

\section{Introduction}

\subsection{Motivation}
\textbf{} \\ \\
\noindent
This paper continues an exploration of torsion points on elliptic curves with complex multiplication (CM) initiated by the last two authors in collaboration with
B. Cook, P. Corn and A. Rice \cite{TORS1}, \cite{TORS2}.
%We will use and extend results from \cite{TORS1} to prove a certain ``absolute %boundedness theorem'' which we would not have guessed from general principles but %was suggested by the calculations of \cite{TORS2}.
\\ \\
The subject of torsion points on CM elliptic curves begins with the following result.

\begin{thm}(Olson \cite{Olson74})
\label{OLSON1}
Let $E_{/\Q}$ be a CM elliptic curve.  Then $E(\Q)[\tors]$ is isomorphic to
one of: the trivial group $\{ \bullet\}$, $\Z/2\Z$, $\Z/3\Z$, $\Z/4\Z$, $\Z/6\Z$ or $\Z/2\Z \times \Z/2\Z$.  Conversely, each such group occurs for at least
one CM elliptic curve $E_{/\Q}$.
\end{thm}
\noindent
We say a finite commutative group $G$ is an \textbf{Olson group} if
it is isomorphic to one of the six groups given in the conclusion of Theorem
\ref{OLSON1}.  A CM elliptic curve $E_{/F}$ is \textbf{Olson} if $E(F)[\tors]$ is an Olson group.
\\ \\
%{\color{red} Can someone reduce the spacing between the two displayed equations %below?}
%$$\{1,2,3,4,6\} =  \{ \exp G \mid \text{$G$ is an Olson group}\} = \{ n \in \Z^+ %\mid \Q(\zeta_n)^+ = \Q \}.$$
It follows from Theorem \ref{IMPRIMITIVELEMMA}a) below that
for every Olson group $G$ and every $d \geq 2$, there are infinitely many degree $d$ number fields $F$ for which there is a CM elliptic curve $E_{/F}$ with $E(F)[\tors] \cong G$.  Similarly, whenever $d_1 \mid d_2$, the list of torsion subgroups of CM elliptic curves in degree $d_2$ will contain the corresponding list in degree $d_1$.   It is more penetrating to ask which \emph{new} groups arise in degree $d$: for $d \in \Z^+$, let $\mathcal{T}_{\CM}(d)$ be the set of isomorphism classes of torsion subgroups of CM elliptic curves defined over
number fields of degree $d$, and for $d \geq 2$ we put
\[\mathcal{T}^{\new}_{\CM}(d) = \mathcal{T}_{\CM}(d) \setminus \bigcup_{d' \mid d, \ d' \neq d} \mathcal{T}_{\CM}(d'). \]
From \cite[$\S 4$]{TORS2} we compile the following table.

\[\begin{tabular}{c|ccccccccccccc}
d & 2&3&4&5&6&7&8&9&10&11&12&13  \\ \hline
$\# \mathcal{T}^{\new}_{\CM}(d)$  & 5& 2&9&1&7&0&14&3&4&0&13& 0
\end{tabular} \]
\[ \text{\sc table 1} \]

%\begin{center}
%\begin{tabular}{c|c}
%d & $\# \mathcal{T}^{\new}_{\CM}(d)$ \\ \hline
% 2 &  5 \\ [.5ex]
% 3 &  2 \\  [.5ex]
% 4 &  9\\   [.5ex]
% 5 &  1 \\   [.5ex]
% 6 &  7 \\   [.5ex]
% 7 &  0 \\     [.5ex]
% 8 &  14 \\  [.5ex]
% 9 &  3 \\   [.5ex]
%10 &  4 \\   [.5ex]
%11 &  0\\   [.5ex]
%12 &  13 \\  [.5ex]
%13 &  0
%\end{tabular}
%\end{center}

\begin{remark}
\label{KEYREMARK}
a) The size of $\mathcal{T}^{\new}_{\CM}(d)$ is strongly influenced by the $2$-adic valuation $v_2(d)$: for all $2 \leq d_1,d_2 \leq 13$, $v_2(d_1) < v_2(d_2) \implies \# \mathcal{T}^{\new}_{\CM}(d_1) <
\# \mathcal{T}^{\new}_{\CM}(d_2)$.
\\
b) There is very little new torsion when $d$ is odd.
\\
c) When we restrict to prime values of $d$, the sequence of values is
$5,2,1,0,0,0$.
\end{remark}
\noindent
Remark \ref{KEYREMARK}c) was made to us by M. Sch\"utt.  He also asked the following question.

\begin{ques}(Sch\"utt)
\label{SCHUETT}
 Is $\# \mathcal{T}^{\new}_{\CM}(p) = 0$ for all sufficiently large primes $p$?
\end{ques}

\subsection{The Main Results of the Paper}
\textbf{} \\ \\
\noindent
Sch\"utt's question is equivalent to asking whether there is an \emph{absolute bound} on the size of the torsion subgroup of all CM elliptic curves defined over all number fields of prime degree.  The first main result of this paper is an affirmative answer.
\\ \\
For $b,c \in F$ we define the \textbf{Kubert-Tate curve}
\[E(b,c): y^2+(1-c)xy-by=x^3-bx^2. \]
%When nonsingular, $E(b,c)_{/F}$ is an elliptic curve with an  $F$-rational point %$(0,0)$ of order at least $4$.  For every elliptic curve $E_{/F}$ and $P \in %E(F)$ with $\# \langle P \rangle \geq 4$ there are $b,c \in F$ and an %$F$-isomorphism $(E,P) \stackrel{\sim}{\ra} (E(b,c),(0,0))$ \cite{Kubert76}.  %Thus if $E_{/F}$ is non-Olson, either $E$ is isomorphic to some $E(b,c)$ or %$E(F)[\tors] \cong \Z/3\Z \times \Z/3\Z$.  To handle the latter case we introduce
For $\lambda \in F$ we define the \textbf{Hesse curve}
\[ E_{\lambda}: X^3 + Y^3 + Z^3 + \lambda XYZ = 0. \]
%When nonsingular, $E_{\lambda}$ is an elliptic curve with $E[3] \cong_F \Z/3\Z %\times \mu_3$, and thus $\Z/3\Z \times \Z/3\Z \subset E(F) \iff \Q(\sqrt{-3}) %\subset F$.  Conversely for every elliptic curve $E_{/F}$ with $E[3] \cong_F %\Z/3\Z \times \mu_3$, there is $\lambda \in F$ with $E \cong E_{\lambda}$
%\cite[Table 1]{Kubert76}.

\begin{thm}(Prime Degree Theorem)
\label{MAINTHM}
Let $F$ be a prime degree number field, and let $E_{/F}$ be a
non-Olson elliptic curve with CM by a quadratic order of discriminant $\Delta$.  Then $F$ is isomorphic to one of the fields listed below, and over that field $E$ is isomorphic to exactly one listed curve.
{\footnotesize
\begin{center}
    \begin{tabular}{ c|c|c|c}
    Number Field $F$ & Elliptic Curve & $\Delta$ & $E(F)[\tors]$\\  \hline
    $\Q(\sqrt{-3})$ & $E_0$ & $-3$ & $\Z/3\Z \oplus \Z/3\Z$ \\   [.5 ex]
    $\Q(i)$ & $E(-\frac{1}{8},0)$ & $-4$ & $\Z/2\Z \oplus \Z/4\Z$ \\   [.5 ex]
    $\Q(\sqrt{2})$ & $E(1+\frac{3}{4}\sqrt{2},0)$ & $-4$ & $\Z/2\Z \oplus \Z/4\Z$ \\   [.5 ex]
    $\Q(\sqrt{2})$ & $E(-\frac{1}{32},0)$ & $-16$ & $\Z/2\Z \oplus \Z/4\Z$ \\   [.5 ex]
    $\Q(\sqrt{2})$ & $E(\frac{1+\sqrt{2}}{8},0)$ & $-8$ & $\Z/2\Z \oplus \Z/4\Z$ \\   [.5 ex]
    $\Q(\sqrt{-7})$ & $E(\frac{-31+3\sqrt{-7}}{512},0)$ & $-7$ & $\Z/2\Z \oplus \Z/4\Z$ \\   [.5 ex]
    $\Q(\sqrt{-7})$ & $E(\frac{-1+3\sqrt{-7}}{32},0)$ & $-7$& $\Z/2\Z \oplus \Z/4\Z$ \\   [.5 ex]
    $\Q(\sqrt{-3})$ & $E(-\frac{2}{9},-\frac{1}{3})$ & $-3$ & $\Z/2\Z \oplus \Z/6\Z$ \\   [.5 ex]
    $\Q(\sqrt{3})$ & $E(\frac{1-\sqrt{3}}{9},\frac{-2+\sqrt{3}}{3})$ & $-12$ & $\Z/2\Z \oplus \Z/6\Z$ \\   [.5 ex]
    $\Q(\sqrt{3})$ & $E(\frac{4}{9},\frac{1}{3})$ & $-12$ & $\Z/2\Z \oplus \Z/6\Z$ \\   [.5 ex]
    $\Q(\sqrt{-3})$ &$ E(\frac{-1+\sqrt{-3}}{2},-1)$ & $-3$ & $\Z/7\Z$ \\   [.5 ex]
    $\Q(i)$ & $E(i,i)$ & $-4$ & $\Z/10\Z$ \\   [.5 ex]
    $\dfrac{\Q[b]}{(b^3-15b^2-9b-1)}$ & $E(\frac{1}{4}b^2+\frac{5}{2}b+\frac{1}{4},b)$ & $-3$ & $\Z/9\Z $ \\   [.5 ex]
    $\dfrac{\Q[b]}{(b^3+105b^2-33b-1)}$ & $E(-\frac{17}{76}b^2+\frac{25}{19}b+\frac{1}{76},b)$ & $-27$ & $\Z/9\Z $ \\   [.5 ex]
    $\dfrac{\Q[b]}{(b^3-4b^2+3b+1)}$ & $E(-2b^2+4b+1,b)$ & $-7$ & $\Z/14\Z$ \\   [.5 ex]
    $\dfrac{\Q[b]}{(b^3-186b^2+3b+1)}$ & $E(\frac{2}{27}b^2+\frac{10}{27}b-\frac{1}{27},b)$ & $-28$ & $\Z/14\Z$ \\   [.5 ex]
    $\dfrac{\Q[b]}{(b^5-9b^4+6b^3+42b^2-7b-1)}$ & $E(-\frac{1}{16}b^4+\frac{1}{4}b^3+\frac{5}{8}b^2+\frac{1}{4}b-\frac{1}{16},b)$ & $-11$ & $\Z/11\Z $ \\   [.5 ex]

     \end{tabular}
\end{center}
}

%Here,
%\begin{center}
%\begin{eqnarray*}
%f_1 & = & b^3-15b^2-9b-1 \\
%f_2 & = & b^3+105b^2-33b-1\\
%f_3 & = & b^3-4b^2+3b+1\\
%f_4 & = & b^3-186b^2+3b+1 \\
%f_5 & = & b^5-9b^4+6b^3+42b^2-7b-1.
%\end{eqnarray*}
%\end{center}
\end{thm}
\noindent
Although the Prime Degree Theorem was our initial goal, it has become more of a gadfly, as we will now explain.  The literature of the field contains several results of the form ``If an $\OO$-CM elliptic curve over a number field $F$ has an $F$-rational point of order $N$, then $[F:\Q]$ is bounded below by some function of $N$.''  The prototype was given by Silverberg and later by Prasad-Yogananda \cite{Silverberg88} and \cite{PY01}.  These \textbf{SPY-bounds} were refined by Clark-Cook-Stankewicz \cite[Theorem 3]{TORS1}.  The proof of the Prime Degree Theorem requires further refinements, to which we found our way by observing a \textbf{real cyclotomy} phenomenon in the tables of \cite{TORS2}: for every CM elliptic curve $E_{/F}$ in the tables containing an $F$-rational point of order $N \geq 3$, $F$ contains either the CM field $K$
or $\Q(\zeta_N)^+$.  In particular, if $F$ has odd degree then $F \supset \Q(\zeta_N)^+$ and thus $\frac{\varphi(N)}{2} \mid [F:\Q]$.
\\ \indent
We present here two versions of real cyclotomy.  \textbf{Real Cyclotomy I} deals with an $\OO(\Delta)$-CM elliptic curve over a number field $F$ not containing $\Q(\sqrt{\Delta})$ under the additional hypothesis that $\gcd(N,\Delta) = 1$.  The conclusion is $(\Z/N\Z)^2 \hookrightarrow E(FK)$.  It follows that $\zeta_N \in FK$, so $\varphi(N) \mid 2[F:\Q]$.  When $F$ is real (and in some other cases), $\zeta_N \in FK$ implies $\Q(\zeta_N)^+ \subset F$, confirming real cyclotomy in this case.  In \textbf{Real Cyclotomy II} we assume moreover that $F$ is real and drop the condition $\gcd(N,\Delta) = 1$.  The conclusion is that $\zeta_N \in FK$, hence $\Q(\zeta_N)^+ \subset F$.
\\ \\
After this paper was submitted we learned of important work of N. Aoki \cite{Aoki95}, \cite{Aoki06} showing that if $E_{/F}$ is a $K$-CM elliptic curve defined over a number field $F \not \supset K$ and $E(FK)$ has a point of order $N$, then $\zeta_N \in FK$. Thus Aoki's Theorem implies Real Cyclotomy II (though not Real Cyclotomy I) and confirms the observed real cyclotomy phenomenon.  Aoki's work in fact applies to a large class of CM abelian varieties $A$ defined over a number field $F$ -- e.g. when $A$ is geometrically simple and $F$ is real \cite[Thm. 6.2]{Aoki06}.  His proof uses class field theory and Shimura-Taniyama's adelic \textbf{main theorem of complex multiplication}. \\ \indent In contrast, the main ingredients of our proof of Real Cyclotomy II are the ideal theory of imaginary quadratic orders (essentially due to Gauss), especially the complete determination of $\Pic \OO(\Delta)[2]$ given by \textbf{classical genus theory}, and the relation between genus theory and the uniformization of CM elliptic curves by \textbf{real lattices} $\Lambda = \overline{\Lambda} \subset \C$.  Thus our proof of Real Cyclotomy II is quite different from Aoki's proof of his theorem, and it proceeds by establishing some pleasant results which have a very classical feel but which in many cases we have nevertheless not been able to find in the literature.  For these reasons we have chosen to include our proof.  In fact, the idea to exploit classical genus theory leads to an improvement of Aoki's bound by a factor of $\frac{\# \Pic \OO(\Delta)}{\# (\Pic \OO(\Delta)[2])}$, which we have incorporated into our statement of Real Cyclotomy II: see (\ref{SWEETCYCEQ1}).  This extra factor simplifies the proof of the Prime Degree Theorem and is used elsewhere in the paper.
\\ \\
The hypothesis that $E_{/F}$ is an elliptic curve defined over a number field not containing the CM field is most naturally achieved by requiring that the degree $[F:\Q]$ is odd.  This forces $F$ to be real.  When $[F:\Q]$ is odd, Aoki's bound $\varphi(N) \mid 2[F:\Q]$ implies -- just using the elementary properties of Euler's $\varphi$ function -- strong restrictions on the exponent of $E(F)[\tors]$, and since $F$ does not contain the CM field the order of $E(F)[\tors]$ divides twice its exponent.  Thus the torsion subgroup $E(F)[\tors]$ is highly constrained.  This was observed by Aoki \cite[Cor. 9.4]{Aoki95}; we present Aoki's constraints in slightly refined form in Theorem \ref{JIMSLEMMA}.
\\ \indent
Our next main result shows that all the groups not eliminated above do in fact arise as torsion subgroups of CM elliptic curves in some odd degree.

\begin{thm}(Odd Degree Theorem)
\label{ODDDEGREETHMINTRO}
Let $F$ be a number field of odd degree, let $E_{/F}$ be a $K$-CM elliptic curve, and let $T = E(F)[\tors]$.  Then: \\
 a) One of the following occurs:
\begin{enumerate}
\item $T$ is isomorphic to the trivial group $\{ \bullet\}, \, \Z/2\Z, \, \Z/4\Z,$ or $\Z/2\Z \times \Z/2\Z$;
\item $T \cong \Z/\ell^n \Z$ for a prime $\ell \equiv 3 \pmod{8}$ and $n \in \mathbb{Z}^+$ and
$K = \Q(\sqrt{-\ell})$;
\item $T \cong \Z/2\ell^n \Z$ for a prime $\ell \equiv 3 \pmod{4}$ and $n \in \mathbb{Z}^+$ and
$K = \Q(\sqrt{-\ell})$.
\end{enumerate}
b) If $E(F)[\tors] \cong \Z/2\Z \oplus \Z/2\Z$, then $\End E$ has discriminant $\Delta = -4$.  \\
c) If $E(F)[\tors] \cong \Z/4\Z$, then $\End E$ has discriminant $\Delta \in \{ -4,-16\}$. \\
d) Each of the groups listed in part a) arises up to isomorphism as the torsion subgroup $E(F)$ of a CM elliptic curve $E$ defined over an odd degree number field $F$.
\end{thm}
%\begin{thm}(Odd Degree Theorem)
%\label{ODDDEGREETHMINTRO}
%Let $F$ be a number field of odd degree, let $E_{/F}$ be a CM %elliptic curve, and let $T = E(F)[\tors]$.  Then $T$ is %isomorphic to one of the following groups:
%\begin{enumerate}
%\item the trivial group $\{ \bullet\}, \, \Z/2\Z, \, \Z/4\Z,$ or %$\Z/2\Z \times \Z/2\Z$;
%\item the group $\Z/\ell^n \Z$ for a prime $\ell \equiv 3 %\pmod{8}$ and $n \in \mathbb{Z}^+$;
%\item the group $\Z/2\ell^n \Z$ for a prime $\ell \equiv 3 %\pmod{4}$ and $n \in \mathbb{Z}^+$.
%\end{enumerate}
% Conversely, each of the above groups arises up to isomorphism %as the torsion subgroup $E(F)$ of a CM elliptic curve $E$ 5defined over an odd degree number field $F$.
%\end{thm}
\noindent
With the Odd Degree Theorem and Real Cyclotomy II in hand, we can readily prove the Prime Degree Theorem.  In fact we are able to view
 the Prime Degree Theorem as an explicit version of the $k= 1$ case of the following result.

\begin{thm}(Shifted Prime Degree Theorem)
\label{SPDTINTRO} \\
%\label{THM6.1} \\
a) There is a function $C: \Z^+ \ra \Z^+$ such that: for all
$k \in \Z^+$, all primes $p$, all number fields
$F/\Q$ of degree $(2k-1)p$, and all CM elliptic curves
$E_{/F}$, we have
\[ \# E(F)[\tors] \leq C(k). \]
b) There is a function $P: \Z^+ \ra \Z^+$ such that:
for all $k \in \Z^+$, all primes $p \geq P(k)$ all number fields $F$ of degree $(2k-1)p$, and all CM elliptic curves $E_{/F}$, there is a subfield $F_0 \subset F$ of degree
dividing $(2k-1)$ and an $F_0$-rational model $(E_0)_{/F_0}$ of $E_{/F}$ such that $E_0(F_0)[\tors] = E(F)[\tors]$.
\end{thm}
\noindent
We go on to analyze possible CM torsion in degree $p^2$ for a prime number $p$, in degree $2p$ for a prime number $p$, and in degree $p_1 p_2$ for distinct odd primes $p_1$, $p_2$.  In the first case we get a complete, finite classification.  In the second case we show that if there are infinitely many Sophie Germain primes in certain congruence classes (a consequence of Schinzel's Hypothesis H) then infinitely many groups arise.  In the third case we show that the infinitude of the set of groups which arise is equivalent to the infinitude of a certain set of Sophie Germain primes.
\\ \\
Our overarching moral is that by studying torsion in the CM case we can get precise classification results, well beyond what one could reasonably hope to do in the non-CM case in the near future.  Of course, whenever one gets such a result in the CM case one wonders whether it could be true for non-CM curves as well.  We end the paper with a negative result of this form (Theorem \ref{INDEXTHM}).

 %as we will see, if $E_{/F}$ is a CM elliptic curve defined over %a number field $F$ with an $F$-rational point of prime order %$\ell > 2$,  then $\ell-1 \mid 12[F:\Q]$.  However, in the %non-CM case, for no $N \in \S^+$ does the existence of an %$F$-rational point of order $N$ on an elliptic curve $E_{/F}$ %force the existence of any divisibility condition on $[F:

%If we further assume that $F/\Q$ has odd degree $d$ and Galois %group $S_d$, then every CM elliptic curve $E_{/F}$ is Olson ($\S %6.3$). By contrast,
%we show in $\S 6.2$ that any positive integer arises as the %order of a point of some (not necessarily CM) elliptic curve %defined over some odd degree number field.
%\\ \\
%Even if we restrict our attention to number fields of degree %twice a prime we see fundamentally different behavior from %Theorem \ref{MAINTHM}.

%\begin{thm}
%\label{SCHINZELTHM}
%Assume Schinzel's Hypothesis H.  As $F$ ranges over all number %fields of degree twice a prime and $E$ ranges over all CM %elliptic curves over $F$, the set of primes which divide the %order of some torsion subgroup $E(F)[\tors]$ is infinite.  In %particular:
%\[ \limsup_{p \in \mathcal{P}} \# \mathcal{T}^{\new}_{\CM}(2p) %\geq 1. \]
%\end{thm}
%\noindent
%We prove this result in a more general form in $\S$6.1.
%\\ \\

\subsection{The Contents of the Paper}
\textbf{} \\ \\ \noindent
We will now give a more linear guide to the contents of the paper.
\\ \\
In $\S$2 we prove a result (Theorem \ref{IMPRIMITIVELEMMA}) on the preservation of torsion subgroups under extension from one number field to another.  This was essentially known for elliptic curves; we show it here for abelian varieties.
\\ \\
$\S$3 recalls basic results about imaginary quadratic orders, CM elliptic curves and Cartan subgroups.  It also develops the theory of uniformization of elliptic curves by real lattices and, in the CM case, relates it to classical genus theory.  These results are related to work of S. Kwon \cite{Kwon99a} but differ in their perspective and in many of the details.  We also generalize a very classical result on ray class containment from $\OO_K$-CM elliptic curves to $\OO$-CM elliptic curves (Theorem \ref{SCHERTZTHM}).
\\ \\
$\S4$ is devoted to giving restrictions on the torsion subgroup of various kinds.  In $\S$4.1 we treat points of order $2$.  In $\S$4.2 we extend work of \cite{TORS1} -- itself a refinement of bounds of Silverberg and Prasad-Yogananda -- to the ramified case.  In the next three sections we state and prove Real Cyclotomy I, state Aoki's Theorem, and state and prove Real Cycltomy II.  In $\S$4.6 we discuss a theorem of Parish which implies in particular that if $F = \Q(j)$ then every CM elliptic curve is Olson.  Parish's Theorem is used only in order to rederive the classification of CM torsion in degree $2$ (already done in \cite{Clark04} and \cite{TORS2}) but reproduced here for completeness.  However this result has a place in any compilation of bounds on torsion in the CM case, and it turns out that there are some interesting relations with real cyclotomy.  $\S$4.7 gives an application of Real Cyclotomy I and includes an example that explores the boundaries between Real Cyclotomy I and Real Cyclotomy II.
\\ \\
$\S$5 gives a proof of the Odd Degree Theorem.  Aoki's contribution is revisited and slightly refined in $\S$5.1.  After a brief treatment of primes equivalent to 7 mod 8 in $\S$5.2, the remainder of the proof involves the construction of all nonexcluded torsion subgroups in odd degree.  This is done via some considerations on twisting which will be of future use (to us, at least), so $\S$5.3 contains some basic constructions at the level of generality of abelian varieties over number fields, while $\S$5.4 contains their application to construct torsion on CM elliptic curves in odd degree.  In $\S$5.5 we show that restrictions on the Galois closure of an odd degree number field $F$ force CM elliptic curves $E_{/F}$ to be Olson (Theorem \ref{100PERCENTJIM}).
\\ \\
$\S$6 begins with the Shifted Prime Degree Theorem (Theorem \ref{SPDT}).  In the rest of $\S$6 this result is applied to study torsion in prime degrees (Theorem \ref{MAINTHM}), prime squared degrees (Theorem \ref{PSDT}), fields of degree twice a prime (Theorem \ref{5.2}), and fields of degree a product of
distinct odd primes (Theorems \ref{TWOODDPRIMESTHM1} and \ref{TWOODDPRIMESTHM2}).
\\ \\
In $\S$7 we prove a degree nondivisibility theorem for elliptic curves over number fields \emph{without} complex multiplication
(Theorem \ref{INDEXTHM}).

\subsection{Acknowledgments}
\textbf{} \\ \\ \noindent
 Thanks to Jordan Ellenberg, Damien Robert, Will Sawin, and Matthias Sch\"utt for useful discussions.  We thank Paul Pollack for many fruitful discussions, yielding in particular the main idea of the proof of Theorem \ref{PSDT}, and also for providing us with the heuristic of $\S$6.5.  We are indebted to Alice Silverberg for pointers to the literature.
 \\ \\
 A.B. was supported in part by NSF grant DMS-1344994 (RTG in Algebra, Algebraic Geometry, and Number Theory, at the University of Georgia). J.S. was supported by the Villum Fonden through the network for Experimental Mathematics
in Number Theory, Operator Algebras, and Topology.

\section{Preservation of Torsion Subgroups Under Base Extension}

\begin{thm}
\label{IMPRIMITIVELEMMA}
Let $A_{/F}$ be an abelian variety over a number field, and let $d \geq 2$. \\
a) There are infinitely many $L/F$ such that $[L:F] = d$ and $A(L)[\tors] = A(F)[\tors]$. \\
b) If $d$ is prime, then for all but finitely many $L/F$ with $[L:F] = d$, we have
$A(L)[\tors] = A(F)[\tors]$.  \\
c) For all but at most finitely many quadratic twists $A^t$ of $A_{/F}$ we have
$A^t(F)[\tors] = A^t(F)[2] = A(F)[2]$.
\end{thm}
\begin{proof} a) By work of Masser \cite[Corollary 2]{Masser87}, there is $N \in \Z^+$ depending only on $[L:\Q]$ such that $A(L)[\tors] = A(L)[N]$.
Let $M = F(A[N])$.  Then
$A(L)[\tors] \supsetneq A(F)[\tors]$ implies $A(L)[N] \supsetneq A(F)[N]$ and thus $M \cap L \supsetneq F$.
For each $d \geq 2$ there are infinitely many degree $d$ $L/F$ with $M \cap L = F$:
let $v$ be a finite place of $F$ which is unramified in $M$, and choose $L/F$ to be totally
ramified at $v$.  This gives one extension $L_1/F$; replacing $M$ with $L_1M$ gives another extension $L_2/F$; and so forth.  \\
b) If $d = [L:F]$ is prime, then $M \cap L \supsetneq F$ implies $L \subset M$. \\
c) For $t \in F^{\times}/F^{\times 2}$, we denote by $A^t_{/F}$ the quadratic
twist by $t$ and the involution $[-1]$ on $A$.  We have monomorphisms
$A(F) \hookrightarrow A(F(\sqrt{t}))$, $A^t(F) \hookrightarrow A(F(\sqrt{t}))$,
and
\[ A(F) \cap A^t(F) = A(F)[2] = A^t(F)[2]. \]
By part b), for all
but finitely many $t$ we have \\ $A(F)[\tors] = A(F(\sqrt{t}))[\tors]$ and thus $A^t(F)[\tors] = A^t(F)[2] = A(F)[2]$.
\end{proof}

\begin{remark}
In 2001, Qiu and Zhang used Merel's theorem to prove Theorem \ref{IMPRIMITIVELEMMA}b) when $A$ is an elliptic curve \cite[Theorem 1]{Qiu-Zhang01}. When $A$ is an elliptic curve and $F = \Q$, Theorem \ref{IMPRIMITIVELEMMA}c) was proved by Gouv\^ea and Mazur \cite[Proposition 1]{Gouvea-Mazur91}.  This extends to a number field $F$ unless $A$ has complex multiplication by an imaginary quadratic field $K \subset F$.  Results of Silverberg \cite{Silverberg88} handle the case of all abelian varieties with complex multiplication. Alternately, Mazur and Rubin use Merel's theorem to establish Theorem \ref{IMPRIMITIVELEMMA}c) for elliptic curves: however, as in their application $A(F)$ has no points of order $2$, they record the result (only) in the form that all but finitely many quadratic twists of $A_{/F}$ have no odd order torsion \cite[Lemma 5.5]{Mazur-Rubin10}.  The full statement of Theorem \ref{IMPRIMITIVELEMMA}c) for elliptic curves first appears in a recent work of F. Najman \cite[Theorem 12]{Najman13}.
\end{remark}

%Therefore the quantity $\#\mathcal{T}_{\mathrm{CM}}^{new}(d)$ is well-defined. We %now show that for $d= p$, a large prime, it is always zero.

\section{$\R$-Structures, Complex Conjugation and Cartan Subgroups}

\subsection{Orders and Ideals in Imaginary Quadratic Fields}
\textbf{} \\ \\ \noindent
Let $K$ be a number field.  A \textbf{lattice} in $K$
is a $\Z$-module $\Lambda \subset K$ obtained as the $\Z$-span of a $\Q$-basis for
$K$.  An \textbf{order} $\OO$ in $K$ is a lattice which is also a subring.  The ring of integers $\OO_K$ is an order of $K$; conversely, since every element of an order $\OO$ is integral over $\Z$, $\OO \subset \OO_K$ with finite index.  For any lattice $\Lambda$,
\[ \OO(\Lambda) = \{x \in K \mid x \Lambda \subset \Lambda\} \]
is an order of $K$, and $\Lambda$ is a fractional $\OO(\Lambda)$-ideal of $K$.
For all $\alpha \in K^{\times}$ we have $\OO(\alpha \Lambda) = \OO(\Lambda)$.
For any order $\OO$, a fractional $\OO$-ideal $\Lambda$ is \textbf{proper}
if $\OO = \OO(\Lambda)$.  A fractional $\OO_K$-ideal is necessarily proper, whereas for any nonmaximal order $\OO$, $[\OO_K:\OO]\OO_K$ is an $\OO$-ideal which is not proper.

% A \textbf{real number field} is a number field
%which admits an embedding into $\R$.
%  A
%fractional $\OO$-ideal $\Lambda$ which is projective as an $\OO$-module is %proper.  The converse is not true over an arbitrary number field $K$!

\begin{lemma}
\label{PROPERPROJECTIVE}
Let $\OO$ be an order in a \emph{quadratic} field $K$, and let $\Lambda$ be
a fractional $\OO$-ideal.  The following are equivalent: \\
(i) $\Lambda$ is a projective $\OO$-module.  \\
(ii) For every prime $p$, $\Lambda \otimes_{\Z} \Z_{(p)}$ is a principal fractional $\OO \otimes_{\Z} \Z_{(p)}$-ideal. \\
(iii) $\Lambda$ is a proper $\OO$-ideal.
\end{lemma}
\begin{proof}
To prove (i) $\Leftrightarrow$ (ii) is an exercise in commutative algebra \cite[Exercise 4.11]{Eisenbud}.
The local characterization of lattices in any number field gives (ii) $\implies$ (iii) \cite[p. 97]{LangEll} and the converse is standard \cite[Theorem 9, p. 98]{LangEll}.
\end{proof}
\noindent
From now on we assume that $K$ is an imaginary quadratic field.  If $\OO' \subset \OO$ are quadratic orders in $K$, their discriminants are related as follows:
\[ \Delta(\OO') = [\OO:\OO']^2 \Delta(\OO). \]
For an order $\OO$ in $K$, we define the \textbf{conductor} $\mathfrak{f} =
[\OO_K:\OO]$.  Let us write $\Delta_K$ for $\Delta(\OO_K)$.  Then if $\OO$ has conductor
$\mathfrak{f}$, we have
\[ \Delta(\OO) = \mathfrak{f}^2 \Delta_K. \]
Observe that $\Delta(\OO)$ is negative and congruent to $0$ or $1$ modulo $4$; we call such integers \textbf{imaginary quadratic discriminants}.  For any $K$ and $\mathfrak{f} \in \Z^+$, $\Z[\frac{\mathfrak{f^2}\Delta_K +\mathfrak{f}\sqrt{\Delta_K}}{2}]$ is the unique order $\OO$ in
$K$ of conductor $\mathfrak{f}$ \cite[p. 90]{LangEll}.  It follows that for every imaginary quadratic discriminant $\Delta$, there is a unique imaginary quadratic order $\OO(\Delta)$ of discriminant $\Delta$.

\subsection{Basics on CM Elliptic Curves}
\textbf{} \\ \\ \noindent
Let $A_{/F}$ be an abelian variety over a field $F$.  By $\End A$ we mean
the ring of endomorphisms of $A_{/F^{\sep}}$, endowed with the structure of a
$\mathfrak{g}_F$-module.  It is known that $\End^{\circ} A = \End A \otimes_{\Z} \Q$ is a semisimple $\Q$-algebra and $\End A$ is an order in $\End^{\circ} A$.
When $F$ has characteristic $0$ and $A = E$ is an elliptic curve, $\End^{\circ} E$ is either $\Q$ or an imaginary quadratic field $K$: in the latter case we say
that $E$ has \textbf{complex multiplication (CM)}.  Thus $\End E$ is\footnote{An imaginary quadratic order $\OO$ has a unique nontrivial ring automorphism (complex conjugation), so there are two different ways to identify $\OO$ with $\End E$.  As is standard, we take the identification which is compatible with the action of $\OO$ on the tangent space at the origin.} an imaginary
quadratic order $\OO$, and we say that $E$ has $\OO$-\textbf{CM}. We summarize some basic facts of CM theory \cite[Fact1]{TORS1}. Proofs are found throughout the literature \cite{Cox89, LangEll, SilvermanII}.

\begin{fact}
\label{FACT1}
\begin{enumerate}[a)]

\item There exists at least one complex elliptic curve with $\mathcal{O}$-CM.
\item Let $E$, $E'$ be any two complex elliptic curves with $\mathcal{O}(\Delta)$-CM.  The $j$-invariants $j(E)$ and $j(E')$ are
Galois conjugate algebraic integers.  In other words, $j(E)$ is a root of some monic polynomial with $\Z$-coefficients, and
if $P(t)$ is the minimal such polynomial, $P(j(E')) = 0$ also.
\item Thus there is a unique irreducible, monic polynomial $H_\Delta(t) \in \Z[t]$ whose roots are the $j$-invariants of all $\mathcal{O}(\Delta)$-CM complex elliptic curves.
\item The degree of $H_\Delta(t)$ is the class number $h(\Delta) = \# \Pic \OO(\Delta)$.  In particular, when $\OO = \OO_K$ we have $\deg(H_\Delta(t)) = h(K)$,
the class number of $K$.
\item Let $F_\Delta := \Q[t]/H_\Delta(t)$.  Then $F_\Delta$ can be embedded in the real numbers, so in particular is linearly disjoint
from the imaginary quadratic field $K$.  Let $K_\Delta$ denote the compositum of $F_\Delta$ and $K$, the
  \textbf{ring class field} of the order $\OO$.  Then $K_\Delta/K$ is abelian,
with Galois group canonically isomorphic to $\Pic(\OO)$.
%  Moreover, $K_D/\Q$ is Galois and the exact sequence
%\[1 \ra \Gal(K_D/K) \ra \Gal(K_D/\Q) \ra \Gal(K/\Q) \ra 1 \]
%splits, i.e., $\Gal(K_D/\Q)$ is up to isomorphism the semidirect product of %$\Pic(\mathcal{O})$ with the cyclic group
%$Z_2$ of order $2$, where the map $Z_2 \ra \Aut(\Pic(\mathcal{O}))$ takes the %nontrivial element of $Z_2$ to inversion:
%$x \mapsto x^{-1}$.
\end{enumerate}

\end{fact}
\noindent
Let $E_{/\C}$ be an elliptic curve with $\OO$-CM; by the Uniformization
Theorem there is a lattice $\Lambda \subset \C$ such that $1 \in \Lambda$ and $E \cong \C/\Lambda$.  Then $\Lambda$ is a fractional $\OO$-ideal of $K$ and
$\OO(\Lambda) = \OO$, so by Lemma \ref{PROPERPROJECTIVE} $\Lambda$ is a  projective $\OO$-module. Conversely, if $\Lambda$ is a rank one projective $\OO$-module, then
$E_{\Lambda} = (\Lambda \otimes_{\OO} \C)/\Lambda$ is an elliptic curve, and the $\C$-isomorphism class of $E_{\Lambda}$ depends only on the isomorphism
class of $\Lambda$ as an $\OO$-module.  The map $\Lambda \mapsto E_{\Lambda}$
induces a bijection from $\Pic \OO$ to the set of isomorphism classes of
elliptic curves $E_{/\C}$ with $\End E \cong \OO$.
% Thus we have a
%distinguished $\OO$-CM elliptic curve $E_{\OO} = \C/\OO$, and we have
%$j(E_{\OO}) = j(\tau_D)$.

\subsection{$\R$-structures on Elliptic Curves}
\textbf{} \\ \\ \noindent
For a subset $S \subset \C$, we put $\overline{S} = \{\overline{z} \mid z \in S\}$.  We say a lattice $\Lambda \subset \C$ is \textbf{real} if $\overline{\Lambda} = \Lambda$.  For a lattice $\Lambda \subset \C$, we associate the complex torus $\C/\Lambda$ to the Weierstrass equation
\[ E_{\Lambda}: y^2 = 4x^3 - g_2(\Lambda)x - g_3(\Lambda), \]
via the Eisenstein series $g_2, g_3$ \cite[Proposition VI.3.6]{SilvermanII}.

\begin{lemma}
\label{realEC}
a) Let $E_{/\C}$ be an elliptic curve.  The following are equivalent: \\
(i) There is an elliptic curve $(E_0)_{/\R}$ such that $(E_0)_{/\C} \cong E$. \\
(ii) $j(E) \in \R$.  \\
(iii) $E \cong E_{\Lambda}$ for a real lattice $\Lambda$. \\
b) Let $\Lambda_1, \Lambda_2$ be real lattices.  The following are equivalent: \\
(i) There is $\alpha \in \R^{\times}$ such that $\Lambda_2 = \alpha \Lambda_1$.  \\
(ii) $E_{\Lambda_1}$ and $E_{\Lambda_2}$ are isomorphic as elliptic curves over $\R$.
\end{lemma}
\begin{proof}
To prove a), it is immediate that (i) $\implies$ (ii). For (ii) $\implies$ (iii), take $g_2,g_3 \in \R$ such that
$E': y^2 = 4x^3-g_2x - g_3$ is an elliptic curve with $j$-invariant $j(E)$ \cite[Proposition III.1.4]{Silverman}.  There is a \emph{unique} lattice $\Lambda \subset \C$ such that $g_2(\Lambda) = g_2$ and $g_3(\Lambda) = g_3$ and thus $E' \cong \C/\Lambda$ \cite[Theorem VI.5.1]{Silverman}.  Since $j(E) = j(E')$ and $\C$ is algebraically closed, we have $E \cong \C/\Lambda$.  Since $g_2(\Lambda) = \overline{g_2(\Lambda)} = g_2(\overline{\Lambda})$,  $g_3(\Lambda) = \overline{g_3(\Lambda)} = g_3(\overline{\Lambda})$, we have $\overline{\Lambda} = \Lambda$.  Finally, if $\Lambda$ is a real lattice then $g_2(\Lambda),g_3(\Lambda)\in \R$ and (iii) $\implies$ (i) is immediate.  (Alternately,
since $\overline{\Lambda} = \Lambda$, complex conjugation on $\C$ descends to an antiholomorphic involution on $\C/\Lambda$ and thus gives descent data for an
$\R$-structure on $E$.)  \\
To prove b), if  $\Lambda_1,\Lambda_2 \subset \C$, we have $\C/\Lambda_1 \cong \C/\Lambda_2$ iff $\Lambda_2 = \alpha \Lambda_1$ for some $\alpha \in \C^{\times}$. In terms of Weierstrass equations, $E_{\alpha \Lambda}$ is the quadratic twist of $E_{\Lambda}$ by $\alpha^2$.
Thus $E_{\Lambda_1} \cong E_{\alpha \Lambda_1} = E_{\Lambda_2}$ if $\alpha\in \R$.  Conversely, if $E_{\Lambda_1} \cong E_{\Lambda_2}$,
then the standard theory of Weierstrass equations \cite[$\S$ III.1]{Silverman}
shows: there is $\alpha \in \R^{\times}$ with $g_2(\Lambda_2) = \alpha^4
g_2(\Lambda_1) = g_2(\alpha^{-1} \Lambda_1)$, $g_3(\Lambda_2) = \alpha^6 g_3(\Lambda_1) = g_3(\alpha^{-1} \Lambda_2)$ and thus $\Lambda_2 = \alpha^{-1}
\Lambda_1$.
\end{proof}

\begin{lemma}
\label{LEMMA2.1}
a) Let $\Lambda$ be a real lattice.  If $j(E_{\Lambda}) \neq 1728$,
then $E_{i\Lambda}$ and $E_{\Lambda}$ are isomorphic as $\C$-elliptic curves
but \emph{not} as $\R$-elliptic curves.  If $j(E_{\Lambda}) = 1728$,
then $E_{\zeta_8 \Lambda}$ and $E_{\Lambda}$ are isomorphic as $\C$-elliptic curves but \emph{not} as $\R$-elliptic curves. \\
b) Let $j \in \R$.  Then there are precisely two $\R$-isomorphism classes
of elliptic curves $E_{/\R}$ with $j(E) = j$.
\end{lemma}
\begin{proof}
If $j \neq 1728$ then $g_3(\Lambda) \neq 0$, so
$g_3(i\Lambda) = - g_3(\Lambda)$, whereas as above any real change of variables
takes $g_3(\Lambda) \mapsto \alpha^{-6} g_3(\Lambda)$ for some $\alpha \in \R^{\times}$.  If $j = 1728$ then
$g_3(\Lambda) = 0$, so the above argument shows instead that $i\Lambda = \Lambda$
(as it should, since $\Lambda$ is homothetic to $\Z[i]$).  In this case $g_2(\Lambda) \neq 0$ and $g_2(\zeta_8 \Lambda) = -g_2(\Lambda)$, whereas
any real change of variables takes $g_3(\Lambda) \mapsto \alpha^{-4} g_3(\Lambda)$ for some $\alpha \in \R^{\times}$. The standard theory of real elliptic curves gives b) \cite[Prop. V.2.2]{SilvermanII}.
\end{proof}

\begin{lemma}
\label{LEMMA2.2}
Let $\OO$ be an order in the imaginary quadratic field $K$, and let $I$ be a proper fractional $\OO$-ideal.  The following are equivalent: \\
(i) $[I] = [\overline{I}] \in \Pic \OO$. \\
(ii) $I^2$ is principal. \\
(iii) $j(E_I) \in \R$.
\end{lemma}
\begin{proof}
a) Since $I \overline{I} = N_{K/\Q}(I) \OO$, we have $[\overline{I}] = [I]^{-1} \in \Pic \OO$, so (i) $\iff$ (ii).  Work of Shimura gives (ii) $\iff$ (iii) \cite[(5.4.3)]{Shimura}.
\end{proof}
\noindent
For an imaginary quadratic discriminant $\Delta$, let $\tau_\Delta = \frac{\Delta+\sqrt{\Delta}}{2}$, so $\OO(\Delta) = \Z[\frac{\Delta+\sqrt{\Delta}}{2}]$ is the imaginary quadratic order of  discriminant $\Delta$.  Then $j(\C/\OO) = j(\tau_\Delta)$.  Let $\sigma_1,\ldots,\sigma_{h(\Delta)}: \Q(j(\tau_\Delta))/\Q \hookrightarrow \C$ be the $\# \Pic \OO(\Delta)$ field embeddings, with $\sigma_1$ taken to be inclusion.  By Lemma \ref{LEMMA2.2}, $j(\tau_\Delta) \in \R$.  The other embeddings $\sigma_2,\ldots,\sigma_h$ may in general be either real or complex: Lemma \ref{LEMMA2.2} implies that the number of real embeddings is $\# (\Pic \OO)[2]$.

\begin{lemma}
\label{LEMMA2.5}
For an imaginary quadratic discriminant $\Delta$, let $r$ be the number of distinct
odd prime divisors of $\Delta$.  We define $\nu \in \N$ as follows: \\
\[\nu =
\begin{cases}
r-1, \ \Delta \equiv 1 \pmod{4} \text{ or }  \Delta \equiv 4 \pmod{16} \\
r, \ \Delta \equiv 8,12 \pmod{16} \text{ or } \Delta \equiv 16 \pmod{32} \\
 r+1, \ \Delta \equiv 0 \pmod{32}.
\end{cases} \]
a) We have $\Pic \OO(\Delta)[2] \cong (\Z/2\Z)^{\nu}$. \\
b) In particular, $h(\Delta)$ is odd iff $\Delta \in \{-4,-8,-16\}$ or $\Delta = - 2^{\epsilon} \ell^{2a+1}$ for $\epsilon \in \{0,2\}$, $\ell \equiv 3 \pmod{4}$ a prime and $a \in \N$. \\
c) There are precisely $2^{\nu+1}$ $\R$-homothety classes of $\OO$-CM real lattices.
\end{lemma}
\begin{proof}
Part a) is essentially due to Gauss \cite[Proposition 3.11]{Cox89},\cite[Theorem 5.6.11]{Halter-Koch}.  Part b) follows immediately.  Part c) is obtained by combining part a) with Lemmas \ref{LEMMA2.1} and \ref{LEMMA2.2}.
\end{proof}
\noindent
A fractional $\OO$-ideal $I$ is \textbf{primitive} if $I \subset \OO$ and for all $e \geq 2$, $I \not \subset e \OO$.

\begin{lemma}
\label{lemma3.6}
a) Let $E \cong_{\R} E_{\Lambda}$ be a real $\OO$-CM elliptic curve.  The $\R$-homothety class of $\Lambda$ contains a unique primitive $\OO$-ideal $I$.
The ideal $I$ is proper and real.  \\
b) (\cite[Theorem 5.6.4]{Halter-Koch}) There are precisely $2^{\mu}$ primitive proper real $\OO$-ideals.
\end{lemma}
\begin{proof}
The lattice $\Lambda$ contains an element $a+bi$ with $a \neq 0$.  Since $\Lambda$ is real by Lemma \ref{realEC}, we have $a-bi \in \Lambda$ and thus also $2a = (a+bi) + (a-bi) \in \Lambda$.
Then $\frac{1}{2a} \Lambda$ is a proper $\OO$-ideal which is $\R$-homothetic to $\Lambda$.
%a) Fix $\alpha \in \Lambda \setminus \{0\}$.  Then $\frac{1}{\alpha} \Lambda$ is a %proper fractional $\OO$-ideal.  All fractional $\OO$-ideals in the class of $\Lambda$ %are proper, and by Lemma \ref{LEMMA2.2}, the $\C$-homothety class of $\Lambda$ %contains a real $\OO$-ideal $I$.  We put
%\[ I' = \sqrt{-D} I, \ D \neq -4 \]
%\[ I' = (1+i) I, D = -4. \]
%Then $I'$ is also a proper real $\OO$-ideal.  Since $I'$ is $\R$-homothetic to
%$i I$ if $D \neq -4$ and $\zeta_8 I$ if $D = -4$, by Lemma \ref{LEMMA2.1}a), $I'$ is
%\emph{not} $\R$-homothetic to $I$.  By Lemma \ref{LEMMA2.1}b), $E$ is isomorphic over %$\R$ to exactly one of $E_{I}$, $E_{I'}$.
If two fractional $\OO$-ideals are
$\R$-homothetic, then one is real iff the other is real, and the $\R$-homothety class of any fractional $\OO$-ideal contains a unique primitive $\OO$-ideal. To prove part b), combine part a) with Lemma \ref{LEMMA2.5}b).
\end{proof}
\noindent
Ideals of this type are completely classified.  We use the following notation:
for $\alpha,\beta \in \C$ which are linearly independent over $\R$, we define the lattice
\[ [\alpha,\beta] = \{ a \alpha + b \beta \mid a,b \in \Z\}. \]

\begin{thm}
\label{HalterKoch} \cite[Theorem 5.6.4]{Halter-Koch}
Let $\mathcal{O}$ be an order in $K$ of discriminant $\Delta$. A primitive proper $\OO$-ideal $I$ is real iff it is one of the following two types:
%\begin{enumerate}
%\item
\begin{equation}
\label{HKONE}
I =\left[a, \frac{\sqrt{\Delta}}{2} \right], \text{ where } a \in \Z^+, \ 4a | \Delta \text{ and } \gcd \left(a, \frac{\Delta}{4a}\right)=1
\end{equation}
\begin{equation}
\label{HKTWO}
I=\left[a, \frac{a+\sqrt{\Delta}}{2} \right], \text{ where } a \in \Z^+,\ 4a | a^2-\Delta \text{ and } \gcd \left(a, \frac{a^2-\Delta}{4a}\right)=1.
\end{equation}
%\end{enumerate}
Moreover, if $\Delta \equiv 1 \pmod{4}$, there are no such ideals of type (1).
\end{thm}
%\begin{proof}
%See Theorem 5.6.4 in Halter-Koch.
%\end{proof}

\begin{cor}
\label{COR3.8}
Let $I$ be a primitive proper real $\OO$-ideal.  Then $[\OO:I] \mid \Delta$.
\end{cor}
\begin{proof}
For all ideals $I$ of the form (\ref{HKONE}) and (\ref{HKTWO}) above, we have that $[\OO:I] = a \mid \Delta$ \cite[Theorem 5.4.2]{Halter-Koch}.
\end{proof}

\noindent
Let $R$ be a domain with fraction field $K$.  For fractional $R$-ideals $I$ and $J$ we define the \textbf{colon ideal}
\[ (J:I) = \{x \in K \mid x I \subset J\}; \]
it is also a fractional $R$-ideal.  If $I$ is invertible, then $(J:I) = I^{-1} J$.
\\ \\
Let $E_{/F}$ be an $\OO$-CM elliptic curve defined over a number field, and let $I$ be a nonzero
ideal of $\OO$.  We then have an \textbf{I-torsion kernel}
\[ E[I] = \{x \in E(\overline{F}) \mid \forall \alpha \in I, \  \alpha x = 0 \}. \]
If $I \subset J$ then $E[J] \subset E[I]$.  Since $I$ contains the positive
integer $[\OO:I]$, $E[I] \subset E[[\OO:I]]$ and thus $E[I]$ is a finite $\OO$-submodule of $E(\overline{F})$.  Evidently $E[I]$ is stabilized by the
action of $\mathfrak{g}_{FK}$, so it corresponds to a finite \'etale $FK$-subgroup scheme
of $E$.  If $FK \supsetneq F$, then the nontrivial element $c \in \Aut(FK/F)$
acts as complex conjugation on $\End^0 E$ and thus $c(E[I]) = E[\overline{I}]$.
It follows that $E[I]$ is defined over $F$ iff $I$ is real.  For for any nonzero ideal $I$ of $\OO$ we have an isogeny $E \ra E/E[I]$, defined over $FK$ in general and over $F$ if $I$ is real. \\ \indent  Let $\iota: F \hookrightarrow \C$ be a field embedding.  Then $E \cong_{\C} E_{\Lambda}$
for some proper $\OO$-ideal $\Lambda$.  Observe that the kernel of the natural map
$E_{\Lambda} \ra E_{(\Lambda:I)}$ is $(\Lambda:I)/\Lambda = (\C/\Lambda)[I]$, so
$E/E[I] \cong_{\C} E_{(\Lambda:I)}$.  Furthermore, if $I$ is invertible,
then $E_{(\Lambda:I)} = E_{I^{-1} \Lambda}$.  Thus we get an explicit
description of the $I$-torsion kernel and the associated isogeny in terms of
uniformizing lattices.  \\ \indent
Now let $\iota: F \hookrightarrow \R$ be a field embedding.  By Lemma \ref{realEC} we may choose $\Lambda$ to be a real lattice.  Suppose that $I$ is moreover real.  Then so is $(\Lambda:I)$ and thus $E_{\Lambda} \ra E_{(\Lambda:I)}$ is an explicit description of the $I$-torsion kernel and the associated isogeny in
terms of uniformizing real lattices.

\begin{thm}
\label{THM3.9}
Let $\Delta$ be an imaginary quadratic discriminant. \\
a) Let $F \subset \R$, and let $E_{/F}$ be an $\OO$-CM elliptic curve.  Then there is an $\OO$-CM elliptic curve $E'_{/F}$ such that $E' \cong_{\R} E_{\OO}$ and
an $F$-rational isogeny $\varphi: E \ra E'$.  \\
b) Let $N$ be a positive integer which is prime to $\Delta$.  Then the isogeny
$\varphi$ of part a) induces a $\mathfrak{g}_F$-module isomorphism $E[N] \stackrel{\sim}{\ra} E'[N]$.
\end{thm}
\begin{proof}
There is a primitive, proper, real $\OO$-ideal $\Lambda$ such that
$E \cong_{\R} E_{\Lambda}$.  Let $I$ be any proper, real $\OO$-ideal.  As above
there is an $F$-rational isogeny $\varphi_I: E \ra E/E[I]$ which over $\R$ is given as $E_{\Lambda} \ra E_{(\Lambda:I)} = E_{I^{-1} \Lambda}$.  Taking $\varphi = \varphi_{\Lambda}$ establishes part a).   The degree of $\varphi_{\Lambda}$ is
$[\OO:\Lambda]$, which by Corollary \ref{COR3.8} divides $\Delta$.  It follows
that $\deg \varphi_{\Lambda}$ is prime to $N$ and thus $\varphi_{\Lambda}: E[N] \stackrel{\sim}{\ra} E'[N]$.
\end{proof}

\begin{remark}
A work of S. Kwon \cite{Kwon99a} contains related results.  Especially, the discussion preceding Theorem \ref{THM3.9}
generalizes \cite[Prop. 2.3.1]{Kwon99a}, and the classification of primitive, proper real ideals in an imaginary quadratic order is given in \cite[$\S 3$]{Kwon99a} and is used to classify cyclic isogenies on CM elliptic curves rational over $\Q(j(E))$.
\end{remark}

\subsection{Cartan Subgroups}
\textbf{} \\ \\ \noindent
%{\color{red} The material in this section should be better ordered.   material %after Lemma \ref{LEMMA2.6} does not depend on the previous material in this %section but is rather general nonsense about $\ell$-adic representations on %abelian varieties over any field.}
%\\ \\
Let $\Lambda \subset \C$ be a lattice, and let $E_{\Lambda} = \C/\Lambda$.  For $N \in \Z^+$ and $\ell \in \mathcal{P}$, put
\[ \Lambda_N = (\frac{1}{N})\Lambda/\Lambda = E_{\Lambda}[N] \]
%
%\cong \Lambda \otimes \Z/N\Z, \]
\[ T_{\ell} \Lambda = \varprojlim \Lambda_{\ell^n} = T_{\ell}(E_{\Lambda}). \]
%\cong \Lambda \otimes \Z_{\ell}, \]
%\[ \widehat{\Lambda} = \prod_{\ell \in \mathcal{P}} T_{\ell} \Lambda . \]
If $F \subset \C$ and $E_{/F}$ is an elliptic curve, then $E_{/\C} \cong E_{\Lambda}$ for some lattice $\Lambda$, uniquely determined up to homothety.  If $F \subset \R$, then $E_{/\R} \cong E_{\Lambda}$ for some real lattice $\Lambda$, uniquely determined up to real homothety.
\\ \\
We have $E(\C)[\tors] = E(\overline{F})[\tors]$, so $\Aut(\C/F)$ acts on
$\Lambda_N$, $T_{\ell} \Lambda$ and $\widehat{\Lambda}$.   We assume that
$\Aut(\C/\overline{F})$ is a normal subgroup of $\Aut(\C/F)$: this holds
if $F$ is a number field or if $F = \R$.  Then we get an induced action of
$\mathfrak{g}_F$ on $\Lambda_N$.  If moreover $F \subset \R$, then complex conjugation $c \in \Aut(\C/\R) \subset \Aut(\C/F) \twoheadrightarrow \mathfrak{g}_F$
acts on $\Lambda_N$.
\\ \\
Let $\OO$ be an imaginary quadratic order.  For $N,\ell$ as above, consider the $\OO$-algebras
\[ \OO_N = \OO \otimes_{\Z} \Z/N\Z, \]
\[ T_{\ell}(\OO) =  \OO \otimes_{\Z} \Z_{\ell} = \varprojlim \OO_{\ell^n}. \]
%\[ \widehat{\OO} = \OO \otimes_{\Z} \hat{\Z} = \prod_{\ell} T_{\ell} \OO. \]
Let $E_{/F}$ be an $\OO$-CM elliptic curve.  As above, there is a proper integral
$\OO$-ideal $\Lambda$, with uniquely determined class in $\Pic \OO$, such that $E_{/\C} \cong E_{\Lambda}$.  Then $\Lambda_N$ (resp. $T_{\ell} \Lambda$) has a natural $\OO_N$-module (resp. $T_{\ell}( \OO)$-module) structure.  If $F \subset \R$ we may take $\Lambda$ to be a
real ideal: $\overline{\Lambda} = \Lambda$.

\begin{lemma}
\label{LEMMA2.6} \textbf{} \\
a) For every $N \in \Z^+$, $\Lambda_N = E[N]$ is free of rank $1$ as an $\OO_N$-module.  \\
b) $T_{\ell}\Lambda = T_{\ell}(E)$ is free of rank $1$ as $T_{\ell}(\OO)$-module.
\end{lemma}
\begin{proof} Part b) is known by work of Serre and Tate \cite[p. 502]{Serre-Tate68} while part a) can be deduced from work of Parish \cite[Lemma 1]{Parish89}. Either part can be used to deduce the other.
\end{proof}
\noindent
In particular, for all primes $\ell$ we have a homomorphism of $\Z_{\ell}$-algebras
\[ \iota_{\ell}: T_{\ell}(\OO) \rightarrow \End T_{\ell} (E). \]
The map $\iota_{\ell}$ is $\ggg_F$-equivariant; further, it is injective with torsion-free cokernel \cite[Lemma 12.2]{Milne}.  Tensoring with $\Z/\ell^n\Z$ and applying primary decomposition, we get for each $N \in \Z^+$ an injective $\ggg_F$-equivariant ring homomorphism
\[ \OO_N \hookrightarrow \End E[N]. \]
%{\color{red} The injectivity also follows from Lemma \ref{LEMMA2.6}.  Oh well; %better two explanations than none!}
Tensoring to $\Q_{\ell}$ gives
\[ \iota_{\ell}^0: V_{\ell}(\OO) \hookrightarrow \End V_{\ell}(E). \]
\\
We define the \textbf{Cartan subalgebras}
\[ \mathcal{C}_{\ell} = \iota_{\ell}(T_{\ell}(\OO)) \subset \End T_{\ell}(E) ,\]
\[ \mathcal{C}_{\ell}^0 = \iota_{\ell}^0(V_{\ell}(\OO)) \subset \End V_{\ell}(E) \]
and the \textbf{Cartan subgroups}
\[ \mathcal{C}_{\ell}^{\times} \subset \Aut T_{\ell}(E),\]
\[ (\mathcal{C}_{\ell}^0)^{\times} \subset \Aut V_{\ell}(E). \]
Then $\mathcal{C}_{\ell}^0 \cong K \otimes \Q_{\ell}$ is a maximal \'etale
subalgebra of $\End V_{\ell}(E) \cong M_2(\Q_{\ell})$.  We may view
$\mathcal{C}_{\ell}^{0} \hookrightarrow M_2(\Q_{\ell})$ as the regular representation.
We write $C(\mathcal{C}_{\ell}^0)$ for the commutant and $N(\mathcal{C}_{\ell}^0)$
for the normalizer of $\mathcal{C}_{\ell}^0$ inside $\End V_{\ell}(E)$.  By the Double Centralizer Theorem  \cite[$\S 12.7$]{Pierce}, we have \[C(\mathcal{C}_{\ell}^0) = \mathcal{C}_{\ell}^0. \]
Using the Skolem-Noether Theorem \cite[$\S 12.6$]{Pierce}, we find that
\[ N \mathcal{C}_{\ell}^0/ (\mathcal{C}_{\ell}^0)^\times \cong \Aut_{\Q_{\ell}}
\mathcal{C}_{\ell}^0 \]
has order $2$.

The fixed field of the kernel of the representation $\ggg_F \ra V_{\ell}(\OO)$ is
$FK$, so
\[ \rho_{\ell^{\infty}}(\ggg_{FK}) \subset \mathcal{C}_{\ell}^0, \]
and if $FK \supsetneq F$ then
\[ \rho_{\ell^{\infty}}(\ggg_F) \not \subset \mathcal{C}_{\ell}^0. \]
In fact \cite[$\S 4$, Corollary 2]{Serre-Tate68} we have
\[ \rho_{\ell^{\infty}}(\ggg_{FK}) \subset
\iota_{\ell}(T_\ell(\OO)^\times). \]
Moreover, $\ggg_F$-equivariance of $\iota_{\ell}^0$ gives
\[ \rho_{\ell^{\infty}}(\ggg_F) \subset N \mathcal{C}_{\ell}^0. \]
This recovers a standard result of Serre \cite[Theorem 5]{Serre66}.

%Lemma \ref{NONSCALARLEMMA} implies that for all $N \geq 3$,
%complex conjugation $c$ acts nontrivially on $\Aut E[N]$ and thus also
%on $\Aut T_{\ell}(E)$ for all primes $\ell$.  On the other hand,
%since $\mathcal{C}_{\ell}^0$ is its own commuting algebra,
%\[\rho_{\ell}|_{KF} \subset \mathcal{C}_{\ell}^0 \]
%and it follows that $c$ normalizes
%$\mathcal{C}_{\ell}^0$.  So $c$ lies in $N \mathcal{C}_{\ell}^0$.

% and thus represents the unique nontrivial coset of $C_{\ell_{\infty}}^0$.  %This implies %the following important result.

\begin{lemma}
%The $\ell$-adic Galois representation lies in the normalizer of the Cartan %subgroup.
Let $G_{\ell^{\infty}} = \rho_{\ell^{\infty}}(\ggg_F)$ be the image of the $\ell$-adic Galois representation.  The following are equivalent: \\
(i) $G_{\ell^{\infty}}$ lies in the Cartan subgroup.  \\
(ii) $G_{\ell^{\infty}}$ is commutative.  \\
(iii) $K \subset F$.
\end{lemma}
%\begin{proof} See \cite[Theorem 5]{Serre66}.
%\end{proof}
\noindent
We now deduce a stronger version of a result of Serre \cite[Lemma 15]{TORS1}. Let us first note the following in the case $\Lambda = \OO$.

\begin{lemma}
\label{LEMMA2.3}
Let $K = \Q(\sqrt{\Delta_0})$ be an imaginary quadratic field, and let
$\OO$ be an order in $K$ of discriminant $\Delta = \mathfrak{f}^2 \Delta_0$: thus $\OO = \Z\left[\frac{\Delta+\sqrt{\Delta}}{2}\right]$.  Let $c$ be the nontrivial element of
$\Aut(K/\Q)$.  Let $N \geq 2$, put $\OO_N = (1/N)\OO/\OO$ and $i \OO_N =
(\frac{i}{N})\OO/(i\OO)$. \\
a) If $\Delta$ is even or $N$ is odd, then $\frac{1}{N},\frac{\sqrt{\Delta}}{N}$ (resp. $\frac{i}{N}, \frac{\sqrt{-\Delta}}{N}$) is a $\Z/N\Z$-basis for $\OO_N$ (resp. $i\OO_N)$.  The corresponding matrix of $c$  is $\left[ \begin{array}{cc} 1 & 0 \\
 0 & -1 \end{array} \right]$ (resp.  $\left[ \begin{array}{cc} -1 & 0 \\
 0 & 1 \end{array} \right]$).  \\
b) In all cases $\frac{1}{N}$, $\frac{\Delta+\sqrt{\Delta}}{2N}$ (resp. $i \left(\frac{\Delta+\sqrt{\Delta}}{2N}\right)$) is a $\Z/N\Z$-basis for
$\OO_N$ (resp. $i\OO_N$).  The corresponding matrix of $c$ is
$\left[ \begin{array}{cc} 1 & \Delta \\ 0 & -1 \end{array} \right]$ (resp. $\left[ \begin{array}{cc} -1 & -\Delta \\ 0 & 1 \end{array} \right]$).
%   Moreover
%$\frac{i}{N}$, $i \left(\frac{D+\sqrt{D}}{2N}\right)$ is a $\Z/N\Z$-basis for %$i\OO_N$,
%and the matrix of $c$ with respect to this basis is $\left[ \begin{array}{cc} -1 & -D %\\ 0 & 1 \end{array} \right]$.
\end{lemma}

\begin{cor}
\label{COR2.4}
a) If $N \geq 3$, then $c$ acts nontrivially on $\OO_N$.  \\
b) If $N = 2$, then $c$ acts nontrivially on $\OO_N$ iff $\Delta$ is odd.
\end{cor}

\begin{lemma}
\label{LEMMA2.7}
Let $K$ be an imaginary quadratic field, and let $\OO$ be an order in $K$
of discriminant $\Delta = \mathfrak{f}^2 \Delta_0$.   Let $F$ be a field of characteristic $0$, and let $E_{/F}$ be an $\OO$-CM elliptic curve.  Let $N \in \Z^+$, and suppose at least one of the following holds: \\
$\bullet$ $N \geq 3$; \\
$\bullet$ $N = 2$ and $\Delta$ is odd. \\
Then $F(E[N]) \supset K$.
\end{lemma}
\begin{proof}
We may certainly assume $K \not \subset F$.  Let $\sigma \in \ggg_F$ be any element
which restricts nontrivially to $KF$.  Then $\sigma$ acts on  $\OO_N$ as $c$ acts on $\OO_N$, and by Corollary \ref{COR2.4}, this action is nontrivial.  Since $\iota_N: \OO_N \hookrightarrow \End E[N]$
is injective and $\ggg_F$-equivariant, it follows that $\sigma$ acts nontrivially
on $\End E[N]$. For any $G$-module $M$, if $\sigma\in G$ acts nontrivially on $\End(M)$ then $\sigma$ acts nontrivially on $M$.
\end{proof}

\subsection{Ray Class Field Containment}

\begin{thm}
\label{SCHERTZTHM}
\label{3.16}
Let $\OO$ be an order in an imaginary quadratic field $K$.  Let $F$ be a field of characteristic $0$, and let $E_{/F}$ be an $\OO$-CM elliptic curve.  Let
$N \in \Z^+$.  Let
$\mathfrak{h}_{/F}: E \ra \PP^1$ be a \textbf{Weber function} for $E$: that is, $\mathfrak{h}$ is the composition of the quotient map $E \ra E/(\Aut E)$ with an isomorphism
$E/(\Aut E) \cong \PP^1$. Then the field $FK(\mathfrak{h}(E[N]))$ contains the $N$-ray class field $K^{(N)}$ of $K$.
\end{thm}
\begin{remark}
When $\OO = \OO_K$, the \emph{equality} $K(j(\C/\OO_K), \mathfrak{h}(E[N])) = K^{(N)}$ is one of the central results of classical CM theory \cite[Theorem II.5.6]{SilvermanII}.
%The general case may also be ``classically known'', but for lack of a suitable %reference we include a proof.  %In fact we include two proofs: the first one is intended for readers who are %using Silverman's text \cite{SilvermanII} as a primary reference and shows how %to modify the proof given there for $\OO_K$ to the general case.  The second %uses the adelic perspective \cite{LangEll}.
\end{remark}

\begin{proof} We use the results and notation of Lang \cite[$\S 10.3$]{LangEll}.  Applying Theorem 7 first with $\mathfrak{a} = \OO$ and $u = \frac{1}{N}$ and then with $\mathfrak{a} = \OO_K$ and $u = \frac{1}{N}$. We observe that for an idele $b$, $b \OO = \OO \implies b \OO_K = \OO_K$.  This is much as in \cite[Thm. 6, $\S 10.3$]{LangEll}.  We conclude
\[ L \supset K(j(\C/\OO),\mathfrak{h}(\frac{1}{N}+\OO)) \supset K(j(\C/\OO_K),\mathfrak{h}(\frac{1}{N}+\OO_K)). \]
But as an $\OO_K$-module, $\frac{1}{N} \OO_K/\OO_K$ is generated by $\frac{1}{N} + \OO_K$ \cite[p. 135]{LangEll}, so
\[ K(j(\C/\OO_K),\mathfrak{h}(\frac{1}{N}+\OO_K)) = K(j(\C/\OO_K),\mathfrak{h}(E[N])) = K^{(N)}.  \qedhere\]
\end{proof}

\section{Restrictions on the Torsion Subgroup}
\noindent
The goal of this section is to assemble a toolkit of results on torsion points on CM elliptic curves defined over number fields $E_{/F}$.  In many of these results the number field $F$ is not as restricted as it will be in the later results of the paper but the conclusion is of a
rather intricate or technical nature.  There is also a mixture of old and new: we draw from the work of Parish \cite{Parish89}, Silverberg \cite{Silverberg88}, \cite{Silverberg92}, Prasad-Yogananda \cite{PY01} and Aoki \cite{Aoki95}, \cite{Aoki06} and establish variants,
make refinements and extract key consequences.
\\ \\
Throughout this section we will use the following setup: $\OO$ is an imaginary quadratic order with fraction field $K$.  Let $\Delta_K$ be the discriminant of $\OO_K$, $\mathfrak{f}$ the conductor of $\OO$, and $\Delta$ the discriminant of $\OO$, so $\Delta = \mathfrak{f}^2 \Delta_K$.  Let $F$ be a subfield of $\C$, and let $E_{/F}$ be an $\OO$-CM elliptic curve. Again $\mathfrak{h}_{/F}$ will denote a Weber function.
%The morphism $\mathfrak{h}$ is defined over $FK$.

\subsection{Points of Order 2}
\textbf{} \\ \\ \noindent
%{\color{red} It's not clear to me that the material of this section belongs in %the published version of this paper, but it's a cool application of some of the %technical results we've established, so I'll write it out for now.}
%\\ \\
Let $\OO$ be an imaginary quadratic order of discriminant $\Delta < -4$, with fraction field $K$.  Let $E_{/\C}$ be an $\OO$-CM elliptic curve.  Let $F = \Q(j(E))$ and let $L = F(E[2])$, so $L/F$ is Galois of
degree dividing $6$.  By Fact 1, the isomorphism class of $F$ depends only on $\Delta$.  Since $\Delta < -4$, the $x$-coordinate is a Weber function on $E$, and thus $L$ does not depend upon the chosen Weierstrass model (any two $\OO$-CM elliptic curves with the same $j$-invariant are \emph{quadratic} twists of each other, and $2$-torsion points are invariant under quadratic twist).  Thus as an abstract number field and a Galois extension thereof, $F$ and $L$
depend only on $\Delta$.

\begin{thm}(Parish)
\label{PARISH2}
\label{4.1}
For all $\Delta< -4$, $F \subsetneq L$.
\end{thm}
\begin{proof}
Equivalently, $E$ does not have full $2$-torsion defined over $F = \Q(j(E))$.
This follows from more precise results of Parish \cite[Table 1]{Parish89}.
\end{proof}

\begin{thm}
\label{THM4.2}
\label{4.2}
Let $\OO$ be an imaginary quadratic order of discriminant $\Delta < -4$ and with fraction field $K$.  Let $E_{/\C}$ be an elliptic curve with $\OO$-CM.  Let $F = \Q(j(E))$, and let $L = F(E[2])$.  \\
%a) Up to isomorphism, the number fields $F$ and $L$ depend only on $D$.  \\
a) We have $K \subset L$ iff $\Delta$ is odd. \\
b) If $\Delta \equiv 1 \pmod{8}$, then $L = FK$ and $[L:F] = 2$.  \\
c) If $\Delta \equiv 5 \pmod{8}$ and $K \neq \Q(\sqrt{-3})$ then $L = K^{(2)}F$
and $[L:F] = 6$.  \\
d) If $\Delta$ is even, then $[L:F] = 2$.
\end{thm}
\begin{proof}
%a) That the isomorphism class of $F$ depends only on $D$ follows from Fact 1.
%Since $D < -4$, the $x$-coordinate is a Weber function on $E$ and thus
%$E[2]$ is independent of the model.  Alternately: $D < -4$ means that any two
%elliptic curves with the same $j$-invariant are quadratic twists of each other, %and quadratic twists do not change the $2$-torsion.  \\
a) If $\Delta$ is odd, then Lemma \ref{LEMMA2.7} gives $K \subset L$.  Suppose
$\Delta$ is even.  Then Lemma \ref{LEMMA2.3} implies $E[2](\R) = \Z/2\Z \times \Z/2\Z$.  Since
$K \not \subset \R$, the result follows.  \\
b) If $\Delta \equiv 1 \pmod{8}$, then the mod $2$ Cartan subgroup is isomorphic
to $(\Z/2\Z)^{\times} \times (\Z/2\Z)^{\times}$, better known as the trivial group.  It
follows that $FK(E[2]) = FK$.  Together with part a) this shows
$L = F(E[2]) = FK$.  \\
c) By part a) we have $K \subset L$, so by Theorem \ref{SCHERTZTHM}
we have $L \supset K^{(2)}F$.  Since $\Delta$ is odd, $K(j)$ and $K^{(2)}$
are linearly disjoint over $K^{(1)}$, so
\[ [K^{(2)}:K^{(1)}] = [K(j)K^{(2)}:K(j)K^{(1)}] = [FK^{(2)}:FK^{(1)}] \mid
[L:FK] = \frac{[L:F]}{2}. \]  Since $K \neq \Q(\sqrt{-3})$,
we have (c.f. Proposition \ref{4.12}) $[K^{(2)}:K^{(1)}] = 3$.
 Since for any
elliptic curve $E_{/F}$ we have $[F(E[2]):F] \mid 6$, the result
follows.  \\
d) Since $\Delta$ is even, the mod $2$ Cartan subgroup is cyclic of order $2^2-2 = 2$, and thus $[FK(E[2]):FK] \mid 2$.  It follows by using the result of part a) that $K \not \subset F(E[2])$, or just the fact that $[L:F] \mid 6$ so
we cannot have $[L:F] = 4$, that $[L:F] \mid 2$.  Combining with Theorem \ref{PARISH2} we get the result.
\end{proof}
%\noindent
%{\color{red} The above proof is not the one I had originally wanted.  I was %using class field theoretic arguments to show part d), but in some cases this %involved knowing extensions of Theorem \ref{SCHERTZTHM}.  One interesting case %is when $D = -4p$ for a prime $p \equiv 7 \pmod{8}$.  In this case we have
%$K^{(2)} = K^{(1)}$, so Theorem \ref{SCHERTZTHM} is too weak.  Also $\# \Pic %\OO$ is odd in this case, which is the unfavorable case for deducing Corollary %\ref{COR4.3}.  However, computations suggest in this case a beautifully %explicit result:
%$L = F(\sqrt{p})$.  I wish I knew enough explicit class field theory to prove %this!}

\begin{cor}
\label{COR4.3}
\label{4.3}
Suppose $\Delta \neq -4$.   Let
$F$ be a number field, and let $E_{/F}$ be an $\OO(\Delta)$-CM elliptic curve.
If $\Z/2\Z \times \Z/2\Z \subset E(F)$, then $2 \mid [F:\Q]$.
\end{cor}
\begin{proof}
If $\Delta = -3$, then by Theorem \ref{4.2}a) we have $F \supset K$ hence $2 \mid [F:\Q]$.  Otherwise we have $\Delta < -4$, so $\Q(E[2]) = \Q(\mathfrak{h}(E[2]))$, and the $2$-torsion field is
independent of the model of $E$. We may thus assume without loss of generality that
$E$ is obtained by base extension from an elliptic curve $E_{/\Q(j(E))}$.  Applying Theorem \ref{THM4.2} we get
\[ 2 \mid [\Q(j(E),E[2]):\Q(j(E))] \mid [F:\Q]. \qedhere \]
\end{proof}

\begin{remark}
\label{4.4}
The elliptic curve $E_{/\Q}: y^2 = x^3-x$ shows that the hypothesis $\Delta \neq -4$ in Corollary \ref{COR4.3} is necessary.
\end{remark}

\begin{cor}\label{odd2tors} \label{4.5} If $[F:\Q]$ is odd, $E_{/F}$ is a CM elliptic curve and $(\Z/2\Z)^2\subset E(F)$ then $E(F)[12] = (\Z/2\Z)^2$.
\end{cor}
\begin{proof}
By Corollary \ref{COR4.3} we may assume $\Delta = -4$.  It is
enough to show $E(F)$ has no subgroup isomorphic to $\Z/6\Z$ or
$\Z/2\Z \times \Z/4\Z$.  This follows from Table 2.
\end{proof}
\noindent
This result will be sharpened in $\S 5$.

\subsection{Bounds of Silverberg and Prasad-Yogananda Type}

%{\color{red} Note that the statement of the following theorem has been subtly %weakened since the last draft: the case in which $\ell$ is inert in $K$
%but divides the discriminant of $\OO$ is now counted as ramified, and the bound %is weaker.  I suspect that the previous version was false.}

\begin{thm}
\label{OLDNEWTHM2}
\label{4.6}
Let $\OO$ be an imaginary quadratic order of discriminant $\Delta=\mathfrak{f}^2\Delta_K$; put $K = \Q(\sqrt{\Delta})$.  Let $F$ be a number field, and let $E_{/F}$ an elliptic curve
with $\mathcal{O}$-CM.  Let $w(K) = \# \mathcal{O}_K^{\times}$.  Suppose  $E(F)[\tors]$ contains a point of prime order $\ell > 2$.  \\
a) If $(\frac{\Delta}{\ell}) = -1$, then
\[\left( \frac{2(\ell^2-1)}{w(K)} \right)h(K) \mid [FK:\Q]. \]
b) If $(\frac{\Delta}{\ell}) = 1$, then
\[ \left( \frac{2(\ell-1)}{w(K)} \right) h(K) \mid [FK:\Q]. \]
c) If $(\frac{\Delta}{\ell}) = 0$, then:
\begin{enumerate}
\item[1.] $ (\ell-1) h(K) \mid [FK:\Q]$ if $\ell$ is ramified in $K$.
% and $\Delta_K < -4$.
\vspace{1mm}
\item[2.] $ \left( \dfrac{2(\ell-1)^2}{w(K)} \right) h(K) \mid [FK:\Q]$ if $\ell$ is split in $K$.
\vspace{1mm}
\item[3.] $ \left( \dfrac{2(\ell^2-1)}{w(K)} \right) h(K) \mid [FK:\Q]$ if $\ell$ is inert in $K$.
\end{enumerate}

\end{thm}

\begin{proof}
\noindent
The cases $\OO = \OO_K$ and $\ell \nmid \Delta$ of Theorem \ref{OLDNEWTHM2} were proved in \cite[Theorem 2]{TORS1}.  The hypothesis $\OO = \OO_K$ comes into the proof only via the statement that $K(j(E))(\mathfrak{h}(E[N])) = K^{(N)}$, the $N$-ray class field of $K$.  And in fact we used only that the former field contains the latter field, which holds for all $\OO$ by Theorem \ref{SCHERTZTHM}.  So it remains to
consider the case in which $\ell \mid \Delta$.  If $\Delta = -3$ then $\ell = 3$ and the result is clear; if $\Delta = -4$ there
is no such $\ell$.  Henceforth we assume $\Delta < -4$.
Since $\ell \mid \Delta$ we have by \cite[$\S 2.3$]{TORS1} that
\[ \OO_{\ell} = \OO \otimes \Z/\ell \Z \cong \F_{\ell}[t]/(t^2). \]
Thus its image $C_{\ell} = \iota(\OO_{\ell}) \subset \End E[\ell] \cong M_2(\F_{\ell})$ is generated over the scalar matrices by a single nonzero nilpotent matrix
$g$.  Since the eigenvalues of $g$ are $\F_{\ell}$-rational we can put it in Jordan canonical form over $\F_{\ell}$.  We get a choice of basis $e_1,e_2$ of $E[\ell]$ such that
\[ \mathcal{C}_{\ell} \cong \left\{ \left[ \begin{array}{cc} \alpha & \beta \\ 0 & \alpha \end{array} \right] \mid  \alpha,\beta \in \F_{\ell} \right\} .\]
\\
Let $x=a e_1+b e_2 \in E(F)$ have order $\ell$.  For all $S = \left[
  \begin{array}{ c c }
     \alpha & \beta \\
     0 & \alpha
  \end{array} \right] \in \rho_{\ell}(\ggg_{FK})$
we have \[(\alpha a + \beta b)e_1 + (\alpha b) e_2 = Sx = x = a e_1 + b e_2, \]
and thus
\[ (\alpha-1)b = (\alpha-1)a + \beta b = 0. \]
If $\alpha \neq 1$, then $b = 0$ and thus also $a = 0$ -- contradiction -- so $\alpha = 1$ and $\rho_{\ell}(\ggg_{FK})$ consists of elements of the form
$ \left[
  \begin{array}{ c c }
     1 & \beta \\
     0 & 1
  \end{array} \right].$
Hence $\rho_{\ell}(\ggg_{FK})$ has size 1 or $\ell$.
%\\ \\
%Lemma \ref{LEMMA2.6} gives $K \subseteq F(E[\ell])$.
\\
Case 1: Suppose $\# \rho_{\ell}(\ggg_{FK}) = 1$.  By Theorem \ref{SCHERTZTHM} we have $FK \supset K^{(\ell)}$, and the expression for $[K^{(\ell)}:K^{(1)}]$ found in \cite[Corollary 9]{TORS1} gives
\[  [K^{(\ell)}:\Q] = \frac{2(\ell-1)h(K)}{w(K)} \left(\ell-\left(\frac{\Delta_K}{\ell}\right)\right) \mid
[FK: \Q]. \]
This gives the result, in fact with an extra factor of $\ell$ when $\ell \mid \Delta_K$. \\

\noindent Case 2: If $\# \rho_{\ell}(\ggg_{FK}) = \ell$, then by Lemma \ref{LEMMA2.7} and Theorem \ref{SCHERTZTHM} we have the following diagram of fields. Note $K^{(1)} \subset K(j(E))$, the ring class field of $K$ with conductor $\mathfrak{f}$.
\begin{center}
\begin{tikzpicture}[node distance=2cm]
\node (Q)                  {$\mathbb{Q}$};
\node (K) [above of=Q, node distance=1cm] {$K$};
\node (Kj) [above of=K, node distance=1.5cm] {$K^{(1)}$};
\node (Kl) [above of =Kj, node distance=2 cm] {$K^{(\ell)}$};
\node (FK)  [above right of=Kj, node distance=2.5 cm]   {$FK$};
\node (Ftor) [above of =FK, node distance=1.8cm] {$F(E[\ell]))$};

 \draw[-] (Kj) edge node[auto] {} (FK);
 \draw[-] (Q) edge node[left] {2} (K);
 \draw[-] (FK) edge node[right] {$\ell$} (Ftor);
 \draw[-] (Kj) edge node[left] {$\frac{\ell-1}{w(K)}\left(\ell-\left(\frac{\Delta_K}{\ell}\right)\right)$} (Kl);
 \draw[-] (Kj) edge node[left] {$h(K)$} (K);
 \draw (Kl) -- (Ftor);

\end{tikzpicture}
\end{center}

\noindent Thus $\frac{(\ell-1)}{w(K)} \left(\ell-\left(\frac{\Delta_K}{\ell}\right)\right)$ divides $[FK: K^{(1)}] \cdot \ell$.
\begin{enumerate}
\item If $(\frac{\Delta_K}{\ell})=0$ and $w(K)=2$, this implies $\frac{(\ell-1)}{2}\mid [FK: K^{(1)}]$.
\item Suppose $(\frac{\Delta_K}{\ell})=1$. Since $\frac{(\ell-1)^2}{w(K)}$ is prime to $\ell$, it follows that $\frac{(\ell-1)^2}{w(K)} \mid [FK: K^{(1)}]$.
\item Suppose $(\frac{\Delta_K}{\ell})=-1$. Since $\frac{(\ell^2-1)}{w(K)}$ is prime to $\ell$, it follows that $\frac{(\ell-1)^2}{w(K)} \mid [FK: K^{(1)}]$.
\end{enumerate}
Accounting for the additional factor of $2 \cdot h(K)$ obtained from $K^{(1)}\subset FK$ yields the result.\end{proof}

\subsection{Real Cyclotomy I}

\begin{lemma}
\label{CYCLOLEMMA}
\label{4.7}
Let $\Delta$ be an imaginary quadratic discriminant, and let $K = \Q(\sqrt{\Delta})$.  Let $F$ be a number field, and let $E_{/F}$ be an elliptic
curve.  Let $N \geq 3$, and suppose $(\Z/N\Z)^2 \subset E(FK)$.  Then: \\
a) We have $[\Q(\zeta_N):\Q(\zeta_N) \cap F] \leq 2$. \\
b) Suppose $F$ is real.  Then $\Q(\zeta_N) \cap F = \Q(\zeta_N)^+$. \\
c) Suppose $\gcd(N,\Delta_K) = 1$.  Then $\Q(\zeta_N) \subsetneq FK$. \\
%Suppose $F$ is real and $\gcd(N,\Delta_K) = 1$.  Then $\Q(\zeta_N) \subsetneq %FK$ and $\Q(\zeta_N)^+ \subsetneq F$.  \\
d) If $K \not \subseteq F(\zeta_N)$, then $\Q(\zeta_N) \subset F$. \\
e) Suppose $N$ is an odd prime power.  Then $\Q(\zeta_N)^+ \subset F$. \\
f) Suppose $N = 2^a$ with $a \geq 3$.  Then $\Q(\zeta_{N/2})^+ \subsetneq F$.
\end{lemma}
\begin{proof}
a) Let $\chi_N: \mathfrak{g}_F \ra (\Z/N\Z)^{\times}$ be the mod $N$ cyclotomic
character, and let $H_F = \chi_N(\mathfrak{g}_F)$.  As usual the Weil pairing gives $\Q(\zeta_N) \subset FK$, so $\chi_N(\mathfrak{g}_{FK}) \equiv 1$, and thus $\# H_F \leq 2$.  Moreover
$F = F(\zeta_N)^{H_F} \supset \Q(\zeta_N)^{H_F}$, and the result follows.  \\
b) Since $N \geq 3$, $\Q(\zeta_N) \not \subseteq F$ and by part a) we have
$[\Q(\zeta_N):\Q(\zeta_N) \cap F] = 2$.  Further, $\Q(\zeta_N) \cap F \subset \Q(\zeta_N) \cap \R = \Q(\zeta_N)^+$.  \\
c) We have $\Q(\zeta_N) \subset FK$; if equality held, then $K \subset \Q(\zeta_N)$.  But $K$ is ramified at some prime $\ell \nmid N$ and
$\Q(\zeta_N)$ is ramified only at primes dividing $N$. \\
%  Thus
%$1 < [FK:\Q(\zeta_N)] = [F:\Q(\zeta_N)^+]$. \\
d) The hypothesis implies that $FK$ and $F(\zeta_N)$ are linearly disjoint over $F$, $\chi_N|_{\mathfrak{g}_{FK}} \equiv 1$ implies $\# H_F = \{1\}$ and
$\Q(\zeta_N) \subset F$.  \\
e) If $N$ is an odd prime power, then $(\Z/N\Z)^{\times}$ is cyclic, so either
$H_F = 1$ and $\Q(\zeta_N) = \Q(\zeta_N)^{H_F} \subset F$ or $H_F = \{ \pm 1\}$ and $\Q(\zeta_N)^+ = \Q(\zeta_N)^{H_F} \subset F$.  \\
f) Since $N = 2^a$ with $a \geq 3$, $(\Z/N\Z)^{\times}$ has three elements of
order $2$: $-1$ and $2^{a-1} \pm 1$.  So we have $H_F \subsetneq \{\pm 1, 2^{a-1} \pm 1 \}$, and thus \[F \supset \Q(\zeta_N)^{H_F} \supsetneq \Q(\zeta_N)^{ \{ \pm 1, 2^{a-1} \pm 1 \}} = \Q(\zeta_{N/2})^+. \qedhere \]
\end{proof}

\begin{thm}(Real Cyclotomy I)
\label{REALCYCI}
\label{4.8}
Let $\Delta$ be an imaginary quadratic discriminant, and let $K = \Q(\sqrt{\Delta})$.  Let $N \in \Z^+$ be such that $\gcd(N,\Delta) = 1$.
Let $F \not \supseteq K$ be a number field, and let $E_{/F}$ be an $\OO(\Delta)$-CM elliptic curve.  Suppose that $E(F)$ contains a point of
order $N$.  \\
a) We have $(\Z/N\Z)^2 \subset E(FK)$.  \\
b) $F$ contains an index $2$ subfield of $\Q(\zeta_N)$.  \\
c) If $N$ is an odd prime power, then $\Q(\zeta_N)^+ \subsetneq F$.  If
$N \geq 8$ is an even prime power, then $\Q(\zeta_{N/2})^+ \subsetneq F$. \\
d) If $F$ is real and $N \geq 3$, then $\Q(\zeta_N)^+ \subsetneq F$.
\end{thm}
\begin{proof}
a) We immediately reduce to the case that $N = \ell^a$ is a power of a prime $\ell$.  Let $\Ohell = \OO \otimes \Z_{\ell}$, and identify $\Ohell$ with its isomorphic image in $\End T_{\ell} (E)$.  For $b \in \Z^+$,
let $\OO_{\ell^b} = \Ohell/\langle \ell^b \rangle$.  The hypothesis
$\gcd(N,\Delta) = 1$ implies that $\Ohell$ is the maximal $\Z_{\ell}$-order
in $K_{\ell} = K \otimes \Q_{\ell}$.  We know that $T_{\ell}(E)$ is free of rank one as a $\Ohell$-module by Lemma \ref{LEMMA2.6}.  \\
Case 1: Suppose $\left(\frac{\Delta}{\ell}\right) = 1$.  Then $\Ohell \cong \Z_{\ell} \times \Z_{\ell}$.  Put $\iota = \left[ \begin{array}{cc} 0 & 1 \\ 1 & 0 \end{array} \right]$.  Then $N \Ohell^{\times} = \langle \Ohell^{\times}, \iota \rangle$.  Because $F$ does not contain $K$, there is
$\sigma \in \mathfrak{g}_F$ such that $\rho_{\ell^{\infty}}(\sigma) \in N \Ohell^{\times} \setminus \Ohell^{\times}$.  We may choose a $\Z_{\ell}$-basis $\widetilde{e_1}$, $\widetilde{e_2}$ of $T_{\ell}(E)$ and represent the $\Ohell$-action on $T_{\ell}(E)$ via $\{ \left[ \begin{array}{cc} \alpha & 0 \\ 0 & \beta \end{array} \right] \mid \alpha,\beta \in \Z_{\ell} \}$.
Let \[\widetilde{I_1} = \left[ \begin{array}{cc} 1 & 0 \\ 0 & 0 \end{array} \right], \
\widetilde{I_2} = \left[ \begin{array}{cc} 0 & 0 \\ 0 & 1 \end{array} \right] \in \Ohell. \]
For $i = 1,2$, let $\widetilde{V_i} = \langle \widetilde{e_i} \rangle_{\Z_{\ell}}$, and observe that each $\widetilde{V_i}$ is an $\Ohell$-submodule of $T_{\ell}(E)$.  For $i = 1,2$, put $V_i = \widetilde{V_i} \pmod{\ell} \subset E[\ell](\overline{F})$.  Because we may write
$\rho_{\ell^{\infty}}(\sigma)$ as $\iota M$ with $M \in \Ohell^{\times}$, we have
\[ \rho_{\ell^{\infty}}(\sigma)(\widetilde{V_1}) = \widetilde{V_2}, \
\rho_{\ell^{\infty}}(\sigma)(\widetilde{V_2}) = \widetilde{V_1} \]
and thus also
\[\rho_{\ell}(\sigma)(V_1) = V_2, \ \rho_{\ell}(\sigma)(V_2) = V_1. \]
Lift $P \in E[\ell^n](F)$ to a point $\tilde{P} = a \widetilde{e_1} +
b \widetilde{e_2} \in T_{\ell}(E)$.  We claim $a,b \in \Z_{\ell}^{\times}$: if not, $P' = [\ell^{n-1}]P \in V_1 \cup V_2 \setminus \{0\}$, and $\rho_{\ell}(\sigma)(P') = P'$ gives a contradiction.  It follows that for
$i = 1,2$, $\langle \widetilde{I}_i \tilde{P} \rangle_{\Z_{\ell}} = \widetilde{V_i}$, so the $\Ohell$-submodule
generated by $\tilde{P}$ is $T_{\ell}(E)$.  Going modulo $\ell^a$ we get
that the $\OO_{\ell^a}$-submodule generated by $P$ is $E[\ell^a]$, and thus
$(\Z/\ell^a \Z)^2 \subset E(FK)$.
% The result follows in this case by applying
%Lemma \ref{PRELEMMA}.
\\
Case 2: Suppose $\left( \frac{\Delta}{\ell} \right) = -1$.  Then $\Ohell$
is a discrete valuation ring with uniformizing element $\ell$ and fraction field $K_{\ell}$ and thus
$\OO_{\ell^a}$ is a finite principal ring with maximal ideal $\langle \ell \rangle$.  The elements of $(\OO_{\ell^a},+)$ of order $\ell^a$
are precisely the units, so $\OO_{\ell^a}^{\times}$ acts transitively
on the order $\ell^a$ elements of $E[\ell^a]$ and thus the $\OO_{\ell^a}$-submodule of $E[\ell^a]$ generated by $P$ is $E[\ell^a]$. \\
%  We finish as in the previous case.  \\
b) This follows from Lemma \ref{CYCLOLEMMA}a).  \\
c) If $N \geq 3$ is an odd prime power then by Lemma \ref{CYCLOLEMMA}e) we have
$\Q(\zeta_N)^+ \subset F$.  Applying Lemma \ref{CYCLOLEMMA}c) we get
\[ 1 < [FK:\Q(\zeta_N)] = [F:\Q(\zeta_N)^+]. \]
The case of an even prime power $N \geq 8$ is similar but easier, since the strictness in the containment $\Q(\zeta_{N/2})^+ \subsetneq F$ comes from
Lemma \ref{CYCLOLEMMA}f).  For part d) we apply Lemma \ref{CYCLOLEMMA}b) and
deduce the strictness of the containment as above.
\end{proof}

\subsection{Aoki's Theorem}
\textbf{} \\ \\
For a number field $F$ and a prime number $\ell$, we denote by
$w_{\ell^{\infty}}(F)$ the cardinality of the group of $\ell$-power roots of unity in $F$.
\\ \\
Let $K$ be an imaginary quadratic field, and let $F$ be a number field which does not contain $K$.  For an odd prime
$\ell$, put \[m_{\ell}(FK/F) = \begin{cases} 1 & F \supset \Q(\zeta_{\ell}) \\ w_{\ell^{\infty}}(FK) & F \not \supset \Q(\zeta_{\ell}). \end{cases} \]
Suppose $w_{2^{\infty}}(FK) = 2^{\eta}$.  Put
\[ m_2(FK/F) = \begin{cases} 2 & \eta = 1 \text{ or }
F \supset \Q(\sqrt{-1}), \\
2^{\eta} & \eta \geq 2 \text{ and } F \cap \Q(\zeta_{2^{\eta}}) =
\Q( \cos \frac{\pi}{2^{\eta-1}}) \\
2^{\eta-1} & \eta \geq 3 \text{ and } F \cap \Q(\zeta_{2^{\eta}}) = \Q(\sqrt{-1} \sin \frac{\pi}{2^{\eta-1}}). \end{cases} \]
Finally, put
\[m(FK/F) = \prod_{\ell \in \mathcal{P}} m_{\ell}(FK/F). \]

\begin{remark}
a) We have $m(FK/F) \mid w(FK)$.  Equality holds if $F$
is real. \\
b) For $2 < \ell \in \mathcal{P}$, if $m_{\ell}(FK/F) > 1$,
then $FK = F(\zeta_{\ell})$.  \\
c) If $m_2(FK/F) > 2$, then $FK = F(\sqrt{-1})$.
\end{remark}

\begin{thm}(Aoki \cite[Thm. 9.4]{Aoki06})
\label{AOKITHM}
\label{THM4.10}
Let $K$ be an imaginary quadratic field, let $F$ be a number
field which \emph{does not} contain $K$, and let $E_{/F}$ be a $K$-CM elliptic curve.  Put $m = m(FK/F)$ and $n = w(FK)$ (so $m | n$, with equality if $F$ is real).  Then: \\
a) $\exp E(FK)[\tors] \mid n$ and $\# E(FK)[\tors] \mid m^2$. \\ b) $\# E(F)[\tors] \mid m$.  \\
c) $\bullet$ If $\# E(F)[2] \leq 2$, then $E(F)[\tors] \cong \Z/N\Z$ for some $N \in \Z^+$.  \\
$\bullet$ If $\# E(F)[2] = 4$, then $E(F)[\tors] \cong \Z/2\Z \oplus \Z/2N \Z$ for some $N \in \Z^+$.
\end{thm}
\noindent
Aoki records the following immediate consequence.

\begin{cor}
\label{COR4.11}
Let $K$ be an imaginary quadratic field, $F$ a number field
not containing $K$, and $E_{/F}$ a $K$-CM elliptic curve.  If $E(FK)$ has a point of order $N \geq 3$, then $\Q(\zeta_N) \subset FK$ and thus
\[ \frac{\varphi(N)}{2} \mid [F:\Q]. \]
\end{cor}
%  Then
%\[ \varphi( \exp E(FK)[\tors]) \mid 2 [F:\Q]. \]
%If $E(F)$ has a point of order $N$, then
%\[ \varphi(N) \mid 2[F:\Q]. \]
%]\end{cor}
\noindent
Taking $F = \Q$ in Aoki's Theorem and Corollary \ref{COR4.11}, we recover Olson's Theorem.

%  \\
%a) Suppose $E(FK)$ has a point of order $N$.  Then
%\[ \Q(\zeta_N) \subset FK \]
%and thus
%\[ \frac{\varphi(N)}{2} \mid [F:\Q]. \]
%b) Let $w(FK)$ denote the number of roots of unity in $FK$.
% and let \\ \[e = \begin{cases} 0 & w(FK) \equiv 2 \pmod{4} %\\
%\min ([F:\Q], \ord_2 w(FK)) & 4 \mid w(FK). \end{cases}. \]
%Then \[\# E(F)[\tors] \mid w(FK) .\]
%and $E(F)[\tors]$ is isomorphic to a subgroup of $\Z/w(FK)\Z %\oplus \Z/2^e \Z$.
%\end{thm}

%\begin{remark}
%Applying Aoki's Theorem with $F = \Q$, we get $w(FK) = w(K) %\in \{2,4,6\}$, so $\# E(F)[\tors] \in \{1,2,3,6\}$: Olson's %Theorem.
%\end{remark}

\subsection{Real Cyclotomy II}

\begin{thm}(Real Cyclotomy II)
\label{JIM2}
\label{REALCYCTHM}
\label{4.9}
\label{REALCYCII}
Let $\mathcal{O}$ be an order of discriminant $\Delta$ in an imaginary quadratic field $K$,
let $F$ be a real number field, and let $E_{/F}$ be an $\OO$-CM elliptic curve.  Let $N \geq 1$, and suppose $E(F)$ contains a point of
order $N$.   \\
a) We have $\Q(\zeta_N)K_{\Delta} \subset FK$ and $\Q(\zeta_N)^+F_{\Delta} \subset F$.  \\
b) If $\gcd(N,\Delta_K) = 1$ and $N \ge 3$, then $\Q(\zeta_N)^+ \subsetneq F$. \\
c) Suppose $N \geq 3$, and let $\nu = \# \Pic(\OO)[2]$. %(Recall that classical genus theory determines $\nu$ in terms %of the prime factorization of $\Delta$: cf. Lemma %\ref{LEMMA2.5})a)).
Then
\begin{equation}
\label{SWEETCYCEQ1}
\frac{\varphi(N)}{2} \frac{h(\Delta)}{2^{\nu}}  \mid [F:\Q].
\end{equation}
d) If $N \geq 3$ and $h(\Delta)$ is odd -- cf. Lemma \ref{LEMMA2.5}b) -- then
\begin{equation}
\label{SWEETCYCEQ2}
\frac{ \varphi(N)h(\Delta)}{2} \mid [F:\Q].
\end{equation}
\end{thm}

\begin{remark}
\label{4.10}
When $\ell^n= 2$, Theorem \ref{JIM2}a) is vacuous.  Part b) holds in this case unless
$F = \Q$ and $j \in\{0,54000, -15^3,255^3\} $. These curves have CM by $\OO(\Delta)$ for $\Delta\in\{ -3,-12,-7, -28\}$ and a point of order 2.
\end{remark}

%\begin{cor}
%\label{3.14}
%\label{4.11}
%a) Let $N \geq 5$, let $F$ be an odd degree number field, and let $E_{/F}$ be a CM elliptic curve such that $E(F)$ contains a point of order %$N$.  Then there is a prime $p \equiv 3 \pmod{4}$ and a positive integer $a$ such that $N \in \{p^a, 2p^a\}$. \\
%b) Let $\mathcal{S}$ be the set of positive integers $N$ such that
%there is an odd degree number field $F$ and a CM elliptic curve
%$E_{/F}$ such that $E(F)$ contains a point of order $N$.  Then $\mathcal{S}$
%has density $0$.
%\end{cor}
%\begin{proof}
%a) By Real Cyclotomy II (Theorem \ref{REALCYCTHM}) $\Q(\zeta_N)^+ \subset F$,
%so $\frac{\varphi(N)}{2} \mid [F:\Q]$ and thus $4 \nmid \varphi(N)$.  Since $\varphi(N)$ is divisible by $4$ if $N>4$ is divisible by $4$, %by a prime $p \equiv 1 \pmod{4}$, or by two odd primes, the result follows. \\
%b) This follows easily from part a).  In fact the Prime Number Theorem gives the more precise bound
%$\mathcal{S} \cap [1,X] = O\left(\frac{X}{\log X}\right)$.
%\end{proof}

\begin{proof}
a) To establish $\Q(\zeta_N) \subset FK$ we reduce to the case
in which $N = \ell^n$ is a prime power.  It will then follow
that $\Q(\zeta_N)^+ = \Q(\zeta_N)^c \subset (FK)^c = F$.
\\ \indent
Let $\Lambda$ be the real lattice associated to $E$, unique up to $\R$-homothety.  By Lemma \ref{lemma3.6} there is $r \in \mathbb{R}^{\times}$ such that $\Lambda'=r \Lambda$ is a primitive proper $\mathcal{O}$-ideal.

First suppose $\Delta=4D$. Then $\mathcal{O} = [1,\sqrt{D}]$, and $\Lambda'$ is of type I or II as in Theorem \ref{HalterKoch} above. If it is of type I, it follows that
\[
\Lambda = \left[\frac{t}{r}, \frac{\sqrt{D}}{r} \right],\text{ where }t \in \mathbb{Z}^+\text{, } t | D \text{ and }\left(t, \frac{D}{t}\right)=1.
\]
With respect to this basis the action of complex conjugation is given by
\[
 T=\left[ \begin{array}{cc} 1 & 0 \\ 0 & -1 \end{array} \right],
 \]
 and the action of $\mathcal{O}$ on $\Lambda$ is given by
 \[
 \alpha+\beta\sqrt{D} \mapsto \left[ \begin{array}{cc} \alpha & \beta \left(\frac{D}{t}\right) \\ \beta t & \alpha \end{array} \right].
 \]

 We may choose a $\Z_{\ell}$-basis $\widetilde{e_1},\widetilde{e_2}$ for $T_{\ell}(E)$ such that the image of the Cartan subgroup in $\GL_2(\Z_{\ell})$ is
\[ \mathcal{C}_{\ell}^{\times} = \left\{ \left[ \begin{array}{cc} \alpha & \beta \left(\frac{D}{t}\right) \\ \beta t & \alpha \end{array} \right] \mid \alpha^2 -  \beta^2D \in \Z_{\ell}^{\times} \right\} \]
and the element $c \in \mathfrak{g}_F$ induced by complex conjugation corresponds to $\rho_{\ell^{\infty}}(c) =T$. Let $e_i = \tilde{e_i} \pmod{\ell^n}$.  For $ x \in E(F)[\ell^n] \setminus E(F)[\ell^{n-1}]$, we may choose $a,b \in \Z/\ell^n \Z$ such that
$ x = a e_1 + b e_2$.
Then
\[ae_1 + be_2 = x = Tx = ae_1 - b e_2, \]
so $2b \equiv 0 \pmod{\ell^n}$. We assume for the moment that $\ell$ is odd, so it follows that $b \equiv 0 \pmod{\ell^n}$. Thus $a \in (\Z/\ell^n \Z)^{\times}$.

Let $G_{\ell^{\infty}} = \rho_{\ell^{\infty}}(\ggg_F)$ be the image of the $\ell$-adic Galois representation. For $S = \left[ \begin{array}{cc} \alpha & \beta \left(\frac{D}{t}\right) \\ \beta t & \alpha \end{array} \right] \in \mathcal{C}_{\ell}^{\times} \cap G_{\ell^{\infty}}$, we have
\begin{equation}
\label{MatrixMult1}
ae_1 +be_2= x = Sx =\left(\alpha a+\beta b \left(\frac{D}{t}\right) \right)e_1 + (\beta at + \alpha b) e_2.
\end{equation}
Modulo $\ell^n$ this becomes
\[
ae_1= x = Sx =\alpha a e_1 + \beta at e_2.
\]
It follows that  $\alpha \equiv 1 \pmod{\ell^n}$ and $\beta t \equiv 0 \pmod{\ell^n}$, and thus
\[ S \equiv \left[ \begin{array}{cc} 1 &  \beta \left(\frac{D}{t}\right) \\   0 & 1 \end{array} \right] \pmod{\ell^n}. \]  Let $\sigma \in \mathfrak{g}_{FK}$. Then there exists $\beta_0 = \beta_0(\sigma)$ such that
\[ \rho_{\ell^n}(\sigma) = \left[ \begin{array}{cc} 1 &  \beta_0 \left(\frac{D}{t}\right) \\ 0 & 1 \end{array} \right]. \]
By Galois equivariance of the Weil pairing, $\sigma \zeta_{\ell^n} = \zeta_{\ell^n}^{\det \rho_{\ell^n}(\sigma)} = \zeta_{\ell^n}$,
so $\zeta_{\ell^n} \in FK$.
% On the other hand, $c \zeta_{\ell^n} =
%\zeta_{\ell^n}^{\det (T)} = \zeta_{\ell^n}^{-1}$, so $\Q(\zeta_{\ell^n})^+ %\subset
%(FK)^c = F$.

If $\ell=2$, we must adjust our approach by working mod $2^{n-1}$. Indeed, $Tx=x$ will only imply $b \equiv 0 \pmod{2^{n-1}}$ and $a \in (\Z/2^n \Z)^{\times}$.  For $S  \in \mathcal{C}_{2}^{\times} \cap G_{2^{\infty}}$, $Sx=x$ gives $\alpha \equiv 1 \pmod{2^{n-1}}$, $\beta t \equiv 0 \pmod{2^{n-1}}$. In fact, $\beta t \equiv 0 \pmod{2^{n}}$ as well.  Indeed, by (\ref{MatrixMult1}), $b=\beta at + \alpha b$, which means
\[a^{-1}b(1-\alpha) \equiv \beta t \pmod{2^n}.\]
As $2^{n-1} \mid b$ and $2^{n-1} \mid (1-\alpha)$, the claim follows since $n \geq 2$. Thus
\[
\det S = \alpha^2 - \beta^2 D =  \alpha^2 - \beta^2 t \frac{D}{t} \equiv  \alpha^2 \equiv 1 \pmod{2^n},
\]
and $\det \circ \rho_{2^n} |_{\ggg_{FK}}$ is trivial. We conclude $\zeta_{2^n} \in FK$.
%The containment $\Q(\zeta_{\ell^n})^+ \subset F$ follows as above.

If $\Lambda'$ is of type II, then
\[
\Lambda=\left[\frac{t}{r}, \frac{t+\sqrt{\Delta}}{2r} \right],\text{ where }t \in \mathbb{Z}^+\text{, }4t | t^2-\Delta \text{ and } \left(t, \frac{t^2-\Delta}{4t}\right)=1.
\]
With respect to this basis the action of complex conjugation is given by
\[
 T=\left[ \begin{array}{cc} 1 & 1 \\ 0 & -1 \end{array} \right],
 \]
 and the action of $\mathcal{O}$ on $\Lambda$ is given by
 \[
 \alpha+\beta\sqrt{D} \mapsto \left[ \begin{array}{cc} \alpha-\beta\left(\frac{t}{2}\right) & -\beta \left(\frac{t^2-\Delta}{4t}\right) \\ \beta t & \alpha + \beta \left(\frac{t}{2}\right) \end{array} \right].
 \]
 This gives rise to the Cartan subgroup
 \[
  \mathcal{C}_{\ell}^{\times} = \left\{ \left[ \begin{array}{cc} \alpha-\beta\left(\frac{t}{2}\right) & -\beta \left(\frac{t^2-\Delta}{4t}\right) \\ \beta t & \alpha + \beta \left(\frac{t}{2}\right) \end{array} \right] \mid \alpha^2 -  \beta^2\frac{t^2}{4}+\beta^2\left(\frac{t^2-\Delta}{4}\right) \in \Z_{\ell}^{\times} \right\}
 \]
As before, we may choose a $\Z_{\ell}$-basis $\widetilde{e_1},\widetilde{e_2}$ for $T_{\ell}(E)$ such that $\rho_{\ell^{\infty}}(\ggg_{FK}) \subset \mathcal{C}_{\ell}^{\times}$ and the element $c \in \mathfrak{g}_F$ induced by complex conjugation corresponds to $\rho_{\ell^{\infty}}(c) =T$.

Let $e_i = \tilde{e_i} \pmod{\ell^n}$.  For $ x \in E(F)[\ell^n] \setminus E(F)[\ell^{n-1}]$, we may choose $a,b \in \Z/\ell^n \Z$ such that
$ x = a e_1 + b e_2$.
Then
\[ae_1 + be_2 = x = Tx = (a+b)e_1 - b e_2, \]
so $b \equiv 0 \pmod{\ell^n}$ and $a \in (\Z/\ell^n \Z)^{\times}$. For $S = \left[ \begin{array}{cc} \alpha-\beta\left(\frac{t}{2}\right) & -\beta \left(\frac{t^2-\Delta}{4t}\right) \\ \beta t & \alpha + \beta \left(\frac{t}{2}\right) \end{array} \right] \in \mathcal{C}_{\ell}^{\times} \cap G_{\ell^{\infty}}$, we have
\[
ae_1 = x = Sx =\left(\alpha-\beta\left(\frac{t}{2}\right)\right) a e_1 + \beta a t e_2 \pmod{\ell^n}.
\]
Thus $\alpha - \beta\left(\frac{t}{2}\right) \equiv 1 \pmod{\ell^n}$ and $\beta t \equiv 0 \pmod{\ell^n}$. It follows that
\[
\alpha + \beta \left(\frac{t}{2}\right) \equiv \alpha+ \beta \left(\frac{t}{2}\right)-\beta t = \alpha - \beta \left( \frac{t}{2} \right) \equiv 1 \pmod {\ell^n}.
\]
Hence
\[ S \equiv \left[ \begin{array}{cc} 1 & -\beta \left(\frac{t^2-\Delta}{4t}\right)\\   0 & 1 \end{array} \right] \pmod{\ell^n}. \]
We conclude $\zeta_{\ell^n} \in FK$ as before.

Finally, we consider the case when $\Delta \equiv 1 \pmod{4}$. Then $\mathcal{O} =\left[1,\frac{1+\sqrt{\Delta}}{2}\right]$ and $\Lambda'$ is of type II as in Theorem \ref{HalterKoch}. Thus
\[
\Lambda=\left[\frac{t}{r}, \frac{t+\sqrt{\Delta}}{2r} \right],\text{ where }t \in \mathbb{Z}^+\text{, }4t | t^2-\Delta \text{ and } \left(t, \frac{t^2-\Delta}{4t}\right)=1.
\]
Following the method used above, we may choose a $\Z_{\ell}$-basis $\widetilde{e_1},\widetilde{e_2}$ for $T_{\ell}(E)$ such that the image of the Cartan subgroup in $\GL_2(\Z_{\ell})$ consists of matrices of the form
\[ \left[ \begin{array}{cc} \alpha-\beta \left(\frac{t-1}{2}\right) & \beta \left(\frac{\Delta-t^2}{4t}\right) \\ \beta t &\alpha + \beta\left(\frac{t+1}{2}\right) \end{array} \right]  \]
and the element $c \in \mathfrak{g}_F$ induced by complex conjugation corresponds to
\[
 T=\left[ \begin{array}{cc} 1 & 1 \\ 0 & -1 \end{array} \right].
\]

Let $e_i = \tilde{e_i} \pmod{\ell^n}$.  For $ x \in E(F)[\ell^n] \setminus E(F)[\ell^{n-1}]$, we again choose $a,b \in \Z/\ell^n \Z$ such that
$ x = a e_1 + b e_2$.
Then $Tx=x$ gives $b \equiv 0 \pmod{\ell^n}$ and hence $a \in (\Z/\ell^n \Z)^{\times}$. For $S = \left[ \begin{array}{cc} \alpha-\beta \left(\frac{t-1}{2}\right) & \beta \left(\frac{\Delta-t^2}{4t}\right) \\ \beta t & \alpha + \beta\left(\frac{t+1}{2}\right) \end{array} \right] \in \mathcal{C}_{\ell}^{\times} \cap G_{\ell^{\infty}}$, we have
\[
ae_1 = x = Sx =\left(\alpha-\beta\left(\frac{t-1}{2}\right)\right) a e_1 + \beta a t e_2 \pmod{\ell^n}.
\]
As before, this implies $\alpha - \beta \left( \frac{t-1}{2} \right) \equiv 1 \pmod{\ell^n}$ and $\beta t \equiv 0 \pmod{\ell^n}$. Since
\[
\alpha + \beta\left(\frac{t+1}{2}\right) \equiv \alpha + \beta\left(\frac{t+1}{2}\right) - \beta t = \alpha - \beta \left( \frac{t-1}{2} \right) \equiv 1 \pmod{\ell^n},
\]
we have
\[ S \equiv \left[ \begin{array}{cc} 1 &  \beta \left(\frac{\Delta-t^2}{4t}\right) \\   0 & 1 \end{array} \right] \pmod{\ell^n}. \]
It follows that $\zeta_{\ell^n} \in FK$. \\
b) Suppose $\gcd(N,\Delta_K) = 1$ and $N \geq 3$.  Seeking a contradiction, we suppose
that $F = \Q(\zeta_N)^+$.  Then $K \subset FK = \Q(\zeta_N)$, so $K$ is ramified
at some prime divisor of $N$ and thus $\gcd(N,\Delta_K) > 1$. \\
c) Let $G(\Delta) := K_{\Delta} \cap \Q^{\ab}$ be the maximal abelian subextension of $K_{\Delta}/\Q$.  We have $\Aut(K_{\Delta}/K) = \Pic \OO(\Delta)$, and by
\cite[Thm. 9.18]{Cox89}, if $M$ is a subextension of $K_{\Delta}/K$ then $M/\Q$ is generalized dihedral.  It follows that
\[ \Aut(G(\Delta)/K) = \Pic \OO(\Delta) / \Pic(\OO(\Delta))^2 \cong (\Z/2\Z)^{\nu} \]
and
\[ \Aut(G(\Delta)/\Q) \cong (\Z/2\Z)^{\nu + 1}. \]
Since $K \subset G(\Delta)$, we have
\[\Aut(G(\Delta) \cap F_{\Delta}/\Q) = \Aut(\Q^{\ab} \cap F_{\Delta} / \Q)
\cong (\Z/2\Z)^{\nu} \]
and thus
\[ \Aut(\Q(\zeta_N)^+ \cap F_{\Delta}/\Q) = (\Z/2\Z)^{\xi} \]
for some $0 \leq \xi \leq \nu$.  Equation (\ref{SWEETCYCEQ1}) follows immediately.  \\
d) This is an immediate consequence of part c).
\end{proof}

\begin{remark}
\label{4.14}
a) If $E$ has an $FK$-rational point of order $N$, then some quadratic twist $(E^t)_{/F}$ has an $F$-rational point of
order $N^* = \frac{N}{\gcd(N,2)}$ and thus Real Cyclotomy II implies $\frac{\varphi(N^*)}{2} \mid [F:\Q]$. \\
b) Combining Aoki's Theorem with the proof of parts c) and d) of Real Cyclotomy II, we find that (\ref{SWEETCYCEQ1}) and (\ref{SWEETCYCEQ2}) hold under the weaker assumptions that
$F \not \supset K$ and $E(FK)$ has a point of order $N$.
\end{remark}

\subsection{Parish's Theorem}
\textbf{} \\ \\ \noindent
Above we saw that Aoki's Theorem implies a generalization of Olson's Theorem.  A different such generalization
was given by J.L. Parish \cite{Parish89}.

\begin{thm}(Parish)
\label{PARISHTHM}
Let $F$ be a number field, and let $E_{/F}$ be an elliptic curve with CM by the imaginary quadratic order $\OO$, with
discriminant $\Delta$, in the imaginary quadratic field $K$. \\
a) If $F = \Q(j(E))$, then $E(F)[\tors]$ is an Olson group. \\
b) If $F = K(j(E))$ and $\Delta < -4$, then $E(F)[\tors]$ is either an Olson group or isomorphic to $\Z/2\Z \oplus \Z/4\Z$
or $\Z/2\Z \oplus \Z/6\Z$.
\end{thm}

\begin{remark}
The table of Theorem \ref{MAINTHM} shows that the CM discriminants $\Delta = -3$ and $\Delta = -4$ are indeed exceptions
to the statement of Theorem \ref{PARISHTHM}b).
\end{remark}
\noindent
Aoki's Theorem quickly implies a slightly weaker version of Theorem \ref{PARISHTHM}a).
 %which is sufficient for our applications.

\begin{cor}(Parish In Odd Degree)
\label{ODDDEGREEPARISH}
 \\
a) Let $E_{/F}$ be a $K$-CM elliptic curve with $F = \Q(j(E))$.  Then the exponent of $E(F)[\tors]$ divides $24$.  \\
b) If $[F:\Q]$ is odd, then $E(F)[\tors]$ is an Olson group.
\end{cor}
\begin{proof}
a) Suppose $E(F)$ contains a point of order $N$.  Since $F$ does not contain $K$, Aoki's Theorem applies: $\Q(\zeta_N) \subset FK = K(j)$.
The maximal abelian subextension of $K(j)/\Q$ has exponent dividing $2$ (cf. proof of Theorem \ref{REALCYCII}, part c).  Thus $(\Z/N\Z)^{\times}$ itself has
exponent dividing $2$, which holds (if and) only if $N \mid 24$.  \\
b) We have $\varphi(N) \mid [FK:\Q]$.  If $[F:\Q]$ is odd, this eliminates $N = 8, 12, 24$.  Since $F$ is real, $E(F)[\tors]$
is either of the form $\Z/N\Z$ or $\Z2\Z \oplus \Z/N\Z$ with $N$ even.  By Corollary \ref{COR4.3}, since $[F:\Q]$ is odd, the latter can only
happen when $\Delta = -4$, in which case $F = \Q(j) = \Q(1728) = \Q$, and we have reduced to Olson's Theorem.
\end{proof}

%\begin{remark}
%Parish's proof of Theorem \ref{PARISHTHM} is somewhat intricate.  %Along the way he establishes other results of interest, %especially: if $\OO$ is an imaginary quadratic order
%of discriminant $\Delta <  -4$ and conductor $F$ and $E_{/K(j)}$ %is an $\OO$-CM elliptic curve and $N \in \Z^+$ then the fixed
%field of the kernel of the \emph{projective} Galois %representation $\mathbb{P} \rho_N: \mathfrak{g}_{K(j)} \ra %\operatorname{PGL}_2(\Z/N\Z)$ is the $(N\mathfrak{f})$-ring %class field of $K$.  S. Kwon has used this result to give a %different, simpler proof of Theorem \ref{PARISHTHM}b) %\cite{Kwon99b}.
%%(Kwon's argument can be slightly shortened by making use %of %Theorem \ref{SCHERTZTHM}.)
%It is interesting to compare
%Kwon's argument to the proof of Corollary \ref{ODDDEGREEPARISH}: %the latter argument naturally yields $\{N \in \Z^+ \mid %\lambda(N) \leq 2\}$, where $\lambda(N) = \exp %(\Z/N\Z)^{\times}$ is Carmichael's function; the former argument %yields $\{N \in \Z^+ \mid \varphi(N) \leq 2\}$.
%\end{remark}

\subsection{Square-Root SPY Bounds}

\begin{prop} \label{hmgen} \label{4.12} (\cite[Corollary 3.2.4]{Cohen2})
Let $K$ be a number field, and let $\mathfrak{m}$ a nonzero ideal
of $\mathcal{O}_K$, hence also a modulus in the sense of class field theory.  Let $U = \Ok^{\times}$ and $U_{\mathfrak{m}} =  \{ \alpha \in U :  \ord_{\mathfrak{p}}(\alpha -1) \geq \ord_{\mathfrak{p}} \mm \text{ for all } \mathfrak{p} \mid \mathfrak{m} \}$.  Then
\begin{center}
$[K_{\mathfrak{m}}:K] = \frac{h(K)}{[U:U_{\mathfrak{m}}]} [\OO_K:\mm] \cdot \prod_{\mathfrak{p} \mid \mathfrak{m}} \left( 1- [\OO_K:\pp]^{-1} \right)$,
\end{center}
where $K_{\mathfrak{m}}$ is the ray class field of $K$ with modulus $\mathfrak{m}$.
%where
%\begin{align*}
%h_K &= \text{class number of } K\\
%U &= \Ok^{\times}\\
%U_{\mathfrak{m}} &= \{ \alpha \in U :  \ord_{\mathfrak{p}}(\alpha -1) \geq %m(\mathfrak{p}) \text{ for all } \mathfrak{p} \text{ dividing } \mathfrak{m} \}.
%\end{align*}
\end{prop}

\begin{thm}(Square-Root SPY Bounds)
\label{SQRTSPY} \label{4.13}
Let $\OO$ be an imaginary quadratic order of discriminant $\Delta$ and fraction field $K$, let $F$ be a number field, and let $E_{/F}$ be an $\OO$-CM elliptic curve.
Let $N \geq 3$, and suppose $E(F)$ has a point of order $N$.  \\
a) If $(\Z/N\Z)^2 \subset E(FK)$, then
\[ \varphi(N) \leq \sqrt{\frac{[F:\Q]w(K)}{h(K)}}. \]
b) The hypothesis of part a) is satisfied when $K \not \subseteq F$ and $\gcd(\Delta,N) = 1$.
\end{thm}
\begin{proof}
a) We have
\[ \prod_{\pp \mid N \OO_K} \left(1-[\OO_K:\pp]^{-1} \right) = \prod_{p \mid N}
\prod_{\pp \mid p \OO_K} \left(1-[\OO_K:\pp]^{-1} \right). \]
Further, we have
\[ \prod_{\pp \mid p\OO_K} \left(1-[\OO_K:\pp]^{-1} \right) = \begin{cases}
(1-\frac{1}{p})^2, & \left(\frac{\Delta_K}{p} \right) = 1 \\
(1-\frac{1}{p}), & \left(\frac{\Delta_K}{p} \right) = 0 \\
(1-\frac{1}{p^2}), & \left(\frac{\Delta_K}{p} \right) = -1
\end{cases}, \]
so
\[ \prod_{\pp \mid N \OO_K} \left(1-[\OO_K:\pp]^{-1} \right) \geq
\prod_{p \mid N} \left(1-\frac{1}{p}\right)^2. \]
Applying Lemma \ref{LEMMA2.7} and Theorem \ref{SCHERTZTHM}, we get
\[ FK = F(E[N]) \supset K^{(N)} \]
and thus
\[ [F:\Q] \geq \frac{ [FK:\Q]}{2} \geq \frac{[K^{(N)}:\Q]}{2} = [K^{(N)}:K] \]
\[ = \frac{h(K)}{[U:U_{N\OO_K}]} [\OO_K:N\OO_K] \prod_{\pp \mid N \OO_K} \left(1-[\OO_K:\pp]^{-1} \right) \]
\[ \geq \frac{h(K)}{w(K)} N^2 \prod_{p \mid N} \left(1-\frac{1}{p}\right)^2 =
\frac{h(K)}{w(K)} \varphi(N)^2.  \]
b) This is Theorem \ref{REALCYCI}a).
\end{proof}

\begin{example}
\label{EXAMPLE3.16}
%\label{4.14}
Let $N \geq 3$, let $F$ be a cubic number field, and let $E_{/F}$ be a CM elliptic curve such that $E(F)$ contains a point of order $N$.  The SPY Bounds give
$\varphi(N) \leq [F:\Q] w(K) \leq 18$; in particular the largest prime value of $N$ permitted is $19$.  The odd
order torsion subgroup of $E(F)$ has size at most $13$ \cite[Main Theorem]{Clark-Xarles08}; in particular, the largest
prime value permitted is $13$.  Real Cyclotomy II (Theorem \ref{REALCYCTHM}) gives
$\varphi(N) \mid 6$ and thus $N \mid 4$, $N \mid 14$ or $N \mid 18$; in particular,
the largest prime value permitted is $7$.  Theorems \ref{IMPRIMITIVELEMMA} and \ref{MAINTHM}
show that all of these values of $N$ are actually attained except for $N = 18$.  \\ \indent
Suppose now that $E_{/F}$ is $\OO(\Delta)$-CM and $\gcd(\Delta,N) = 1$.  Then the Square Root SPY Bounds give
\[\varphi(N) \leq \lfloor \sqrt{18} \rfloor = 4, \ K = \Q(\sqrt{-3}), \]
\[ \varphi(N) \leq \lfloor \sqrt{12} \rfloor = 3, \ K = \Q(\sqrt{-1}), \]
\[ \varphi(N) \leq \lfloor \sqrt{6} \rfloor = 2, \ K \notin \{ \Q(\sqrt{-3}), \Q(\sqrt{-1}) \}. \]
Combining with Real Cyclotomy II, we get that the only prime values of $N$ permitted in this case
are the ``Olson primes'' $2$ and $3$.  The four elliptic curves in rows 13 through 16 of the table
in Theorem \ref{MAINTHM} do not satisfy the Square Root SPY Bounds.  It follows from Theorem \ref{SQRTSPY} that for $N = 9$ and $N = 14$ we \emph{do not have} $(\Z/N\Z)^2 \subset E(FK)$.  This shows that the hypothesis $\gcd(\Delta,N) = 1$ in Real Cyclotomy I (Theorem \ref{REALCYCI}) is necessary in order for this stronger form of real cyclotomy to hold.  On the other hand, we have
\[ \Q[b]/(b^3-15b^2-9b-1) \cong \Q[b]/(b^3+105b^2-33b-1) \cong \Q(\zeta_9)^+, \]
\[ \Q[b](b^3-4b^2+3b+1) \cong \Q[b]/(b^3-186b^2+3b+1) \cong \Q(\zeta_{14})^+, \]
in accordance with Real Cyclotomy II (Theorem \ref{REALCYCTHM}).
\end{example}

\section{The Odd Degree Theorem}

\subsection{Statement of Theorem}
\textbf{} \\ \\ \noindent
Let $F/\Q$ be a number field of odd degree.  Then $F$ is real.  Hence we get a natural terrain of applicability of Real Cyclotomy I,
Aoki's Theorem and Real Cyclotomy II: if $F$ is an odd degree number field, $E_{/F}$ is a CM elliptic curve, and $E(F)$ has a point
of order $N \geq 5$ (as we may as well assume, since points of order $1 \leq N \leq 4$ occur already over $\Q$), then Real Cyclotomy II implies that $\frac{\varphi(N)}{2} \mid [F:\Q]$.  But this implies $\varphi(N) \equiv 2 \pmod{4}$, which is very restrictive!  Indeed, $\varphi(N)$ is always even; moreover $\varphi(N)$ is divisible by $4$ when $N \geq 5$ is divisible by $4$, by a prime $p \equiv 1 \pmod{4}$, or by two distinct odd primes.  So there is a prime $p \equiv 3 \pmod{4}$ and natural numbers $\epsilon,a$ with $0 \leq \epsilon \leq 1$ such that $N = 2^{\epsilon} p^a$.
\\ \\
The odd degree hypothesis also imposes the following restriction.

\begin{thm}
\label{JIMSLEMMA}
Let $\OO$ be an order in an imaginary quadratic field $K$, of discriminant $\Delta$.
Let $F$ be an odd degree number field and $E_{/F}$ an $\OO$-CM elliptic curve. \\
a) (Aoki) If $\Z/\ell\Z \hookrightarrow E(F)$ for some prime $\ell > 2$, then
$\ell \equiv 3 \pmod{4}$ and $K =  \Q(\sqrt{ -\ell})$. \\
b) If $\Z/2Z \oplus \Z/2\Z \hookrightarrow E(F)$, then: \\
(i) (Aoki) $E(F)[\tors] \cong \Z/2\Z \oplus \Z/2\Z$ and $K = \Q(\sqrt{-1})$. \\
(ii) $\Delta = -4$. \\
c) If $\Z/4\Z \hookrightarrow E(F)$, then: \\
(i) (Aoki) $E(F)[\tors] \cong \Z/4\Z$ and $K  = \Q(\sqrt{-1})$. \\
(ii) $\Delta \in \{-4,-16\}$.
\end{thm}
\begin{proof}
a) Real Cyclotomy II gives $\frac{\ell-1}{2} \mid [F:\Q$]; since $[F:\Q]$ is odd, we
 have $\ell \equiv 3 \pmod{4}$.  It also gives $\zeta_\ell \in FK$, so $FK$ contains both $K$ and $\Q(\sqrt{ (-1)^{\frac{\ell-1}{2}}\ell})$.  Since
$[FK:\Q] \equiv 2 \pmod{4}$, we have $K = \Q(\sqrt{ (-1)^{\frac{\ell-1}{2}}\ell})= \Q(\sqrt{ -\ell})$.
\\
b) Corollary \ref{COR4.3} gives $\Delta = -4$, so $K = \Q(\sqrt{-1})$.
By part a, $E(F)[\tors]$ is a $2$-group, and by Corollary \ref{4.5} we have $E(F)[\tors] \cong \Z/2\Z \times \Z/2\Z$.
\\
c) By Real Cyclotomy II, $FK$ contains both $K$ and $\Q(\sqrt{-1})$.  Again, $[FK:\Q] \equiv 2 \pmod{4}$ forces $K = \Q(\sqrt{-1})$.  By parts a) and b), $E(F)[\tors] \cong \Z/2^a\Z$ for some $a \geq 2$.  By Real Cyclotomy II, $2^{a-2} \frac{\varphi(2^a)}{2} \mid [F:\Q]$, so $a = 2$.  The order of conductor $\mathfrak{f}$ in $\Q(\sqrt{-1})$ has odd class number iff $\mathfrak{f} \in \{1,2\}$.
\end{proof}

\begin{remark} \textbf{} \\
It follows from Corollary \ref{7.5} that if $\ell \equiv 3 \pmod{4}$ is a prime, then for all $a \geq 0$, there is an odd degree number field $F$ and an elliptic curve $E_{/F}$ with CM by the quadratic order of discriminant $-\ell^{2a+1}$ and an
$F$-rational point of order $\ell$.  \\ \indent
Olson's work shows that there are elliptic curves over $\Q$ with CM by the quadratic
orders of discriminants $-4$ and $-16$ with $\Q$-rational points of order $4$.
\end{remark}
\noindent
Since $F$ is real, $E(F)[\tors]$ is of the form $\Z/N\Z$ or $\Z/2\Z \oplus \Z/2N \Z$.  Thus the discussion above and Theorem \ref{JIMSLEMMA} imply $E(F)[\tors]$ has order dividing $4$ or is of the form $\Z/2^{\epsilon} \ell^n\Z$ with $\epsilon \in \{0,1\}$, $n \in \Z^+$
and $\ell \equiv 3 \pmod{4}$ a prime.  We have recovered a result of Aoki \cite[Cor. 9.4]{Aoki95}, via an approach which gives slightly more precise information in the case $4 \mid \# E(F)[\tors]$.  This brings us to the main result of this section, which promotes Aoki's Theorem to a complete determination of the torsion subgroups of CM elliptic curves over odd degree number fields.

\begin{thm}(Odd Degree Theorem)
\label{ODDDEGREETHM}
Let $F$ be a number field of odd degree, let $E_{/F}$ be a $K$-CM elliptic curve, and let $T = E(F)[\tors]$.  Then: \\
 a) One of the following occurs:
\begin{enumerate}
\item $T$ is isomorphic to the trivial group $\{ \bullet\}, \, \Z/2\Z, \, \Z/4\Z,$ or $\Z/2\Z \times \Z/2\Z$;
\item $T \cong \Z/\ell^n \Z$ for a prime $\ell \equiv 3 \pmod{8}$ and $n \in \mathbb{Z}^+$ and
$K = \Q(\sqrt{-\ell})$;
\item $T \cong \Z/2\ell^n \Z$ for a prime $\ell \equiv 3 \pmod{4}$ and $n \in \mathbb{Z}^+$ and
$K = \Q(\sqrt{-\ell})$.
\end{enumerate}
b) If $E(F)[\tors] \cong \Z/2\Z \oplus \Z/2\Z$, then $\End E$ has discriminant $\Delta = -4$.  \\
c) If $E(F)[\tors] \cong \Z/4\Z$, then $\End E$ has discriminant $\Delta \in \{ -4,-16\}$. \\
d) Each of the groups listed in part a) arises up to isomorphism as the torsion subgroup $E(F)$ of a CM elliptic curve $E$ defined over an odd degree number field $F$.
\end{thm}
\noindent
We have already established most of parts a) through c).  What remains in those parts is to show that $\Z/\ell^n\Z$ cannot occur as a torsion subgroup for $\ell \equiv 7 \pmod{8}$, which we handle next.  The rest of the section is devoted to the proof of part d).

\subsection{Primes Equivalent to $7 \pmod{8}$}
%\textbf{} \\ \\ \noindent

\begin{lemma}
Let $F$ be an odd degree number field and $E_{/F}$ a CM elliptic curve. If $\Z/\ell\Z \hookrightarrow E(F)$ for some prime $\ell  \equiv 7 \pmod{8}$, then $E(F)$ has a point of order 2.
\end{lemma}
\begin{proof}
Let $E$ be a $K$-CM elliptic curve with $\Z/\ell\Z \hookrightarrow E(F)$ for some prime $\ell  \equiv 7 \pmod{8}$. By Theorem \ref{JIMSLEMMA} we have $K = \Q(\sqrt{-\ell})$, so the CM discriminant $\Delta$ is either even or is $1$ modulo $8$.  Then by Theorem \ref{THM4.2}, every $\OO(\Delta)$-CM elliptic curve defined over $\Q(j(\OO))$ has an $F$-rational point of order $2$.  Since $\Delta \notin \{-3,-4\}$ and quadratic twists preserve the $2$-torsion subgroup, every $\OO(\Delta)$-CM elliptic curve defined over a field of characteristic $0$ has a rational point of order $2$.
\end{proof}

\subsection{Twisting, Top to Bottom}
\textbf{} \\ \\ \noindent
In order to prove part d) of Theorem \ref{ODDDEGREETHM}, we need some preliminary results on twisting, not all of which are inherently concerned with elliptic curves, complex multiplication or odd degree number fields(!).  With an eye to future use we work a bit more generally.  (It costs nothing to do so.)
\textbf{} \\ \\ \noindent
For $N \in \Z^+$, let $U(N) = (\Z/N\Z)^{\times}$ and
let $U(N)^{\pm} = U(N)/(\pm 1)$.

\begin{thm}(Twisting at the Top)
\label{7.2}
Let $N \geq 3$ be an integer.  Let $F$ be a field, let $A_{/F}$ be an abelian
variety, and let $C \subset_F A$ be an \'etale subgroup scheme which is cyclic
of order $N$.
%(In plainer terms: $C$ is a cyclic, order $N$ subgroup of $A(F^{\sep})$ which %is stable under the action of $\ggg_F$.)
Then there is
an abelian extension $L/F$ with $[L:F] \mid \frac{\varphi(N)}{2}$ and a quadratic twist $A'$ of $A_{/L}$ such that $A'(L)$ has a point of order $N$.
\end{thm}
\begin{proof}
The action of $\ggg_F$ on $C$ gives rise to an isogeny character
\[ \Phi: \ggg_F \ra U(N). \]
Composing with $U(N) \ra U(N)^{\pm}$, we get the \textbf{plus-minus isogeny character}
\[ \Phi^{\pm}: \ggg_F \ra U(N)^{\pm}. \]
Let $M = (F^{\sep})^{\Ker \Phi}$.  Then $M/F$ is abelian, $[M:F] \mid \varphi(N)$
and $\Phi|_{\ggg_M}$ is trivial, so any generator of $C$ gives a point of order
$N$ on $A(M)$.  We can however gain back the last factor of $2$ by \emph{twisting at the top}. Let $L = (F^{\sep})^{\Ker \Phi^{\pm}}$.  Then $L/F$ is abelian, $[L:F] \mid \frac{\varphi(N)}{2}$, and
\[ \Phi(\ggg_L) \subset \{ \pm 1\}. \]
If $\Phi|_{\ggg_L} \equiv 1$, we're done: a generator of $C$ gives an $L$-rational point on $A$.  Otherwise $\Phi|_{\ggg_L} = \epsilon$ is a nontrivial
quadratic character on $\ggg_L$, with corresponding separable quadratic extension $M/L$.  Let $A'_{/L}$
be the quadratic twist of $A$ by $M/L$.  The twisted $\ggg_L$-action on $C$
is by $\Phi \cdot \epsilon = \epsilon \cdot \epsilon = 1$.  So $A'(L)$
has a point of order $N$.
\end{proof}

%\begin{cor}(Composite Split Case)
%Let $\OO$ be an imaginary quadratic order, and let $N \in \Z^+$ satisfy the
%\textbf{Heegner hypothesis}: $\gcd(N,\Delta) = 1$ and every prime $p$ dividing
%$N$ splits in $K$.  Then there is a number field $F$ of degree $h(\OO) \cdot %\varphi(N)$ and an $\OO$-CM elliptic curve $E_{/F}$ with an $F$-rational point %of order $N$.
%\end{cor}
%\begin{proof}
%Let $F = K(j(\OO))$, and let $E_{/F}$ be an
%$\OO$-CM elliptic curve.  By the Heegner hypothesis, there is an ideal $I_N %\subset \OO$ such that $\OO/I_N \cong \Z/N\Z$, and thus $E$ has an $F$-rational %cyclic subgroup of order $N$.  Applying Theorem \ref{THM2}, we get
%a point of order $N$ over a twist of $E$ defined over a number field of degree
%dividing $2 h(\OO) \cdot \frac{\varphi(N)}{2}$.
%\end{proof}
%\begin{remark}
%When $N$ is prime, Theorem \cite[Thm. 3a)]{TORS1} gives this result, in fact %with an extra factor of $\frac{2}{\# \OO^{\times}}$ (which is an improvement %when
%$\Delta \in \{-4,-3\}$).  I haven't looked carefully to see whether the method %of proof will give the extra factor in the general case, but it seems likely
%that it would whenever $N$ is prime to $6$.
%\end{remark}

\begin{thm}(Twisting at the Bottom)
\label{THM3}
Let $\ell \equiv 3 \pmod{4}$ be a prime, $n \in \Z^+$, and put $N = \ell^n$.  Let $F$ be a field, $A_{/F}$ an abelian variety, and
$C \subset_F A$ an \'etale subgroup scheme which is cyclic of order $N$.  There is a quadratic
twist $A'$ of $A_{/F}$ and a cyclic extension $L/F$ of degree dividing
$\frac{\varphi(N)}{2}$ such that $A'(L)$ has a point of order $N$.
\end{thm}
\begin{proof}
Let $\Phi: \ggg_F \ra U(N)$ be the isogeny character associated to $C$.  Let
$q: U(N) \ra U(N)/U(N)^2$ be the quotient map.  Since $N = \ell^n$ is an odd
prime power, $U(N)/U(N)^2$ has order $2$ and is thus (uniquely!) isomorphic
to $\{ \pm 1 \}$.  Moreover, since $p \equiv 3 \pmod{4}$, $-1$ is not a
square modulo $p$, so \emph{a fortiori} is not a square modulo $N$.  Thus
under the canonical isomorphism $U(N)/U(N)^2 \ra \{ \pm 1\}$, the
class of $-1$ maps to $-1$.  Let
\[ \epsilon_{\Phi} = q \circ \Phi: \ggg_F \ra U(N)/U(N)^2 = \{ \pm 1\}. \]
Let $A'$ be the quadratic twist of $A$ by $\epsilon_{\Phi}$
(so $A' = A$ iff $\epsilon_{\Phi}$ is trivial), and let
\[ \Phi' = \epsilon_{\Phi} \Phi \]
be the associated isogeny character.  By the above remarks, about $-1 \in U(N)$
mapping to the nontrivial element of $U(N)/U(N)^2$ it follows that $\epsilon_{\Phi'}$ is the trivial quadratic character, and thus $\Phi'(\ggg_F) \subset U(N)^2$.  Thus if $L = (F^{\sep})^{\Ker \Phi'}$, then $A'$ has an
$L$-rational point of order $N$ and $[L:F] \mid \# U(N)^2 = \frac{\varphi(N)}{2}$.
\end{proof}

%\subsection{Construction of Odd Degree Torsion}

\begin{prop}
%\label{THM1}
\label{7.4}
Let $E_{/F}$ be an $\OO$-CM elliptic curve with $F = \Q(j(E))$.  Let $\ell \mid \Delta(\OO)$ be a prime.\\
a) There is a unique prime ideal $\pp_{\ell}$ of $\OO$ such that $\pp_{\ell} \cap \Z = (\ell)$. \\
%  Moreover
%\[ \ell \OO = \pp_{\ell}^2. \]
b) $E[\pp_{\ell}] = \{x \in E(\C) \mid \alpha x = 0 \ \forall \alpha \in \pp_{\ell}\}$ is an $F$-rational subgroup
of order $\ell$.
\end{prop}
\begin{proof}
\renewcommand{\aa}{\mathfrak{a}}
a) Since $\ell \mid \Delta$, we have $\OO/\ell \OO \cong \Z/\ell \Z[\epsilon]/(\epsilon^2)$,
so there is a unique ideal $\pp_{\ell}$ of $\OO$ with $\# \OO/\pp_{\ell} = \ell$.  By uniqueness $\overline{\pp_{\ell}} = \pp_{\ell}$; thus $E[\pp_{\ell}]$ is an $F$-rational subgroup. \\
b) When $\OO = \OO_K$, we have $\# E[\aa] = \# \OO_K/\aa$ for all
nonzero ideals $\aa$ of $\OO_K$ \cite[Prop. II.1.4]{SilvermanII}.
%\footnote{As the remainder of the argument is a bit technical, %it may be worthwhile to note that only the case $\OO = \OO_K$ %will be needed in the proof of Theorem \ref{ODDDEGREETHM}.}
In the general case, we may embed
$\Q(j) \hookrightarrow \C$ so as to have $j(E) = j(\C/\OO)$, and then we get
\[ E[\pp_{\ell}] = \{ x \in \C \mid x \pp_{\ell} \subset \OO\}/\OO. \]
We have
\[ \{ x \in \C \mid x \pp_{\ell} \subset \OO\} = \{ x \in K \mid x \pp_{\ell} \subset \OO\}= (\OO:\pp_{\ell}).\]
Observe that $E[\pp_{\ell}] = (\OO:\pp_{\ell})/\OO$ is a vector space over the field $\F_{\ell} = \OO/\pp_{\ell}$; it remains to compute its dimension.  For any domain $R$, any two elements of a fractional
$R$-ideal are $R$-linearly dependent, so $\dim_{\F_{\ell}} (\OO:\pp_{\ell})/\OO \in \{0,1\}$.  Since $\OO$ is a one-dimensional Noetherian domain and $\pp_{\ell} \subsetneq \OO$ we have $(\OO:\pp_{\ell}) \supsetneq \OO$ \cite[$\S 10.2$, Lemma 4]{Jacobson}
and thus $1 = \dim_{\F_{\ell}} (\OO:\pp_{\ell})/\OO = \dim_{\F_{\ell}} E[\pp_{\ell}]$.
\end{proof}

\begin{cor}
\label{7.5}
\label{COR5.7}
Let $\OO$ be an imaginary quadratic order of discriminant $\Delta$, and let $\ell > 2$ be a prime dividing $\Delta$. \\
a)  There is a number field $L$ of degree $h(\Delta) \left(\frac{\ell-1}{2}\right)$ and an $\OO$-CM elliptic
curve $E_{/L}$ with an $L$-rational torsion point of order $\ell$.  \\
b) Suppose $\ell \equiv 3 \pmod{4}$ and $\OO$ is the quadratic order of
discriminant $-\ell$, i.e., the ring of integers of $K = \Q(\sqrt{-\ell})$.
Let $j = j(\OO)$ and $F = \Q(j)$.  Then: \\
(i) The number field $F(\zeta_{\ell}+\zeta_{\ell}^{-1})$ has degree $h(K)(\ell-1)/2$, an odd number. \\
(ii) There is an elliptic curve $E_{/F}$ such that $E(F(\zeta_{\ell}+\zeta_{\ell}^{-1}))$ has a point of order $\ell$.
%b) When $\ell \equiv 3 \pmod{4}$, there is a CM ellptic curve defined over
%a number field $L$ of odd degree with an $L$-rational point of order $\ell$.
%Let $\ell \equiv 3 \pmod{4}$ be a prime, let $K = \Q(\sqrt{-\ell})$, let
%$\OO = \Z_K$, let $j = j(\OO)$, $F = \Q(j)$, and let $E_{/F}$
%be an $\OO$-CM elliptic curve.  Then there is a quadratic twist $E'$ of $E$ %such that
%$E'(F(\zeta_{\ell} + \zeta_{\ell}^{-1}))$ has a point of order $\ell$.
\end{cor}
\begin{proof}
a) Combine Proposition \ref{7.4} and Theorem \ref{7.2}.  \\
b) (i) Since $[F:\Q]=h(K)$ is odd by Lemma  \ref{LEMMA2.5}, the genus field $FK \cap \Q^{\ab}$ is $K$
and thus $F \cap \Q(\zeta_N)^+ = \Q$, i.e., $F$ and $\Q(\zeta_\ell+\zeta_\ell^{-1})$
are linearly disjoint over $\Q$.  It follows that $[F(\zeta_{\ell}+\zeta_{\ell}^{-1}):\Q] = h(K) \left( \frac{\ell-1}{2} \right)$, which is odd. \\
(ii) Let $E_{/F}$ be an $\OO$-CM elliptic curve. By Proposition \ref{7.4}, $E$ has an $F$-rational subgroup of order $\ell$.
By Theorem \ref{THM3}, after replacing $E$ by a quadratic twist \emph{over $F$},
there is an extension
$L/F$ of degree dividing $\frac{\ell-1}{2}$ such that $E(L)$ has a point of order $\ell$.  By Real Cyclotomy II, $F(\zeta_{\ell} + \zeta_{\ell}^{-1}) \subset L$.  Thus
\[ [L:\Q] \mid h(K) \left(\frac{\ell-1}{2}\right) = [F(\zeta_\ell + \zeta_\ell^{-1}):\Q], \]
so $L = F(\zeta_\ell + \zeta_\ell^{-1})$.
\end{proof}

\begin{remark}
Part b) is due to Lozano-Robledo when $F = \Q$ \cite[Cor. 9.8]{Alvaro}.
\end{remark}

\subsection{Existence of Odd Degree Torsion}
\textbf{} \\ \\ \noindent
The groups $\{ \bullet\}, \, \Z/2\Z, \, \Z/4\Z,$ or $\Z/2\Z \times \Z/2\Z$ are Olson groups, so occur already over $\Q$.  To complete Theorem \ref{ODDDEGREETHM}, it remains to construct CM elliptic curves over
odd degree number fields with torsion subgroups isomorphic to $\Z/\ell^n \Z$ or $\Z/2\ell^n \Z$ for $\ell$ as in Theorem \ref{ODDDEGREETHM} above.
%We do so in this section.

\begin{lemma}
\label{MinField}
Let $E_{/F}$ be an $\OO_K$-CM elliptic curve such that $F(E[\ell])=K^{(\ell)}$ for some prime $\ell> 2$. Suppose additionally that $K \neq \Q(i)$, and $\ell =3$ if $K=\Q(\sqrt{-3})$. Then $F(E[\ell^n])=K^{(\ell^n)}$ for all $n \in \Z^+$.
\end{lemma}
\begin{proof}
We will handle the two cases separately. \\
\indent Case 1: Let $E_{/F}$ be an elliptic curve with CM by the maximal order in $K \neq \Q(i), \Q(\sqrt{-3})$ such that $F(E[\ell])=K^{(\ell)}$ for a prime $\ell>2$. Suppose, for the sake of contradiction, that $F(E[\ell^n]) \neq K^{(\ell^n)}$ for some positive integer $n$. By \cite[Thm. II.5.6]{SilvermanII}, we have \[
K^{(\ell^n)}=K(j(E),\mathfrak{h}(E[\ell^n]))=FK(\mathfrak{h}(E[\ell^n])) \subset F(E[\ell^n]).
\] Since $K \neq \Q(i), \Q(\sqrt{-3})$, we may take $\mathfrak{h}(E[\ell^n])=x(E[\ell^n])$. Thus $[F(E[\ell^n]):K^{(\ell^n)}]=2$ if $F(E[\ell^n]) \neq K^{(\ell^n)}$.  Let $\sigma \in \mathfrak{g}_F$ generate $\Aut(F(E[\ell^n])/K^{(\ell^n)})$. Then $\sigma$ corresponds to a matrix of order 2 in $\GL_2(\Z/\ell^n \Z)$ which is trivial mod $\ell$ since $F(E[\ell]) \subset K^{(\ell^n)}$.  But the kernel of $\GL_2(\Z/\ell^n \Z) \to \GL_2(\Z/\ell \Z)$ is an $\ell$-group. \\
\indent Case 2: If $K = \Q(\sqrt{-3})$ and $\ell =3$, we have $\Ok \otimes \Z_{3} \cong \Z_{3}[\sqrt{-3}]$. Thus we may choose a basis $\widetilde{e_1},\widetilde{e_2}$ for $T_{3}(E)$ for which $\rho_{3^{\infty}}(\mathfrak{g}_{FK})$ lands in the Cartan subgroup

\[ \mathcal{C}_{3}^{\times} = \left\{ \left[ \begin{array}{cc} \alpha & \beta \\ -3\beta & \alpha \end{array} \right] \mid \alpha^2 +3 \beta^2 \in \Z_{3}^{\times} \right\}. \]

 Elements of $\Aut(F(E[3^n])/F(E[3]))=\Aut(F(E[3^n])/K)$ correspond to elements of $\mathcal{C}_{3}^{\times}$ modulo $3^n$ which are congruent to the identity matrix modulo 3. There are precisely $3^{2n-2}$ such matrices, giving an upper bound on $\#\Aut(F(E[3^n])/K)$. On the other hand, by \cite[Thm. II.5.6]{SilvermanII} we have

\[
K^{(3^n)}=K(j(E),\mathfrak{h}(E[3^n]))\subset F(E[3^n]).
\]
As $[K^{(3^n)}:K]=3^{2n-2}$, we have $[F(E[3^n]):K]=3^{2n-2}$ and $F(E[3^n])=K^{(3^n)}$.
\end{proof}

\begin{prop}
\label{Prop1.4}
Let $\ell \equiv 3 \pmod{4}$ be prime, and let $n \in \Z^+$. There exists an elliptic curve $E$ defined over a number field $F$ such that: \\
%\begin{enumerate}
(i) $E$ has CM by the full ring of integers in $K=\Q(\sqrt{-\ell})$; \\
(ii) $[F:\Q]$ is odd; and \\
(iii) $E(F)$ has a point of order $\ell^n$.
%\end{enumerate}
\end{prop}

\begin{proof}
Let $E$ be an elliptic curve with CM by $\Ok$. Choose a model of $E$ defined over $F=\Q(j(E))$. For now, suppose $\ell \neq 3$. By Corollary \ref{7.5}, there is a quadratic twist $E'_{/F}$ of $E$ such that $E'(F(\zeta_{\ell}+\zeta_{\ell}^{-1}))$ has a point of order $\ell$. Hence we may assume $E(F(\zeta_{\ell}+\zeta_{\ell}^{-1}))$ has a point of order $\ell$. In fact, this implies $F(E[\ell])=K^{(\ell)}$, as we will now show. Since $E(F(\zeta_{\ell}))$ contains a point of order $\ell$,
\[
\Aut(F(E[\ell])/F(\zeta_{\ell})) \cong \left<
\left[\begin{smallmatrix}
1&b\\ 0&1
\end{smallmatrix}\right]
\right>,
\]
where $b \in \mathbb{F}_{\ell}$. Thus $[F(E[\ell]):F(\zeta_{\ell})]=1$ or $\ell$. By Theorem \ref{3.16}, we have $K^{(\ell)} \subset F(E[\ell])$, so $[K^{(\ell)}:\Q]=h(K) \ell (\ell-1)$ divides $[F(E[\ell]):\Q]$. This forces $[F(E[\ell]):F(\zeta_{\ell})]=\ell$ and $F(E[\ell])=K^{(\ell)}$.

\begin{center}
\begin{tikzpicture}[node distance=2cm]
\node (Q){$\Q$};
\node (F) [above of=Q, node distance=2cm] {$F=\Q(j(E))$};
\node (Fmu) [above of =F, node distance=1.8cm] {$F(\zeta_{\ell})$};
\node (Ftor) [above of =Fmu, node distance=1.8cm] {$F(E[\ell])$};
\node (K)  [above right of=Q, node distance=2 cm]   {$K$};
\node (Kl) [above of=K, node distance=3.2 cm]   {$K^{(\ell)}$};

 \draw[-] (Q) edge node[auto] {$h(K)$} (F);
 \draw[-] (F) edge node[auto] {$\ell-1$} (Fmu);
 \draw[-] (Fmu) edge node[auto] {$\ell$} (Ftor);
  \draw[-] (Q) edge node[auto] {$2$} (K);
 \draw[-] (K) edge node[right] {$\frac{1}{2}h(K)\ell(\ell-1)$} (Kl);
 \draw[-] (Kl) edge node[left] {} (Ftor);

\end{tikzpicture}
\end{center}

\noindent By Lemma \ref{MinField}, it follows that $F(E[\ell^n])=K^{(\ell^n)}$.  Viewing $F \subset \R$, we have $E(\R) \cong \R/\Z$ or $\R/\Z \times \Z/2\Z$. Thus we have an element of order $\ell^n$ fixed by the element $c \in \mathfrak{g}_F$ induced by complex conjugation. Then $F(E[\ell^n])^c$ contains the coordinates of a point of order $\ell^n$, and
\begin{align*}
[F(E[\ell^n])^c:\Q]&=\frac{1}{2}[F(E[\ell^n]):\Q]\\
&=\frac{1}{2}[K^{(\ell^n)}:\Q]\\
&=\frac{1}{2}h(K)(\ell-1)\ell^{2n-1}.
\end{align*}
Lemma \ref{LEMMA2.5} gives that $h(K)$ is odd, so $[F(E[\ell^n])^c:\Q]$ is odd as desired.

If $\ell=3$, consider the elliptic curve $y^2=x^3+16$. This curve has CM by the full ring of integers in $K=\Q(\sqrt{-3})$ and one finds (e.g. by direct calculation) that $\Q(E[3])=K^{(3)}=K$. By Lemma \ref{MinField}, $\Q(E[3^n])=K^{(3^n)}$. As before, $\Q(E[3^n])^c$ contains the coordinates of a point of order $3^n$, and
\begin{align*}
[\Q(E[3^n])^c:\Q]&=\frac{1}{2}[\Q(E[3^n]):\Q]\\
&=\frac{1}{2}[K^{(3^n)}:\Q]\\
&=3^{2n-2}. \qedhere
\end{align*}
\end{proof}

\begin{thm}
Let $n \in \Z^+$.
\begin{enumerate}
\item If $\ell \equiv 3 \pmod{8}$, there exists a CM elliptic curve $E$ defined over a number field $F$ of odd degree such that $E(F)[\tors] \cong \Z/\ell^n\Z$.
\item If $\ell \equiv 3 \pmod{4}$, there exists a CM elliptic curve $E$ defined over a number field $F$ of odd degree such that $E(F)[\tors] \cong \Z/2\ell^n\Z$.
\end{enumerate}
\end{thm}

\begin{proof}
Suppose $\ell \equiv 3 \pmod{4}$, and let $K=\Q(\sqrt{-\ell})$. By Proposition \ref{Prop1.4}, there is an $\Ok$-CM elliptic curve $E$ defined over $F=\Q(j(E))$ and an extension $\tilde{F}:=F(E[\ell^n])^c$ of odd degree over $\Q$ such that $E(\tilde{F})[\tors]$ contains a point of order $\ell^n$. Since $\tilde{F}$ is odd, we know $E(\tilde{F})[\tors]$ is either cyclic or of the form $\Z/2N\Z \times \Z/2\Z$; otherwise full $N$-torsion would force $\Q(\zeta_{N}) \subset \tilde{F}$ by the Weil pairing. We first establish that $\ell^n$ is the largest power of $\ell$ dividing $\#E(\tilde{F})[\tors]$.

Suppose $\ell^{n+1}$ divides $\#E(\tilde{F})[\tors]$. By Real Cyclotomy II, $\Q(\zeta_{\ell^{n+1}}) \subset \tilde{F}K=K^{(\ell^n)}$. Then elements of
\[
\Gal(F(E[\ell^{n+1}])/F(E[\ell^{n}]))=\Gal(K^{(\ell^{n+1})}/K^{(\ell^n)})
\]
correspond to matrices of the form

\[  \left\{ \left[ \begin{array}{cc} 1 & \beta \\ 0 & 1 \end{array} \right] \mid \beta \equiv 0 \pmod{\ell^n} \right\}. \]
There are $\ell$ such matrices. However, $[K^{(\ell^{n+1})}:K^{(\ell^n)}]=\ell^2$. Thus $\ell^n$ is the largest power of $\ell$ dividing $\#E(\tilde{F})[\tors]$.

Suppose $\ell \equiv 3 \pmod{8}$. Since $\ell^n$ is the largest power of $\ell$ dividing $\#E(\tilde{F})[\tors]$, it suffices to show that $E(\tilde{F})$ contains no point of order 2. Suppose first that $\Delta \neq -3$. Since $\Delta \equiv 5 \pmod{8}$, there are no points of order 2 rational over $F=\Q(j(E))$ by Theorem \ref{THM4.2}. By construction, $\tilde{F} \subset F(E[\ell^n])=K^{(\ell^n)}$, and $F(E[2])=K^{(2)}$ by the proof of Theorem \ref{THM4.2}: $F(E[2])$ has degree $6h(K)$ over $\Q$ and contains $K^{(2)}$, which also has degree $6h(K)$. Thus if $\tilde{F}$ contains the coordinates of a point of order 2, then $3 \mid [\tilde{F} \cap K^{(2)}:F]$ and hence $3 \mid [K^{(\ell^n)} \cap K^{(2)}:F] = [K^{(1)}:F] = 2$.
 %So $E(\tilde{F})$ contains no point of order 2.

We now consider the case where $\Delta = -3$. The curve $y^2=x^3+16$ has no $\Q$-rational points of order 2. If $E(\tilde{F})$ contains a point of order 2, then $\Q(E[3^n])=K^{(3^n)}$ contains a root $\alpha$ of $x^3+16$. But this cannot be, since 2 ramifies in $\Q(\alpha)/\Q$ and is unramified in $K^{(3^n)}$.

Thus if $\ell \equiv 3 \pmod{8}$, we have verified there is an $\Ok$-CM elliptic curve $E$ defined over a number field $\tilde{F}$ of odd degree such that $E(\tilde{F})[\tors] \cong \Z/\ell^n\Z$. If $(\alpha, 0)$ is a point of order 2, then $[\tilde{F}(\alpha):\Q]=3\cdot[F:\Q]$ is odd and $E(\tilde{F}(\alpha))[\tors] \cong  \Z/2\ell^n\Z$.

If $\ell \equiv 7 \pmod{8}$, as described above we have an $\Ok$-CM elliptic curve $E$ defined over $F=\Q(j(E))$ and an extension $\tilde{F}/F$ of odd degree over $\Q$ such that $E(\tilde{F})[\tors]$ contains a point of order $\ell^n$ and no point of order $\ell^{n+1}$. Since $\Delta \equiv 1 \pmod{8}$, by Theorem \ref{THM4.2} we have a $2$-torsion point rational over $F=\Q(j(E))$. Corollary \ref{4.3} ensures we do not have full $2$-torsion. Thus we have $E(\tilde{F})[\tors] \cong \Z/2\ell^n\Z$.
\end{proof}

\subsection{Number Fields of $S_d$-Type}
\textbf{} \\ \\
%Finally we give a counterpoint to the above when we add on the CM assumption.
Let $G$ be a finite group.  A number field $F$ is \textbf{of G-type}
if the automorphism group of the Galois closure of $F/\Q$ is isomorphic to $G$.

\begin{thm}
\label{100PERCENTJIM}
Let $d$ be an odd positive integer, and let $F$ be a degree $d$ number field
of $S_d$-type.  Then every CM elliptic curve $E_{/F}$ is Olson.
\end{thm}
\begin{proof}
Step 1: Let $M$ be the normal closure of $F/\Q$, and choose an isomorphism
$S_d \cong \Aut(M/\Q)$.  Let $A$ (resp. $B$) be the maximal abelian subextension of $F/\Q$ (resp. of $M/\Q$).  Then \[B = M^{[S_d,S_d]} = M^{A_d}, \]
so $[B:\Q] = 2$.  Since $\Q \subset A \subset B \cap F$
and $[F:\Q] = d$ is odd, we have $A = \Q$.   \\
Step 2: Let $E_{/F}$ be a CM elliptic curve, and suppose $E(F)$ contains a point of order $N$.  By Real Cyclotomy II, $F$ contains $\Q(\zeta_N)^+$.  By Step 1, $\Q(\zeta_N)^+ = \Q$ so $N \in \{1,2,3,4,6\}$.  The Odd
Degree Theorem now implies that $E_{/F}$ is Olson.
\end{proof}

%\subsection{Degree 9}

%\begin{thm}
%Let $F/\Q$ be a number field of degree $9$, and let %$E_{/F}$ be a CM elliptic curve.  Then $E_{/F}$ is Olson %or
%\[ E(F)[\tors] \in \{\Z/9\Z, \Z/14\Z, \Z/18\Z, \Z/19\Z %\Z/27\Z \}. \]
%\end{thm}
%\begin{proof}

%\end{proof}

%subsection{Degree 15}

\section{Restricted Odd Degrees}

\subsection{The Shifted Prime Degree Theorem}

\begin{thm}(Shifted Prime Degree Theorem)
\label{NEWABBEYTHM}
\label{THM6.1} \label{SPDT} \\
a) There is a function $C: \Z^+ \ra \Z^+$ such that: for all
$k \in \Z^+$, all primes $p$, all number fields
$F/\Q$ of degree $(2k-1)p$, and all CM elliptic curves
$E_{/F}$, we have
\[ \# E(F)[\tors] \leq C(k). \]
b) There is a function $P: \Z^+ \ra \Z^+$ such that:
for all $k \in \Z^+$, all primes $p \geq P(k)$ all number fields $F$ of degree $(2k-1)p$, and all CM elliptic curves $E_{/F}$, there is a subfield $F_0 \subset F$ of degree
dividing $(2k-1)$ and an $F_0$-rational model $(E_0)_{/F_0}$ of $E_{/F}$ such that $E_0(F_0)[\tors] = E(F)[\tors]$.
\end{thm}
\begin{proof}
a) Fix $k \in \Z^+$, let $p$ be a prime, let $F/\Q$ be a number field of degree $(2k-1)p$, and let $E_{/F}$ be a CM elliptic curve.  By \cite{HS99} we have \[\# E(F)[\tors] \leq 1977408(2k-1)p \log ((2k-1)p), \] and we may assume that $p$ is odd.  We may also assume $\# E(F)[\tors] > 6$.  By
the Odd Degree Theorem we have $E(F)[\tors] \cong \Z/\ell^n \Z$ or $\Z/2\ell^n \Z$ for some prime $\ell \equiv 3 \pmod{4}$ and $E$ has $\Q(\sqrt{-\ell})$-CM.  Real Cyclotomy II implies
\begin{equation}
\label{NEWABBEYEQ1}
\frac{(\ell-1) \cdot h_{\Delta}}{2} \mid (2k-1)p .
\end{equation}If $\frac{\ell-1}{2} > (2k-1)$, then $p \mid \frac{\ell-1}{2}$ and thus
\begin{equation}
\label{NEWABBEYEQ2}
\ell > 4k-1 \implies h_{\Q(\sqrt{-\ell})} \mid h_{\Delta} \mid 2k-1.
\end{equation}
There are only finitely many such $\ell$; for each,
if $n \geq 2$ then by Real Cyclotomy II,
\[ \ell^{n-1} \mid \frac{\varphi(\ell^n)}{2} \mid [F:\Q] =
(2k-1) p \]
and thus
\[ n-2 \leq \ord_{\ell}(2k-1) \leq \log_{\ell} (2k-1)
\leq \log_2 (2k-1). \]
b) Fix $k \in \Z^+$, and let $F$ be a number field of degree $(2k-1)p$ for some odd prime $p$.  Let $E_{/F}$ be a non-Olson elliptic curve with CM by the imaginary quadratic
order $\OO(\Delta)$ with fraction field $K$.  By part a) we have $E(F)[\tors] \cong \Z/2^{\epsilon} \ell^n \Z$ for $\epsilon \in \{0,1\}$,
$\ell \equiv 3 \pmod{4}$ a prime number, $n \in \Z^+$, $K = \Q(\sqrt{-\ell})$ for one of finitely many integers $2^{\epsilon} \ell^n$.  We may assume that $p$ is not equal to any of these finitely many $\ell$.  We then observe that for each
$2^{\epsilon} \ell^n$ and all remaining $p$, there are only finitely many possibilities for the discriminant $\Delta$.  Indeed $\Delta = -\mathfrak{f}^2 \ell$, so it suffices to bound the conductor $\mathfrak{f}$.  Moreover $h_{\Delta}$ is odd, so by Lemma \ref{LEMMA2.5} we have $\mathfrak{f} = 2^{\delta} \ell^m$ for $\delta \in \{0,1\}$ and $m \in \Z^+$.  The relative class number formula -- e.g. \cite[Thm. 7.24]{Cox89} -- shows that
\[ \ord_{\ell} h_{\Delta} = \ord_{\ell} (h_{-(2^{\delta} \ell^m)^2\ell}) \geq m - 1, \]
and thus since $p \neq \ell$ we have $m \leq \ord_{\ell}(2k-1)+1$.  \\ \indent
Let $P \in E(F)$ be a point of order
$2^{\epsilon} \ell^n$.  Having thrown away finitely many primes $p$, it now follows
that $j(E)$ lies in a finite set and thus the pair $(E,P)$ induces one of a finite set of $\overline{\Q}$-valued points, say $\{P_1,\ldots,P_r\}$, on the modular curve $X_1(2^{\epsilon} \ell^n)$.  For $1 \leq i \leq r$, let
$\Q(P_i)$ be the field of definition (= residue field of
the corresponding closed point on the $\Q$-scheme $X_1(2^{\epsilon} \ell^n)$), and let $d_i = [\Q(P_i):\Q]$.  For some $1 \leq i \leq r$ we have an injection $\Q(P_i) \hookrightarrow F$ and thus
\[ d_i \mid (2k-1)p. \]
For each $i$, there is at most one prime, say $p_i$ such that $d_i \mid (2k-1)p_i$ but $d_i \nmid (2k-1)$.  Throwing away
these finitely many primes $p_i$ we get
\[ d_i \mid (2k-1). \]
Since $X_1(2^{\epsilon} \ell^n)$ is a fine moduli space, there is a unique model of $(E,P)$ defined over $\Q(P_i)$, a field
of degree dividing $(2k-1)$.
%Now suppose that there are no points of order $2^{\epsilon} %\ell^n$ in degree $(2k-1)$.  We must show that the set of %primes $p$ for which there are points of order $2^{\epsilon} %\ell^n$ in degree $(2k-1)p$ is finite.  But we have %restricted $\Delta$ and thus the set of $j$-invariants to a %finite set, and it follows that we are contemplating a finite %set $\{P_1,\ldots,P_r\}$ of $\C$-valued points on the modular %curve $X_1(2^{\epsilon} \ell^n)$.  The degrees of number %fields $F$ over which we can define elliptic curves $E_{/F}$ %inducing these points on the modular curve are the multiples %of the degrees
%$\{[\Q(P_1):\Q],\ldots,[\Q(P_r):\Q]\}$ of the minimal fields %of definition of these points (= residue fields of the %corresponding closed points).  If for any $1 \leq i \leq r$ %we had $[\Q(P_i):\Q] \mid (2k-1)p$ for more than one prime %$p$, then we would have $[\Q(P_i):\Q] \mid 2k-1$, %contradicting our assumption.
\end{proof}

\begin{remark}
Along with Siegel's minoration $h_{\Q(\sqrt{-\ell})} \gg_{\epsilon}
\ell^{\frac{1}{2} - \epsilon}$, the proof shows that for all $\epsilon \in (0,\frac{1}{2})$, there is an (ineffective) $c_{\epsilon} > 0$ such that we can take
 \[ C(k) = \max \left\{3954816(2k-1) \log(4k-2), 2 \left(\frac{2k-1}{c_{\epsilon}}\right)^{\frac{2+\log_2(2k-1)}{\frac{1}{2}-\epsilon}}
\right\}. \]
%{\color{red}{b) It should be possible to extract an explicit %bound on P(k)...}}
\end{remark}
%\noindent
%Taking $k = 1$, Theorem \ref{NEWABBEYTHM} yields the absolute %boundedness of torsion on CM elliptic curves over prime degree %number fields and also shows that for all sufficiently large %primes $p$ and all number fields $F$ of degree $p$, every CM %elliptic curve $E_{/F}$ is Olson.  Our next order of business is %to strengthen this result.

\subsection{Proof of the Prime Degree Theorem}
\textbf{} \\ \\ \noindent
In this section we will prove Theorem \ref{MAINTHM}.  We need just one preliminary result.

\begin{lemma}
\label{LITTLEPRIMEDEGREELEMMA}
Let $F \subset \C$, let $j \in F \setminus \{0,1728\}$,
and let $N \geq 3$ be an \emph{odd} integer.  Suppose that there are elliptic curves $E_{/F}$ and $E'_{/F}$ with $j(E) = j(E') = j$, that $E(F)$
and $E'(F)$ each have a point of order $N$, and that
$E \not \cong_F E'$.  Then: \\
a) There is a quadratic extension $F'/F$ such that $F(E[N]) \subset F'$.  \\
b) If $E$ has $K$-CM, then $F' \supset K^{(N)}$.
\end{lemma}
\begin{proof}
a) The hypotheses imply that $E' = E^t$ is the quadratic
twist of $E$ by some $t \in F^{\times} \setminus F^{\times 2}$.  Since $N$ is odd, we have $E(F)[N] \oplus E^d(F)[N] \cong E(F(\sqrt{t}))[N]$.  Taking
$F' = F(\sqrt{d})$, the result follows. \\
b) This follows from part a), Lemma \ref{LEMMA2.7}, and Theorem \ref{SCHERTZTHM}.
\end{proof}
\noindent
Now we begin the proof of Theorem \ref{MAINTHM}. \\
Step 1: Let $F/\Q$ be a quadratic field, and let $E_{/F}$ be a non-Olson CM elliptic curve.  By Theorem \ref{PARISHTHM}, we must have $j(E) \in \Q$.  By the Main Theorem of \cite{Clark-Xarles08}, the subgroup of $E(F)$ of points of finite odd order has size at most $9$.  In particular the largest possible prime divisor of $\# E(F)[\tors]$ is $7$.\footnote{This approach is not very representative of the type of computations used for larger even degrees -- but it is a bit shorter in this case.}   Thus if $E_{/F}$ is not Olson, $E(F)[\tors]$ \emph{contains}
one of the following subgroups:
\[ \Z/5\Z, \, \Z/7\Z, \, \Z/8\Z, \, \Z/9\Z, \, \Z/12\Z, \, \Z/2\Z \oplus \Z/4\Z, \, \Z/2\Z \oplus \Z/6\Z, \, \Z/3\Z \oplus \Z/3\Z. \]
For each such group, we look at the corresponding row in Table 2; for each entry of $2$, we build the corresponding elliptic curve, compute its torsion subgroup, and check for isomorphism over the ground field.  Thus: \\
$\bullet$ For $\Z/5\Z$, we have $\Delta = -4$.  Here the full torsion subgroup is $\Z/10\Z$.  \\
$\bullet$ For $\Z/7\Z$, we have $\Delta = -3$. \\
$\bullet$ For $\Z/8\Z$, $\Z/9\Z$ and $\Z/12\Z$,  there are no such elliptic curves. \\
%$\bullet$ For $\Z/9\Z$, there are no such elliptic curves.  \\
%$\bullet$ For $\Z/12\Z$, there are no such elliptic curves. \\
$\bullet$ For $\Z/2\Z \oplus \Z/4\Z$, we have $\Delta = -4$,
$\Delta = -7$, $\Delta = -8$, or $\Delta = -16$.\footnote{Here, in five out of the six cases, two distinct points on the modular curve $X(2,4)$ give rise to elliptic curves which are isomorphic over the minimal field of definition.  This can be understood by reflecting on the moduli problem, a task we leave to the interested reader.} \\
$\bullet$ For $\Z/2\Z \oplus \Z/6\Z$, we have $\Delta = -3$ or $\Delta = -12$.\\
$\bullet$ For $\Z/3\Z \oplus \Z/3\Z$, we have $\Delta = -3$.\\
Step 2: Let $F/\Q$ be a number field with prime degree $p > 2$, and let $E_{/F}$ be an elliptic curve with CM by the order $\OO$ of discriminant $\Delta$ in the imaginary quadratic field $K$ such that $E(F)[\tors]$ is \emph{not} Olson.  As above, by Real Cyclotomy II and the Odd Degree Theorem: if $E(F)[\tors]$ is isomorphic to $\Z/\ell^n \Z$ or $\Z/2\ell^n \Z$ for some prime $\ell \equiv 3 \pmod{4}$, then $K \cong \Q(\sqrt{-\ell})$ and $h_{\Delta}$ is odd, so
\[ \frac{ (\ell-1) h_{\Delta}}{2} \mid [F:\Q] = p. \]
%This is very restrictive! \\
$\bullet$ If $\ell = 3$ then since $E$ is not Olson, $E(F)$ has a point order $9$.  Thus $\frac{\varphi(9)}{2} =3 \mid p$.  So $p = 3$ and $\Delta \in \{-3,-12,-27\}$.
If $\ell > 3$, then $\frac{\ell-1}{2} > 1$, so $h_{\Delta} = 1$ -- thus Parish's Theorem comes for free in this context -- hence we must have
\[
\Delta \in \{-7,-11,-19,-28,-43,-67,-163\} \]
and
\[ p = \frac{\ell-1}{2}, \]
i.e., $p$ is a \textbf{Sophie Germain prime} (equivalently, $\ell$ is a \textbf{safe prime}). \\
$\bullet$ If $\ell = 7$, then $\frac{\varphi(7)}{2} = 3 \mid p$.  So $p = 3$.  \\
$\bullet$ If $\ell = 11$, then $\frac{\varphi(11)}{2} = 5 \mid p$.  So $p = 5$.  \\
$\bullet$ No $\ell \in \{19,43,67,163\}$ is safe -- i.e., $\frac{\ell-1}{2} \in \{9,21,33,81\}$ is not prime.  It follows that if $p \geq 7$, then $E(F)[\tors]$ is Olson.
\\
Step 3: Suppose $p = 3$. \\
% From Step 1, either $K = \Q(\sqrt{-3})$ and $E$ has a %point of order $9$, or $K = \Q(\sqrt{-7})$
Step 3a): Suppose $\ell = 3$.  From Step 2, $F = \Q(\zeta_9)^+$ and $\Delta \in \{-3,-12,-27\}$.  Using Table 2, we find that when $\Delta \in \{-3,-27\}$ there is exactly one $\OO(\Delta)$-CM elliptic curve over a degree $3$ number field with an
$F$-rational point of order $9$, and that when $\Delta = -12$ there are no such curves.\\
Step 3b): Suppose $\ell = 7$.  From Step 2, $F = \Q(\zeta_7)^+$ and $\Delta \in \{-7,-28\}$, i.e.,
$j(E) \in \{-3375,16581375\}$.  We claim that for each of these two $j$-invariants there is a \emph{unique} CM elliptic curve $E/\Q(\zeta_7)^+$ with a rational point of order $7$, up to isomorphism over $\Q(\zeta_7)^+$.  The existence is guaranteed by Corollary \ref{COR5.7}.  The
uniqueness follows from Lemmma \ref{LITTLEPRIMEDEGREELEMMA}: if there were two such elliptic curves, then there would be a quadratic extension $F'/\Q(\zeta_7)^+$ such that $\Q(\sqrt{-7})^{(7)} \subset F'$.  But $[\Q(\sqrt{-7})^{(7)}:\Q] = 42$ and $[F':\Q] = 6$: a contradiction.
\\
Step 4: Let $p = 5$.  From Step 2, $F = \Q(\zeta_{11})^+$ and $\Delta = -11$.  Arguing as in Step 3b), we find that there is exactly one CM elliptic curve $E_{/F}$ with an $F$-rational point of order $11$: $[\Q(\sqrt{-11})^{(11)}:\Q] = 110 > 2[\Q(\zeta_{11})^+:\Q] = 10$.  It follows from Theorem \ref{THM4.2} and Real Cyclotomy II that $E(F)[\tors] \cong \Z/11\Z$. This completes the proof of Theorem \ref{MAINTHM}.

\begin{remark}
%a)
The degree sequences in Table 2 and the corresponding elliptic curves are computed using the methods of \cite{TORS2}.  The classification of torsion \emph{groups} of CM elliptic curves in degrees up to $13$ was done there.  The method yields the full list of ``new'' elliptic curves, but unfortunately these lists were not recorded.  The computation in degrees $2$ and $3$ had already been done by the second author in 2004 \cite{Clark04}, and the argument of Step 0 essentially reproduces this.  However, looking back at the degree sequences recorded in \cite{Clark04} we found several errors (presumably of transcription), so we decided to recompute Table 2 from scratch.
%\\ \noindent
%b) In an earlier draft, the classification of torsion in degrees %$3$ and $5$ was done purely computationally.  However, the above %``principled derivation'' seems much more enlightening.
\end{remark}

\subsection{The Prime Squared Degree Theorem}

\begin{thm}
\label{PSDT}
Let $p$ be a prime, let $F/\Q$ be a number field of degree $p^2$, and let $E_{/F}$ be a CM elliptic curve. Then: \\
a) If $p = 2$, then $E_{/F}$ is Olson or $E(F)[\tors]$ is isomorphic to one of the following:
\begin{align*}
& \Z/m\Z, \hspace{.3 cm}  m \in \{5, 7, 8, 10, 12, 13, 21\}\\
& \Z/2\Z \oplus \Z/m\Z, \hspace{.3cm}  m \in \{4, 6, 8, 10\}\\
& \Z/3\Z \oplus \Z/m\Z, \hspace{.3cm}  m \in \{3,6\} \\
& \Z/4\Z \oplus \Z/4\Z \hspace{.3cm}
\end{align*}
b) If $p = 3$, then $E_{/F}$ is Olson or $E(F)[\tors]$ is isomorphic to one of the following:
\[ \Z/9\Z, \, \Z/14\Z, \, \Z/18\Z, \, \Z/19\Z, \, \Z/27\Z.                              \]
c) If $p = 5$, then $E_{/F}$ is Olson or $E(F)[\tors] \cong \Z/11\Z$.  \\
d) If $p \geq 7$, then $E_{/F}$ is Olson.
%For all primes $p \geq 5$, if $E_{/F}$ is a CM elliptic %curve and $[F:\Q] = p^2$, then $E_{/F}$ is Olson.
\end{thm}
\begin{proof}
a) See \cite[$\S 4.4$]{TORS2}.  b) See \cite[$\S 4.9$]{TORS2}.  \\
c) Since $\gcd(3,25) = 1$, if $E_{/F}$ is not Olson then $E(F)[\tors] \cong \Z/2^{\epsilon} \ell^n \Z$ with $\epsilon \in \{0,1\}$, $3 < \ell \equiv 3 \pmod{4}$, $n \in \Z^+$, $K = \Q(\sqrt{-\ell})$ and $\frac{(\ell-1) h_{\Delta}}{2} \mid 25$.  Thus $\ell = 11$ and $\Delta = -11$.  We have seen that there is a unique $\OO(-11)$-CM elliptic curve defined over a degree $5$ number field
with a point of order $11$, so by Theorem \ref{IMPRIMITIVELEMMA} there is a CM elliptic curve
defined over a degree $25$ number field with a point
of order $11$.  (There are infinitely many up to
$F$-isomorphism, but they all induce the same closed point on the modular curve $X_1(11)$.)  We cannot have a point of order $11^2$, since then $11 \mid \frac{\varphi(11)^2}{2} \mid [F:\Q] = 25$.  By Theorem
\ref{4.2}c) no $\OO(-11)$-CM elliptic curve has a
point of order $2$ over a number field of degree prime
to $3$. \\
%so the only non-Olson torsion subgroup is $\Z/11\Z$. \\
d) Seeking a contradiction, we suppose that $p \geq 7$ and $E_{/F}$ is not Olson.  Then there is a prime $\ell \equiv 3 \pmod{4}$, $\ell \neq 3$, such that $E(F)$ has a point of order $\ell$ and $E$ has $\Q(\sqrt{-\ell})$-CM.  By Real Cyclotomy II,
\[ \frac{\ell-1}{2} \cdot h_{\Q(\sqrt{-\ell})} \mid  \frac{\ell-1}{2} \cdot h_{\Delta} \mid p^2. \]
Since $\ell > 3$, we must have $h_{\Q(\sqrt{-\ell})} = 1$ or
\begin{equation}
\label{PSQUAREDEQ}
\frac{\ell-1}{2} = p =  h_{\Q(\sqrt{-\ell})}.
  \end{equation}
If $h_{\Q(\sqrt{-\ell})} = 1$, then $\ell \in \{7,11,19,43,67,163\}$ and, since $p \geq 7$, $\frac{\ell-1}{2} \notin \{p,p^2\}$.  So (\ref{PSQUAREDEQ}) holds.  But using e.g. \cite[$\S 2$]{LP92} we have
 %h_{\Delta} = p$ and thus $h_{\Q(\sqrt{-\ell})} \mid p   On the other hand, one has CITE
\[ \frac{\ell-1}{2} = p = h_{\Q(\sqrt{-\ell})} \leq \sqrt{\ell} \log \ell. \]
Thus $\ell \leq 78$, and there are no solutions of
(\ref{PSQUAREDEQ}) in this range.
\end{proof}

%\begin{remark}
%A subsequent work of the first two authors with P. Pollack %contains further results and conjectures on
%the case of prime power degrees.
%\end{remark}

\subsection{Fields of Degree a Fixed Even Number Times a Variable Prime}
%Proof of Theorem \ref{SCHINZELTHM}}
\textbf{} \\ \\ \noindent
%We give the proof of Theorem \ref{SCHINZELTHM}.
Let us recall the statement of Schinzel's Hypothesis H \cite{SS58}, \cite[$\S 1.8.3$]{NABD}.

\begin{conj}(Schinzel's Hypothesis H) Let $f_1,\ldots,f_r \in \Q[t]$ be irreducible and integer-valued.  Suppose: for all $m \geq 2$ there is $n \in \Z^+$ such that $m \nmid f_1(n) \cdots f_r(n)$.  Then $\{ n \in \Z^+ \mid
|f_1(n)|,\ldots,|f_r(n)| \text{ are all primes} \}$ is infinite.
\end{conj}
\noindent

\begin{thm}
\label{5.2}
\label{TWICEPRIMETHM}
Assume Schinzel's Hypothesis H, and let $k \in \Z^+$.  Then
\[ \limsup_{p \in \mathcal{P}} \# \mathcal{T}^{\new}_{\CM}(2kp) \geq 1. \]
\end{thm}
\begin{proof}
Applying Schinzel's Hypothesis H with $f_1(x) = x$, $f_2(x) = 6kx + 1$, we
get infinitely many primes $p$ such that $N = 6kp + 1$ is prime.  Thus
$\frac{N-1}{3} = 2kp$.  In particular $N \equiv 1 \pmod{3}$, so $N$ splits in
$K = \Q(\sqrt{-3})$, and then there is an $\OO_K$-CM
elliptic curve $E$ defined over a number field $F$ of degree $\frac{N-1}{3} = 2kp$ with an $F$-rational point of order $N$ \cite[Theorem 3]{TORS1}.  We claim that for sufficiently large $N$, $E(F)[\tors] \in \mathcal{T}^{\new}_{\CM}(2kp)$: if so, the result follows.  There is $N_0 \in \Z^+$ such that for all primes $N \geq N_0$, if $E_{/F}$ is a CM elliptic curve over a number field $F$ with an $F$-rational point of order $N$, then $[F:\Q] \geq \frac{N-1}{3}$ \cite[Theorem 1]{TORS1}.  Thus for all primes $N \geq N_0$,  $E(F)[\tors]$ is a torsion subgroup that does not occur in any degree
\emph{smaller} than $[F:\Q]$, which certainly implies $E(F)[\tors] \in
\mathcal{T}^{\new}_{\CM}(2kp)$.
\end{proof}

\subsection{Fields of Degree a Product of Two Odd Primes}

\begin{thm}
\label{TWOODDPRIMESTHM1}
Let $p_1 \geq 7$ be a prime.  For all sufficiently large primes $p_2$, every CM elliptic curve defined over a number field of degree $p_1p_2$ is Olson.
\end{thm}
\begin{proof}
This is an immediate consequence of Theorems \ref{NEWABBEYTHM} and \ref{MAINTHM}.
\end{proof}

\begin{thm}
\label{TWOODDPRIMESTHM2}
The following are equivalent: \\
(i) The set $\mathcal{S} = \{k \in \Z^+ \mid \text{$4k+3$, $2k+1$ and $h_{\Q(\sqrt{-4k-3})}$ are all prime}\}$ is infinite. \\
(ii) There are infinitely many primes $\ell$ such that there are odd primes $p_1,p_2$, a number field $F$ of degree $p_1 p_2$ and a CM elliptic curve $E_{/F}$
with an $F$-rational point of order $\ell$.  \\
(iii) As $d$ ranges over products $p_1p_2$ of two
odd primes, $F$ ranges over number fields of degree
$p_1 p_2$ and $E_{/F}$ ranges over CM elliptic curves, $\# E(F)[\tors]$ is unbounded.
\end{thm}
\begin{proof}
(i) $\implies$ (ii): if $\ell = 4k+3$ is prime,
then by Corollary \ref{COR5.7} there is a CM elliptic curve defined over a
number field $F$ of degree $\frac{\ell-1}{2} h_{\Q(\sqrt{-\ell})} = (2k+1) h_{\Q(\sqrt{-4k-3})}$ with an $F$-rational point of degree $\ell$.
(ii) $\implies$ (iii) is immediate. \\
(iii) $\implies$ (i): Let $p_1$ and $p_2$ be odd primes, and let $F$ be a number field of degree $p_1 p_2$.  In view of previous results, we may assume that $p_1$ and $p_2$ are distinct and each at least $7$.  Then if $E_{/F}$ is not Olson, it has a point of prime order $\ell = 4k+3 > 3$, and then as usual we have \[\frac{\ell-1}{2} h_{\Q(\sqrt{-\ell})} = (2k+1) h_{\Q(\sqrt{-\ell})} \mid p_1 p_2 \]
and thus $(2k+1) = p_1$, $h_{\Q(\sqrt{-\ell})} = p_2$.  It remains to show that the unboundedness of $\#E(F)[\tors]$ forces there to be infinitely many such $k$: given that $F/\Q$ has odd degree and a point of prime order $\ell > 3$, its order must be of the form
$2^{\epsilon} \ell^n$ for $\epsilon \in \{0,1\}$ and $n \in \Z^+$.  Thus if the torsion were unbounded but $\mathcal{S}$ were finite, we would get points of order
$\ell^n$ for arbitrarily large $n$.  But by Real Cyclotomy II, $\ell^2 \mid \frac{\varphi(\ell^3)}{2} \mid [F:\Q] = p_1 p_2$, contradiction.
\end{proof}
\noindent
By no means are we in a position to show that the equivalent conditions of Theorem \ref{TWOODDPRIMESTHM2}
hold: it is not even known whether there are infinitely
Sophie Germain primes, and the assertion that $\mathcal{S}$ is infinite is clearly much stronger than that.  \indent
Nevertheless we believe that these conditions hold.  For $X \geq 1$, let
\[ \mathcal{S}(X) = \# \mathcal{S} \cap [1,X]. \]
Since we have three primality conditions on positive integers of size about $k$, it is reasonable to guess that as $X \ra \infty$, we have
\[ \mathcal{S}(X) \approx \frac{X}{\log^3 X}. \]
This is quite crude: $h_{\Q(\sqrt{-4k-3})}$ being \emph{odd} forces $4k+3$ to have no more than one odd prime divisor.  This made one of us think that perhaps
$\mathcal{S}(X) \approx \frac{X}{\log^2 X}$.  But P. Pollack has showed us a more refined heuristic -- incorporating the heuristics of
Cohen-Lenstra on class groups and Soundararajan on the number of imaginary quadratic fields with a given prime
class number -- which leads him to conjecture
\[ \mathcal{S}(X) \sim \frac{C_{\operatorname{PP}}}{X/\log^3(X)} \]
for an \emph{explicit constant} \[C_{\operatorname{PP}} \approx  7.96903. \]
Computations in gp/pari give
$\mathcal{S}(10^9) = 953,967$, so $\frac{\mathcal{S}(10^9)}{10^9/\log^3 10^9} \approx 8.49$,
providing some support for the heuristic.  Certainly it gives evidence that $\mathcal{S}$ is large!

\section{A Degree Non-Divisibility Theorem}
\noindent
We end with a result that exhibits a distinction between the CM and non-CM cases.  If $E_{/F}$ is a $K$-CM
elliptic curve with an $F$-rational point of prime order
$\ell > 2$, then by Theorem \ref{OLDNEWTHM2} we have \[\ell-1 \ \big{\vert} \ \frac{12(\ell-1) h(K)}{w(K)} \ \big{\vert} \ 12[F:\Q]. \]  In contrast, while the existence of a point of order $N$ on a non-CM elliptic curve $E$ defined over a number field $F$ imposes an lower bound on $[F:\Q]$ in terms of $N$ -- Merel's Theorem -- it imposes no divisibility condition on $[F:\Q]$ whatsoever.

\begin{thm}
\label{INDEXTHM}
Let $N \in \Z^+$, and let $p$ be a prime.    The set of algebraic numbers $j \in \overline{\Q}$ such that there is a number field $F$ with $p \nmid [F:\Q]$ and an elliptic curve $E_{/F}$ with $j(E) = j$ and a point of order $N$ in $E(F)$ is infinite.
\end{thm}
\begin{proof}
Let $\mathcal{S}(p,N)$ be the set of $j \in \overline{\Q}$ such that there is a number
field $K$ with $p \nmid [F:\Q]$ and an elliptic curve $E_{/F}$ with $j(E) = j$
and such that $E(F)$ has a point of order $N$.  We want to show that $\mathcal{S}(p,N)$
is infinite.  It suffices to show that for every finite subset $S \subset \mathcal{S}(p,N)$ -- including the empty set -- there is $j \in \mathcal{S}(p,N) \setminus S$.  \\ \indent
Let $\pi: X_1(N) \ra X(1)$ be the natural map.  Let $S \subset \mathcal{S}(p,N)$ be finite.  Identifying $Y(1)$ with $\mathbb{A}^1$, we view $S$ as a finite set of $\overline{\Q}$-valued points of $Y(1)$.  Let $Z_1$ be an effective $\Q$-rational divisor on $X(1)$ whose suppport\footnote{We use the standard bijection between closed points and $\mathfrak{g}_{\Q}$-orbits of $\overline{\Q}$-valued points.} contains $S$ and all the cusps, let $Z_N = \pi^*(Z_1)$, and let $U = X_1(N) \setminus \operatorname{supp}(Z_N)$.  By weak approximation, the least positive degree of a divisor on
$U$ is the least positive degree of a divisor on $X_1(N)$: see \cite[Lemma 12]{Clark07} for a complete treatment of a stronger result.  Since the cusp at $\infty$ is a $\Q$-rational point on $X_1(N)$, this common
quantity is $1$, and thus there is a divisor $\sum_i n_i [P_i]$ supported on $U$
such that $\sum_i n_i [\Q(P_i):\Q] = 1$.  For at least
one $i$ we must have $d \nmid [\Q(P_i): \Q]$.  The point $P_i$ corresponds to at least one pair $(E,x)_{/\Q(P_i)}$ where $E$ is an elliptic curve and
$x \in E(\Q(P_i))$ has order $N$ \cite[Proposition VI.3.2]{Deligne-Rapoport}.  By construction, we have
$j(E) \in \mathcal{S}(p,N) \setminus S$. \qedhere

% (Alternately, $X_1(N)$ is a fine moduli scheme
%when $N \geq 5$ and is isomorphic to $\PP^1$ for all $N \leq 4$, and there are
%well known parameterizations of $N$-torsion points in these cases which show
%in particular that for infinitely many $j \in \Q$ there is an elliptic curve
%$E(\Q)$ with $j(E) = j$ and a point of order $N$ on $E(\Q)$.)
\end{proof}

\section*{Appendix: Table of Degree Sequences}
\noindent
Let $m \mid n$ be positive integers; we exclude the pairs $(1,1)$, $(1,2)$,
$(1,3)$, $(2,2)$.  Then the modular curve $Y(m,n)$ classifying %(roughly: the precise description
%of the moduli problem involves a Cartier-equivariant %isomorphism and is omitted here)
$\left(\mu_m \times \Z/n\Z\right)$-structures on elliptic curves is a fine moduli space.  For every
$j \in \overline{\Q}$, the fiber of the morphism $Y(m,n) \ra Y(1)$
over $j$ is a finite $\Q(j)$-scheme.  The reduced subscheme of the fiber is
therefore isomorphic to a finite product $\prod_{i=1}^N K_i$ of number fields.  By the  \textbf{degree sequence} for $(m,n)$ and $j$ we mean
the sequence of degrees of the number fields $K_i(\zeta_m)$, written in non-decreasing order. These are the degrees of the (unique minimal) fields of definition $K$ such that there is an elliptic curve $E_{/K}$ with $\OO(\Delta)$-CM and an injection $\Z/m\Z \times \Z/n\Z \hookrightarrow E(K)$.
\\ \\
In the table below we list the degree sequences for the $13$ class number one imaginary quadratic discriminants for certain pairs $(m,n)$.  The results of this table are used in the proofs of Corollary \ref{odd2tors} and Theorem \ref{MAINTHM}.
\vspace{.1cm}
{\footnotesize
\begin{center}
    \begin{tabular}{ c|c|c|c|c|c|c|c}

     & -3 & -4 & -7& -8 & -11 & -12 & -16\\  \hline
$(1,4)$ & 2 & 1,2 & 2,2,2 & 2,4 & 6 & 2,4 & 1,1,4  \\ [.5 ex]
$(1,5)$ & 4 & 2,4 & 12 &12 & 4,8 &12 &4,8 \\ [.5 ex]

$(1,6)$ & 1,3  &2,4 & 4,8 &2,2,4,4 &6,6 & 1,2,3,6 & 4,8 \\ [.5 ex]

$(1,7)$ & 2,6 & 12 & 3,21 & 24 & 24 & 6,18 & 24 \\[.5 ex]

$(1,8)$ &  8 & 4,8 & 4,4,8,8 & 8,16 & 24 & 8,16 & 4,4,16 \\[.5 ex]

$(1,9)$ & 3,9 & 18 & 36 & 6,12,18 & 6,12,18 & 9,27 & 36 \\[.5 ex]

$(1,10)$ & 12 & 2,4,4,8 & 12,24 & 12,24 & 12,24 & 12,24 & 4,8,8,16 \\[.5 ex]

$(1,11)$ & 20 & 30 & 10,50 & 10,50& 5,55& 60 & 60 \\[.5 ex]

$(1,12)$ & 4,12 & 8,16 & 16,16,16 & 8,8,16,16 &24,24 &4,8,12,24 & 8,8,32 \\[.5 ex]

$(1,13)$ & 4,24 & 6,36 & 84 & 84 & 84 & 12,72 & 12,72 \\[.5 ex]

$(2,4)$ & 4 & 2,2,2 & 2,2,2,2,4 & 2,2,4,4 & 12 & 4,4,4 & 2,2,4,4 \\[.5 ex]

$(2,6)$ & 2,6 & 4,4,4 & 8,8,8 & 4,4,4,4,4,4 & 6,6,12 & 2,2,2,6,6,6 & 8,8,8 \\[.5 ex]

$(2,8)$ & 16 & 8,8,8 & 4,4,4,4,8,8,16 & 8,8,16,16 & 48 & 8,8,16,16 & 4,4,8,16,16 \\[.5 ex]

$(3,3)$ & 2,2,2 & 4,4 & 8,8 & 4,4,4,4 & 4,4,4,4 & 6,6,6 & 8,8\\[.5 ex]

\end{tabular}
\end{center}
\vspace{.1 cm}

 \begin{center}
    \begin{tabular}{ c|c|c|c|c|c|c}

   & -19 & -27 & -28 & -43 & -67 & -163\\  \hline
    $(1,4)$ & 6 & 6 &2,4 & 6 & 6 & 6\\ [.5 ex]
$(1,5)$& 4,8 & 12 & 12 & 12 & 12 & 12\\ [.5 ex]

$(1,6)$ & 12 & 3,9 & 4,8 & 12 & 12 & 12\\ [.5 ex]

$(1,7)$  & 6,18 & 6,18 & 3,21 & 24 & 24& 24\\[.5 ex]

$(1,8)$  & 24 & 24 & 4,4,16 & 24 & 24 & 24\\[.5 ex]

$(1,9)$ & 36 & 3,6,27 & 36 & 36 & 36 & 36\\[.5 ex]

$(1,10)$& 12,24 & 36 & 12,24 & 36 & 36 & 36\\[.5 ex]

$(1,11)$ & 10,50  & 60 & 10,50 & 10,50 & 60 & 60\\[.5 ex]

$(1,12)$ & 48 & 12,36 & 16,32 & 48 & 48 & 48\\[.5 ex]

$(1,13)$ & 84 & 12,72 & 84 & 12,72 & 84 & 84\\[.5 ex]

$(2,4)$  & 12 & 12 & 4,4,4 & 12 & 12 & 12\\[.5 ex]

$(2,6)$ & 24 & 6,18 & 8,8,8 & 24 & 24 & 24\\[.5 ex]

$(2,8)$ & 48 & 48 & 8,8,16,16 & 48 & 48 & 48\\[.5 ex]

$(3,3)$ & 8,8& 6,6,6 & 8,8 & 8,8 & 8,8 & 8,8\\[.5 ex]

\end{tabular}
\[ \text{\sc table 2} \]
\end{center}
}

%\bibliography{BCS14}
%\bibliographystyle{amsalpha}

\end{document}